\documentclass[a4paper,11pt]{amsart}

\usepackage{geometry}
\usepackage{amssymb,mathrsfs,amsmath,amsthm}
\usepackage{enumitem}
\usepackage{autobreak}
\allowdisplaybreaks
\usepackage{bm}
\usepackage{amsmath}
\usepackage{tabularx}
\usepackage{tabulary}

\usepackage{graphicx}
\usepackage{color}
\usepackage{booktabs}
\usepackage{hyperref}
\usepackage{cleveref}

\newtheorem{theorem}{Theorem}[section]
\newtheorem{lemma}{Lemma}[section]

\newtheorem{corollary}{Corollary}[section]
\newtheorem{definition}{Definition}[section]
\newtheorem{remark}{Remark}[section]

\numberwithin{equation}{section}
\numberwithin{figure}{section}

\begin{document}

\title{Sharp conditions for fully nonlinear equations with $(n-1)$-Hessian matrix: Part II. $k$-Hessian case}

\author{Jingwen Ji}
\address{School of Mathematics and Shing-Tung Yau Center, Southeast University, Nanjing 211189, P. R. China}
\email{jingwen\_ji@seu.edu.cn}

\author{Feida Jiang*}
\address{School of Mathematics and Shing-Tung Yau Center, Southeast University, Nanjing 211189, P. R. China; Shanghai Institute for Mathematics and Interdisciplinary Sciences, Shanghai 200433, P. R. China}
\email{jiangfeida@seu.edu.cn}
\thanks{Corresponding author: Feida Jiang.}

\author{Mengni Li}
\address{School of Mathematics, Southeast University, Nanjing 211189, P. R. China}
\email{lmn@seu.edu.cn}

\subjclass[2020]{Primary 35J60; Secondary 35B08, 35B44, 35J70, 35J75}
\keywords{$k$-Hessian type equation; Keller--Osserman condition; admissible entire subsolution; radial Cauchy problem; radial boundary blow-up solution.}

\begin{abstract}
We study admissible entire subsolutions and radial boundary blow-up solutions of $S_k^{\frac1k}[\Delta uI-D^2u]=f(u)+g(u)|Du|^q$ for $1\leq k<n$, continuing our study of the determinant case $k=n$ in Part I. Under structural assumptions on the continuous, non-decreasing coefficients $f$ and $g$, we establish necessary and sufficient conditions for the existence of entire subsolutions when $q<2$. Our results cover the non-Lipschitz case $0<q<1$ with a possibly vanishing zero-order term, the gradient-singular case $q<0$, and sign-changing gradient coefficients.
We also characterize finite-radius blow-up of radial solutions and derive two-sided estimates for the blow-up radius, as well as pointwise upper bounds for their profiles in terms of inverse Keller--Osserman functions.
\end{abstract}

\maketitle

\setcounter{tocdepth}{1}
\tableofcontents

\section{Introduction}\label{Section1}

We study admissible entire subsolutions and radial boundary blow-up solutions of the normalized Hessian type equation associated with the real $(n-1)$-Hessian matrix
\begin{equation}\label{eq1.1}
S_k^{\frac1k}[\Delta uI-D^2u]=f(u)+g(u)|Du|^q,\quad x\in\mathbb R^n,
\end{equation}
where $n\geq2$, $1\leq k<n$, and $q\in\mathbb R$. Here $Du$, $D^2u$, and $\Delta u$ denote the gradient, Hessian matrix, and Laplacian of $u$, respectively, and $I$ denotes the $n\times n$ identity matrix. The functions $f$ and $g$ are continuous and non-decreasing. For a symmetric matrix $M$ with eigenvalue vector $\lambda(M)=(\lambda_1,\ldots,\lambda_n)$, we define
$$
S_k[M]:=\sigma_k(\lambda(M))
=\sum_{1\leq i_1<\ldots<i_k\leq n}\lambda_{i_1}\ldots\lambda_{i_k}.
$$
The $k$th root in \eqref{eq1.1} makes the operator homogeneous of degree one in the Hessian.

To ensure the ellipticity of the operator, recall the G{\aa}rding cone
$$
\Gamma_k:=\left\{\lambda\in\mathbb R^n:\sigma_j(\lambda)>0,\quad j=1,\ldots,k\right\}.
$$
Let $\Omega\subseteq\mathbb R^n$. A function $u\in C^2(\Omega)$ is called \emph{admissible} in $\Omega$ if
$$
\lambda(\Delta uI-D^2u)\in\overline{\Gamma}_k,\quad x\in\Omega,
$$
and it is called \emph{strictly admissible} if $\lambda(\Delta uI-D^2u)\in\Gamma_k$ in $\Omega$. We denote the class of admissible functions in $\Omega$ by $\Phi^k(\Omega)$. Thus the operator $S_k^{\frac1k}[\Delta uI-D^2u]$ is degenerate elliptic with respect to admissible functions.

This is the second of two companion papers. The first paper \cite{JJL2026} treated the determinant case $k=n$ using a radial divergence identity.
The present paper concerns the remaining range $1\leq k<n$. For $2\leq k<n$, the radial equation couples $\varphi^{\prime\prime}$ and $\varphi^\prime/r$ implicitly and does not admit the divergence representation used in the determinant case.
We establish unified Keller--Osserman conditions for all $1\leq k<n$, whose form and integral exponents are independent of $k$. Beyond the necessary and sufficient conditions for the existence of entire subsolutions established in Part I, the present paper also provides quantitative estimates for radial boundary blow-up solutions, including two-sided bounds for the blow-up radius and pointwise upper bounds for the solution profiles.

We impose the same structural conditions (A) and (B) on $f$ and $g$ as in Part I, and restate them in Section \ref{SectionMainResults}. Condition (A) requires $f,g\geq0$ and $f(t)+g(t)>0$ for all $t>0$. Condition (B) requires $f\geq0$ with $f(t)>0$ for all $t>0$ and permits $g$ to change sign.
Condition (A) covers both gradient-degenerate and gradient-singular cases, and condition (B) is considered for $q>0$.
Despite these common coefficient assumptions, the distinct radial structures lead to different Keller--Osserman conditions and critical gradient exponents: $q=2$ for $1\leq k<n$ and $q=2-\frac1n$ for $k=n$. The local constructions and derivative estimates therefore require a different analysis from that used in the determinant case.

\subsection{Background and Motivation}

The standard $k$-Hessian operator is one of the fundamental fully nonlinear elliptic operators defined by the elementary symmetric functions of the eigenvalues of the Hessian. For $1\leq k\leq n$, the operator
$$
S_k[D^2u]:=\sigma_k\big(\lambda(D^2u)\big)
$$
interpolates between the Laplace operator $S_1[D^2u]=\Delta u$ and the Monge--Amp\`ere operator $S_n[D^2u]=\operatorname{det} D^2u$.
On the G{\aa}rding cone $\Gamma_k$, the operator $S_k$ is elliptic, and $S_k^{1/k}$ is concave. Its classical theory includes the Dirichlet results of Caffarelli, Nirenberg and Spruck \cite{CNS1985} and Trudinger \cite{Trudinger1995}, and the Hessian measure theory of Trudinger and Wang \cite{TW1997,TW1999}. In \eqref{eq1.1}, the elementary symmetric function is applied to the first Newton transformation
$$
T_1(A):=\sigma_1(A)I-A=(\operatorname{tr}A)I-A
$$
of a symmetric matrix $A$; see \cite{Reilly1973,GW2006}. Thus $T_1[u]:=T_1(D^2u)=\Delta uI-D^2u$ has eigenvalues
$$
\widetilde\lambda_i=\sum_{j\neq i}\lambda_j,\quad i=1,\ldots,n,
$$
where $\lambda_1,\ldots,\lambda_n$ are the eigenvalues of $D^2u$. We refer to $T_1[u]$ as the real $(n-1)$-Hessian matrix. Its complex counterpart appears in the equations studied in \cite{FWW2010,TW2017,STW2017}.

More generally, sums of $p$ Hessian eigenvalues occur in real $p$-Monge--Amp\`ere equations and are related to the geometric notion of $p$-convexity studied by Sha \cite{Sha1986,Sha1987}. Dinew \cite{Dinew2023} obtained interior gradient and Hessian estimates for $p$-plurisubharmonic solutions of
\begin{equation}\label{p-MA}
\prod_{1\leq i_1<\cdots<i_p\leq n}\left(\lambda_{i_1}+\cdots+\lambda_{i_p}\right)=f(x).
\end{equation}
Chu and Dinew \cite{ChuDinew2023} proved a Liouville theorem for a broader class containing \eqref{p-MA}, and Dong, Niu and Nizhamuding \cite{DNN2026} studied the corresponding Dirichlet problem on Riemannian manifolds. For $p=n-1$, these sums are the eigenvalues of $T_1[u]$. Jiao and Liu \cite{JL2021} obtained second-order estimates and existence results for Hessian type equations on Riemannian manifolds, including equations involving $\sigma_k(\Delta uI-D^2u)$. For the related quotient equations
\begin{equation*}
\frac{\sigma_k(\Delta uI-D^2u)}{\sigma_l(\Delta uI-D^2u)}=f(x,u,Du),
\end{equation*}
Chen, Tu and Xiang established Pogorelov type estimates \cite{CTX2021} and studied the Dirichlet problem \cite{CTX2023}. Further estimates for sum Hessian equations and their quotients appear in \cite{CDH2023,Mei2023,RW2025,RW2026}, and related Neumann problems in \cite{GLT2025,GT2026}. Geometric counterparts include the equations involving $\eta=Hg-h$ studied by Chu and Jiao \cite{CJ2022} and the problems for $p$-convex hypersurfaces considered by Dong \cite{DongCurvature2023,DongDirichlet2024}.

The growth conditions governing entire subsolutions originate in the work of Keller \cite{K1957} and Osserman \cite{O1957}. For positive, continuous and non-decreasing $f$, the equation $\Delta u=f(u)$ admits an entire subsolution if and only if
\begin{equation}\label{nsc,k=1}
\int^{+\infty}\left(\int_0^\tau f(t)dt\right)^{-\frac12}d\tau=+\infty.
\end{equation}
This is the classical Keller--Osserman condition.
For the standard normalized $k$-Hessian equation $S_k^{\frac1k}[D^2u]=f(u)$, Ji and Bao \cite{JB2010} established the corresponding condition
$$
\int^{+\infty}\left(\int_0^\tau f^k(t)dt\right)^{-\frac1{k+1}}d\tau=+\infty.
$$
Bao and Feng \cite{BF2022} extended this theory to $p$-$k$-Hessian equations.
Bao, Ji and Li \cite{BJL2012} studied $k$-Yamabe type equations involving the conformal Hessian, and established existence and nonexistence results for positive entire subsolutions under different growth assumptions.
Additive zero- and first-order terms were treated by Capuzzo Dolcetta, Leoni and Vitolo \cite{CDLV2014,C2016} for inequalities of the form
$$
\mathcal F(x,D^2u)\geq f(u)+g(u)|Du|^q\quad\text{in }\mathbb R^n.
$$
Their model operators include the degenerate maximal Pucci operator and the $k$-partial trace operator. They obtained conditions for entire viscosity subsolutions and local upper bounds when the corresponding integral conditions fail. Bhattacharya and Mohammed \cite{BM2021} studied the unnormalized standard $k$-Hessian equation
$$
S_k[D^2u]=f(u)+g(u)|Du|^q
$$
and used radial blow-up solutions to establish maximum principles on unbounded domains. For $0<q<k+1$, the generalized Keller--Osserman condition in their analysis is
\begin{equation*}
\int_a^{+\infty}\frac{dt}{\left(\int_0^t f(s)ds\right)^{\frac1{k+1}}+\left(\int_0^t g(s)ds\right)^{\frac1{k+1-q}}}<+\infty\quad\text{for some }a>0,
\end{equation*}
under their positivity and monotonicity assumptions. This is a convergence condition for radial blow-up. In \cite{JJL2023}, we established Keller–Osserman type conditions for
$$
S_k^{\frac1k}[D^2u+\mu|Du|I]=f(u)\quad\text{in }\mathbb R^n
$$
in different ranges of $\mu$. Jiang, Trudinger and Xu \cite{JTX2026} developed a Cauchy-problem framework for Keller--Osserman results concerning $k$-Hessian and $\Pi_k$-Hessian equations. For the weighted real $(n-1)$ Monge--Amp\`ere equation
$$
\operatorname{det}^{\frac1n}(\Delta uI-D^2u)=b(x)f(u)\quad\text{in }\mathbb R^n,
$$
our previous work \cite{JJL2025} established the generalized Keller--Osserman condition
\begin{equation}\label{nsc_n-1}
\int^{+\infty}\left(\int_0^\tau f(t)^{\frac n{n-1}}dt\right)^{-\frac{n-1}{2n-1}}d\tau=+\infty,
\end{equation}
and assumptions on $b$ for entire large and bounded solutions. The first part of this series \cite{JJL2026} extended the determinant problem to the additive gradient-dependent inhomogeneous term in \eqref{eq1.1}.

Boundary blow-up problems for Monge--Amp\`ere equations arise in the construction of complete K\"ahler--Einstein metrics and in affine geometry \cite{CY1,CY2}. For the standard real Monge--Amp\`ere equation, infinite boundary values were studied by Lazer and McKenna \cite{LM1996} and Matero \cite{M1996}. Guan and Jian \cite{GJ2004} treated $\operatorname{det} D^2u=\psi(x,u,Du)$ and obtained nearly optimal growth conditions for solvability. Related problems for standard $k$-Hessian equations were studied by Salani \cite{Salani1998} and Jian \cite{J2006}, and further results involving nonlinear gradient terms were obtained in \cite{Zhang2018,MP2021,ZF2022,ZF2023,Wan2026}. For the real $(n-1)$ Monge--Amp\`ere operator, the gradient-independent problem was considered in \cite{JDJ2025}. Ji and Jiang \cite{JJ2026} studied $\operatorname{det}(\Delta uI-D^2u)=\psi(x,u,Du)$ on bounded domains with positive mean curvature boundary, obtaining existence, uniqueness under additional assumptions, and boundary asymptotic estimates. Here we study entire subsolutions and radial boundary blow-up solutions of the additive equation \eqref{eq1.1}.

\subsection{Main Features}

The main results and methods are summarized as follows.
\begin{enumerate}[label=(\roman*)]
\item We establish necessary and sufficient conditions for the existence of admissible entire subsolutions under the structural assumptions on $f$ and $g$. When $q<2$, the Keller--Osserman integrals involve the powers $\frac12$ and $\frac1{2-q}$, independently of $k$. Their agreement with the corresponding semilinear exponents in \cite{C2016} reflects the comparison between $S_k^{\frac1k}[\Delta uI-D^2u]$ and the Laplacian on convex functions. The analysis includes gradient-degenerate and gradient-singular cases. Under condition (B), the conditions depend on the behavior of $g$ at infinity and involve exponential weights when $\lim\limits_{t\to+\infty}g(t)\leq0$. The exponent $q=2$ is structural critical for the form of the radial derivative estimates, and is not a universal threshold for global existence. Indeed, an eventually positive gradient coefficient implies nonexistence of positive admissible entire subsolutions for $1<q\leq2$; see Remark \ref{Rem4.1}.

\item We construct increasing convex radial solutions satisfying the original equation \eqref{eq1.1}. For $2\leq k<n$, we reformulate the implicit radial equation as a cooperative first-order system and apply the method of ordered sub- and supersolutions. This construction does not require local Lipschitz continuity of the coefficients. In the degenerate case $0<q<1$ and $f(a)=0$, it yields immediately increasing solutions without the assumption $f\equiv0$. In the singular case $q<0$, positive lower bounds for the derivative and estimates near the origin provide the required regularity. These local results form the basis for global continuation and comparison with entire subsolutions.

\item We establish quantitative estimates for radial solutions that blow up at a finite maximal radius. When the corresponding Keller--Osserman conditions fail, we derive two-sided bounds for the maximal radius of existence and pointwise upper bounds for the solution profiles expressed by the inverse Keller--Osserman functions. The estimates apply to solutions of the original equation, with the blow-up radius determined by the initial value and the selected radial solution.
\end{enumerate}

\subsection{Organization}

The main results and structural conditions (A) and (B) are presented in Section \ref{SectionMainResults}. In Section \ref{Section2}, we derive the radial equation and establish the admissibility, monotonicity, convexity, and regularity properties needed in the subsequent analysis.
In Section \ref{Section3}, we establish local existence under conditions (A) and (B) by the method of sub‑ and supersolutions.
In Section \ref{Section4}, we derive two-sided derivative estimates and characterize global radial existence.
Section \ref{Section5} extends the comparison argument of Part I to the $k$-Hessian operator, completing the proofs of Theorems \ref{Th1.1} and \ref{Th1.2}.
In Section \ref{Section6}, we study radial boundary blow-up solutions and establish radius and profile estimates.
Finally, Section \ref{Section7} contains further remarks on equivalent Keller–Osserman conditions, comparisons with Part I, and methodological implications.
The notation used in this paper, unless otherwise specified, follows \cite{JJL2026}.

\section{Main Results}\label{SectionMainResults}

Before stating the main theorems, we specify the notions of entire subsolutions used below.

\begin{definition}[Admissible entire subsolutions]\label{def:entire_admissible_subsolutions}
An admissible entire subsolution of \eqref{eq1.1} is understood in one of the following senses.
\begin{enumerate}[label=(\roman*)]
\item A function $u\in C^2(\mathbb R^n)$ is called an {\bfseries\boldmath admissible classical entire subsolution} of \eqref{eq1.1} if $u\in \Phi^k(\mathbb R^n)$ and
$$
S_k^{\frac1k}[\Delta uI-D^2u]\geq f(u)+g(u)|Du|^q,\quad x\in\mathbb R^n.
$$

\item Let $q<0$ and $\delta\geq1$. A function $u\in C^1(\mathbb R^n)\cap C^2(\mathbb R^n\setminus\{0\})\cap W_{loc}^{2,\delta}(\mathbb R^n)$ is called an {\bfseries\boldmath admissible strong entire subsolution} with an isolated singularity at the origin if $u\in\Phi^k(\mathbb R^n\setminus\{0\})$, $|Du|>0$ a.e. in $\mathbb R^n\setminus\{0\}$, and
$$
S_k^{\frac1k}[\Delta uI-D^2u]\geq f(u)+g(u)|Du|^q\quad\text{for a.e. }x\in\mathbb R^n\setminus\{0\}.
$$
\end{enumerate}
\end{definition}
By Definition \ref{def:entire_admissible_subsolutions}, an admissible strong entire subsolution belongs to $C^1(\mathbb R^n)\cap C^2(\mathbb R^n\setminus\{0\})$. Unless otherwise indicated, these requirements are left implicit, and only the Sobolev regularity and the corresponding exponent range are stated explicitly.

We present the main results in two parts. The first part gives sharp Keller--Osserman conditions for the existence of admissible entire subsolutions in $\mathbb R^n$. The second part characterizes positive radial admissible boundary blow-up solutions in balls and provides quantitative estimates for their finite maximal radii and profiles.

\subsection{Entire Subsolutions in the Whole Space}
We first impose the following nonnegative structural condition:
\begin{enumerate}
\item[\textbf{(A)}] {\bfseries\boldmath $f,g\in C(\mathbb R,[0,+\infty))$ are non-decreasing and satisfy $f(t)+g(t)>0$ for all $t>0$.}
\end{enumerate}

Condition (A)  implies that, for any $t>0$, at least one of $f(t)$ and $g(t)$ is strictly positive. However, this condition does not guarantee that the right-hand side of \eqref{eq1.1} is strictly positive at every point. If $f\equiv0$, then $g(t)>0$ for all $t>0$, and the right-hand side of \eqref{eq1.1} reduces to the pure gradient term $g(u)|Du|^q$. When $q>0$, this term vanishes at critical points of $u$. Consequently, every positive constant is a trivial solution of \eqref{eq1.1}.

Our first main result gives sharp Keller--Osserman conditions for solvability of entire subsolutions to equation \eqref{eq1.1} under condition (A).

\begin{theorem}[Entire subsolution under condition (A)]\label{Th1.1}
Assume that $f$ and $g$ satisfy condition (A) and $q\in\mathbb{R}$. Suppose further that either $f(t)>0$ for all $t>0$, or $q<1$.
\begin{enumerate}[label=(\roman*)]
\item Let $1\leq k<n$ and $0\leq q<2$. Equation \eqref{eq1.1} admits a positive nontrivial admissible classical entire subsolution if and only if
\begin{equation*}
\int^{+\infty}\frac{dt}{\left(\int_0^tf(s)ds\right)^{\frac12}+\left(\int_0^tg(s)ds\right)^{\frac1{2-q}}}=+\infty.\tag{$KO_k$}\label{KO_1}
\end{equation*}

\item Let $1\leq k<n$ and $q<0$. Equation \eqref{eq1.1} admits a positive nontrivial admissible strong entire subsolution $u\in W_{loc}^{2,\delta}(\mathbb R^n)$ for $\delta\in[1,1-\frac1q)$ if and only if \eqref{KO_1} holds.
\end{enumerate}
\end{theorem}

We note that the non-negativity assumption on $g$ in condition (A)  can be relaxed. More precisely, when $q>0$, we introduce the following alternative condition, which allows the function $g$ to be sign-changing:
\begin{enumerate}
\item[\textbf{(B)}] {\bfseries\boldmath $f\in C(\mathbb R,[0,+\infty))$ and $g\in C(\mathbb R,\mathbb R)$ are non-decreasing and satisfy $f(t)>0$ for all $t>0$.}
\end{enumerate}
For $0<q\leq2$, define
\begin{equation}\label{P}
P(t):=\frac{(g^-(t))^{\frac2q}}{(f(t))^{\frac2q-1}},\quad t>0.
\end{equation}
This function $P$ is well defined under condition (B).

We now present the second main result.

\begin{theorem}[Entire subsolution under condition (B)]\label{Th1.2}
Assume that $f$ and $g$ satisfy condition (B) and $q>0$.
\begin{enumerate}[label=(\roman*)]
\item Let $1\leq k<n$, $0<q<2$, and $\lim\limits_{t\to+\infty}g(t)>0$. Equation \eqref{eq1.1} admits a positive admissible classical entire subsolution if and only if
\begin{equation*}
\int^{+\infty}\frac{dt}{\left(\int_0^t f(s)ds\right)^{\frac12}+\left(\int_0^t g^+(s)ds\right)^{\frac1{2-q}}}=+\infty.
\tag{$KO_k^{(g^+)}$}\label{KO_2}
\end{equation*}

\item Let $1\leq k<n$, $0<q\leq2$, and $\lim\limits_{t\to+\infty}g(t)\leq0$. Equation \eqref{eq1.1} admits a positive admissible classical entire subsolution if and only if
\begin{equation*}
\int^{+\infty}\frac{dt}{\left(\int_0^t f(s)e^{-\int_s^tP(\tau)d\tau}ds\right)^{\frac12}}=+\infty.
\tag{$KO_k^{(g^-)}$}\label{KO_3}
\end{equation*}
\end{enumerate}
\end{theorem}

Under condition (B), when $\lim\limits_{t\to+\infty}g(t)\leq0$, the effect of the nonpositive gradient coefficient is reflected in the exponentially weighted condition \eqref{KO_3}. By \cite[Lemma 2.4~(i)]{JJL2026}, the convergence or divergence of the integral in \eqref{KO_3} is unchanged if the coefficient $1$ in the exponential weight is replaced by any positive constant.

Under the hypotheses of Theorems \ref{Th1.1} and \ref{Th1.2}, if $\lim\limits_{t\to+\infty}g(t)>0$, the integrals in \eqref{KO_1} and \eqref{KO_2} are finite for $1<q<2$. Consequently, Theorems \ref{Th1.1}~(i) and \ref{Th1.2}~(i) imply nonexistence of positive admissible entire subsolutions in this range; see Remark \ref{Rem4.1}. When $q=2$, the gradient-independent and eventually positive cases are summarized in the following corollary.

\begin{corollary}\label{Cor1.critical}
Let $q=2$. Assume that $f$ and $g$ satisfy condition (A) or (B), and that $f(t)>0$ for all $t>0$.
\begin{enumerate}[label=(\roman*)]
\item If $g\equiv0$ on $(0,+\infty)$, then \eqref{eq1.1} admits a positive admissible classical entire subsolution if and only if \eqref{nsc,k=1} holds.
\item If $\lim\limits_{t\to+\infty}g(t)>0$, then \eqref{eq1.1} admits no positive admissible classical entire subsolution.
\end{enumerate}
\end{corollary}

Assertion (i) follows from Theorem \ref{Th1.2}~(ii) with $P\equiv0$. Assertion (ii) follows by combining the finite-radius blow-up result in Theorem \ref{Th1.4}~(ii) with the equivalence established in Theorem \ref{Th6.1}~(i).

If the positivity restriction on the subsolution is to be removed, it is enough to strengthen the assumption to $f(t)>0$ for $t\in\mathbb R$. Under this modification, the same conditions are necessary and sufficient for the existence of admissible entire subsolutions without any sign restriction; see Remark \ref{Rem6.2}.

\subsection{Radial Boundary Blow-Up Solutions in Balls}
We next study boundary blow-up solutions in balls, which constitute the finite-radius counterpart of the entire subsolution problem in $\mathbb{R}^n$. In Theorems \ref{Th1.1} and \ref{Th1.2}, the divergence of the relevant Keller--Osserman integrals describes the existence of entire admissible subsolutions.
We now show that the convergence of the corresponding integrals leads to the existence of positive radial boundary blow-up solutions in balls of finite radius.

Let $\Omega\subset\mathbb R^n$ be a bounded domain and set $d_\Omega(x):=\operatorname{dist}(x,\partial\Omega)$. Consider
\begin{equation}\label{eq1.2_general}
\left\{
\begin{aligned}
&S_k^{\frac1k}[\Delta uI-D^2u]=f(u)+g(u)|Du|^q,\quad &&\text{in }\Omega,\\
&u(x)\to+\infty,\quad &&\text{as }d_\Omega(x)\to0.
\end{aligned}
\right.
\end{equation}
An admissible solution of \eqref{eq1.2_general} is called a boundary blow-up solution, or a large solution.
We restrict our attention to positive radial boundary blow-up solutions in balls.

For $z\in\mathbb R^n$ and $R>0$, let $B_R(z):=\{x\in\mathbb R^n:|x-z|<R\}$. Given an arbitrary center value $a>0$, we seek a finite radius $R>0$ and a positive radial admissible solution of the form $u_a(x)=\varphi_a(|x-z|)$ such that
$u_a(z)=a$ and $u_a(x)\to+\infty$ as $|x-z|\to R^-$.
Since the equation in \eqref{eq1.2_general} is invariant under translations, we may assume, without loss of generality, that $z=0$ and write $B_R:=B_R(0)$. For each $a>0$, we consider
\begin{equation}\label{eq1.2}
\left\{
\begin{aligned}
&S_k^{\frac1k}[\Delta u_aI-D^2u_a]=f(u_a)+g(u_a)|Du_a|^q,&&\text{in }B_R,\\
&u_a(x)\to+\infty,&&\text{as }|x|\to R^-,
\end{aligned}
\right.
\end{equation}
where $u_a(0)=a$.
Here, when $q\geq0$, a solution of \eqref{eq1.2} means a function $u_a\in C^2(B_R)$ that satisfies the differential equation pointwise in $B_R$, together with the prescribed center value and boundary blow-up condition. When $q<0$, a solution is understood as an admissible strong solution in $B_R$ that satisfies the differential equation pointwise in $B_R\setminus\{0\}$ and almost everywhere in $B_R$.

For a radial function $u_a(x)=\varphi_a(|x|)$, equation in \eqref{eq1.2} reduces to the Cauchy problem \eqref{CP} introduced in Section \ref{Section2}.
Throughout this paper, a positive radial admissible boundary blow-up solution is understood in the sense of Definition \ref{Def1.radial_solutions}.

To state the existence conditions for radial boundary blow-up solutions and the corresponding estimates for their blow-up radii, we introduce the following Keller--Osserman functions for each $a>0$.
\begin{definition}[Keller--Osserman functions]\label{Def1.boundary_KO}
Under condition (A), define
\begin{align*}
\Psi_k(a)
&:=\int_a^{+\infty}\frac{dt}{\left(\int_a^t f(s)ds\right)^{\frac12}+\left(\int_a^t g(s)ds\right)^{\frac1{2-q}}}, \quad \text{for }q<2, \\
\Psi_k^{(\lambda)}(a)
&:=\int_a^{+\infty}\frac{dt}{\left(\int_a^t f(s)e^{\lambda\int_s^t g(\tau)d\tau}ds\right)^{\frac12}}, \quad \text{for }q=2\text{ and }\lambda\geq0.
\end{align*}
Under condition (B), define
\begin{align*}
\widetilde\Psi_k(a)
&:=\int_a^{+\infty}\frac{dt}{\left(\int_a^t f(s)ds\right)^{\frac12}+\left(\int_a^t g^+(s)ds\right)^{\frac1{2-q}}},\quad \text{for }0<q<2,\\
\widetilde\Psi_k^{(\lambda)}(a)
&:=\int_a^{+\infty}\frac{dt}{\left(\int_a^t f(s)e^{\lambda\int_s^t g^+(\tau)d\tau}ds\right)^{\frac12}},\quad \text{for }q=2\text{ and }\lambda\geq0,\\
\widetilde\Psi_k^{(-\lambda)}(a)
&:=\int_a^{+\infty}\frac{dt}{\left(\int_a^t f(s)e^{-\lambda\int_s^t P(\tau)d\tau}ds\right)^{\frac12}},\quad \text{for }0<q\leq2\text{ and }\lambda>0.
\end{align*}
\end{definition}
Each function introduced in Definition \ref{Def1.boundary_KO} is non-increasing as a function of the center value $a$ and takes values in $(0,+\infty]$.
In particular, we have the following equivalences:
\begin{align*}
\eqref{KO_1}\text{ fails}&\iff\Psi_k(a)<+\infty,\\
\eqref{KO_2}\text{ fails}&\iff\widetilde\Psi_k(a)<+\infty,\\
\eqref{KO_3}\text{ fails}&\iff\widetilde\Psi_k^{(-\lambda)}(a)<+\infty\quad\text{for any }\lambda>0.
\end{align*}
The first two equivalences follow directly from the definitions, and the last one follows from \cite[Lemma 2.4~(i)]{JJL2026}.

The following two theorems establish sharp necessary and sufficient conditions for the existence of positive radial admissible boundary blow-up solutions of \eqref{eq1.2} and provide quantitative two-sided estimates for their blow-up radii.

\begin{theorem}[Radial boundary blow-up solution under condition (A)]\label{Th1.3}
Assume that $f$ and $g$ satisfy condition (A), and suppose that either $f(t)>0$ for all $t>0$, or $q<1$. Fix $a>0$.
\begin{enumerate}[label=(\roman*)]
\item Let $q<2$. Problem \eqref{eq1.2} admits a positive radial admissible boundary blow-up solution with center value $a$ in some ball $B_R$ if and only if
\begin{equation}\label{KO_4}
\Psi_k(a)<+\infty.
\end{equation}
In this case, there exist constants $C_1,C_2>0$, depending only on $n$, $k$ and $q$, such that the blow-up radius of every such solution satisfies
\begin{equation*}
\frac{\Psi_k(a)}{C_1}\leq R\leq\frac{\Psi_k(a)}{C_2}.
\end{equation*}

\item Let $q=2$. Problem \eqref{eq1.2} admits a positive radial admissible boundary blow-up solution with center value $a$ in some ball $B_R$ if and only if
\begin{equation}\label{KO_5}
\Psi_k^{(\lambda)}(a)<+\infty \quad\text{for any }\lambda>0
\end{equation}
If, in addition, $\Psi_k^{(0)}(a)<+\infty$, then the blow-up radius of every such solution satisfies
\begin{equation*}
\frac{\Psi_k^{(\lambda_1)}(a)}{\sqrt{\lambda_1}}\leq R\leq \frac{\Psi_k^{(0)}(a)}{\gamma_1},
\end{equation*}
where $\lambda_1:=2c_3^{-\frac1k}C_0$ and $\gamma_1:=2^{-(n+k-2)}\left(\frac{kC_0}{n-1}\right)^{\frac12}$.
\end{enumerate}
Furthermore, if $q\geq0$, then $u_a\in C^2(B_R)$. If $q<0$, then $u_a$ is an admissible strong solution satisfying $u_a\in W_{loc}^{2,\delta}(B_R)$ for $\delta\in[1,1-\frac1q)$.
\end{theorem}

Let $q<0$. If the center value $a$ satisfies $g(a)>0$, the radial entire subsolutions constructed in Theorem \ref{Th1.1}~(ii) and the radial boundary blow-up solutions constructed in Theorem \ref{Th1.3} belong to $W_{loc}^{2,\delta}(\mathbb R^n)$ and $W_{loc}^{2,\delta}(B_R)$, respectively, for $\delta\in[1,n(1-\frac1q))$. This improved regularity follows from their asymptotic behavior at the origin; see Remark \ref{Rem2.2}.

\begin{theorem}[Radial boundary blow-up solution under condition (B)]\label{Th1.4}
Assume that $f$ and $g$ satisfy condition (B). Fix $a>0$.
\begin{enumerate}[label=(\roman*)]
\item Let $0<q<2$ and $\lim\limits_{t\to+\infty}g(t)>0$.
Problem \eqref{eq1.2} admits a positive radial admissible boundary blow-up solution with center value $a$ in some ball $B_R$ if and only if
\begin{equation*}
\widetilde\Psi_k(a)<+\infty.
\end{equation*}
If, in addition, $g(a)\geq0$, there exist constants $C_1,C_2>0$, depending only on $n$, $k$ and $q$, such that the blow-up radius of every such solution satisfies
\begin{equation*}
\frac{\widetilde\Psi_k(a)}{C_1} \leq R\leq \frac{\widetilde\Psi_k(a)}{C_2}.
\end{equation*}

\item Let $q=2$ and $\lim\limits_{t\to+\infty}g(t)>0$. Problem \eqref{eq1.2} admits a positive radial admissible boundary blow-up solution with center value $a$ in some ball $B_R$.
If, in addition, $g(a)\geq 0$ and $\Psi_k^{(0)}(a)<+\infty$, then the blow-up radius of every such solution satisfies
\begin{equation*}
\frac{\widetilde\Psi_k^{(\lambda_1)}(a)}{\sqrt{\lambda_1}}\leq R \leq \frac{\Psi_k^{(0)}(a)}{\gamma_1},
\end{equation*}
where $\lambda_1$ and $\gamma_1$ are defined in Theorem \ref{Th1.3}~(ii).

\item Let $0<q\leq2$ and $\lim\limits_{t\to+\infty}g(t)\leq0$.
Problem \eqref{eq1.2} admits a positive radial admissible boundary blow-up solution with center value $a$ in some ball $B_R$ if and only if
\begin{equation}\label{KO_7}
\widetilde\Psi_k^{(-\lambda)}(a)<+\infty \quad \text{for any }\lambda>0.
\end{equation}
In this case, the blow-up radius of every such solution satisfies
\begin{equation*}
\frac{\widetilde\Psi_k^{(-\lambda_1)}(a)}{\sqrt{\lambda_1}} \leq R \leq \frac{\widetilde\Psi_k^{(-\lambda_q)}(a)}{\gamma_q},
\end{equation*}
where $\lambda_1$ is defined in Theorem \ref{Th1.3}~(ii), $\lambda_q:=\frac{qkC_0}{n-1}$ and $\gamma_q:=2^{-(n+k-2)}\left(\frac{qkC_0}{2(n-1)}\right)^{\frac12}$.
\end{enumerate}
In all three cases, $u_a\in C^2(B_R)$.
\end{theorem}

With the existence conditions and the two-sided estimates for the blow-up radii established, we next derive pointwise upper bounds for radial boundary blow-up solutions in terms of the distance to the boundary. For $x\in B_R$, set $d(x):=\operatorname{dist}(x,\partial B_R)=R-|x|$.

We introduce the following additional assumptions:
\begin{enumerate}
\item[\textbf{(C1-A)}] At least one of $f$ and $g$ is strictly increasing on $[0,+\infty)$, and $\lim\limits_{t\to+\infty}(f(t)+g(t))=+\infty$.

\item[\textbf{(C1-B)}] At least one of $f$ and $g^+$ is strictly increasing on $[0,+\infty)$, and $\lim\limits_{t\to+\infty}(f(t)+g^+(t))=+\infty$.

\item[\textbf{(C2)}] The function $f$ is strictly increasing on $[0,+\infty)$, and $\lim\limits_{t\to+\infty}f(t)=+\infty$.
\end{enumerate}

Suppose that $\Psi:(0,+\infty)\to(0,+\infty)$ is continuous and strictly decreasing, with
\begin{equation}\label{Psi_0}
\lim\limits_{a\to+\infty}\Psi(a)=0.
\end{equation}
Set $\Psi(0^+):=\lim\limits_{b\to0^+}\Psi(b)\in(0,+\infty]$. Define 
\begin{equation}\label{generalized_KO_inverse}
\Psi^{-1}(s):=\inf\left\{b>0:\Psi(b)\leq s\right\},\quad s>0.
\end{equation}
By \eqref{Psi_0}, the set in \eqref{generalized_KO_inverse} is nonempty for every $s>0$. Hence, $\Psi^{-1}:(0,+\infty)\to[0,+\infty)$ is well defined and continuous. The inverse $\Psi^{-1}$ is non-increasing on $(0,+\infty)$ and strictly decreasing on $(0,\Psi(0^+))$. If $\Psi(0^+)<+\infty$, then $\Psi^{-1}(s)=0$ for any $s\geq\Psi(0^+)$.
Moreover,  \eqref{Psi_0} gives
\begin{equation*}
\lim\limits_{s\to0^+}\Psi^{-1}(s)=+\infty.
\end{equation*}

\begin{theorem}[Blow-up profiles under condition (A)]\label{Th6.profileA}
Assume the hypotheses of Theorem \ref{Th1.3}, and let $u_a$ be a positive radial admissible boundary blow-up solution of \eqref{eq1.2} in $B_R$ with $u_a(0)=a>0$.
\begin{enumerate}[label=(\roman*)]
\item Let $q<2$, and suppose that condition (C1-A) holds. Then
$$
u_a(x)\leq\Psi_k^{-1}\bigl(C_2d(x)\bigr),\quad x\in B_R,
$$
where $C_2>0$ depends only on $n$, $k$, and $q$.

\item Let $q=2$, and suppose that condition (C2) holds and $\Psi_k^{(0)}(a)<+\infty$. Then
$$
u_a(x)\leq\left(\Psi_k^{(0)}\right)^{-1}\bigl(\gamma_1d(x)\bigr),\quad x\in B_R,
$$
where $\gamma_1:=2^{-(n+k-2)}\left(\frac{kC_0}{n-1}\right)^{\frac12}$.
\end{enumerate}
\end{theorem}

\begin{theorem}[Blow-up profiles under condition (B)]\label{Th6.profileB}
Assume the hypotheses of Theorem \ref{Th1.4}, and let $u_a$ be a positive radial admissible boundary blow-up solution of \eqref{eq1.2} in $B_R$ with $u_a(0)=a>0$.
\begin{enumerate}[label=(\roman*)]
\item Let $0<q<2$ and $\lim\limits_{t\to+\infty}g(t)>0$, and suppose that condition (C1-B) holds. Then
$$
u_a(x)\leq\max\left\{t_0,\widetilde\Psi_k^{-1}\bigl(C_2d(x)\bigr)\right\},\quad x\in B_R,
$$
where $t_0:=\inf\{t>0:g(t)\geq0\}$ and $C_2>0$ depends only on $n$, $k$, and $q$.

\item Let $q=2$ and $\lim\limits_{t\to+\infty}g(t)>0$, and suppose that condition (C2) holds and $\Psi_k^{(0)}(a)<+\infty$. Then
$$
u_a(x)\leq\max\left\{t_0,\left(\Psi_k^{(0)}\right)^{-1}\bigl(\gamma_1d(x)\bigr)\right\},\quad x\in B_R,
$$
where $t_0:=\inf\{t>0:g(t)\geq0\}$ and $\gamma_1:=2^{-(n+k-2)}\left(\frac{kC_0}{n-1}\right)^{\frac12}$.

\item Let $0<q\leq2$ and $\lim\limits_{t\to+\infty}g(t)\leq0$, and suppose that condition (C2) holds. Then
$$
u_a(x)\leq\left(\widetilde\Psi_k^{(-\lambda_q)}\right)^{-1}\bigl(\gamma_qd(x)\bigr),\quad x\in B_R,
$$
where $\lambda_q:=\frac{qkC_0}{n-1}$ and $\gamma_q:=2^{-(n+k-2)}\left(\frac{qkC_0}{2(n-1)}\right)^{\frac12}$.
\end{enumerate}
\end{theorem}

These results establish the alternative between global existence and finite-radius blow-up for the radial solutions constructed in this paper. Under the respective hypotheses, they also provide two-sided estimates for the blow-up radius and pointwise upper bounds for the solution in terms of inverse Keller--Osserman functions. The proofs are presented in Section \ref{Section6}.

The boundary blow-up analysis is restricted to radial solutions in balls. Existence on general bounded domains, uniqueness, and higher-order boundary expansions are beyond the scope of the present work. For related results concerning the determinant case $k=n$ with a general gradient-dependent inhomogeneous term on bounded domains, we refer to \cite{JJ2026}.


\section{Preliminaries on Radial Solutions}\label{Section2}

We derive the radial equation and establish the origin regularity, monotonicity, and convexity needed below. Throughout, $r=|x|$ for $x\in\mathbb R^n$.

We first recall the standard eigenvalue computation for radial functions; see Lemma 2.1 of \cite{JJL2025} for details.

\begin{lemma}\label{Lem2.1}
Suppose that $\varphi \in C^2[0,+\infty)$ and $\varphi^\prime(0)=0$. Then for $u(x)=\varphi(r)$, we have $u(x)\in C^2(\mathbb{R}^n)$ and the eigenvalues of $\Delta u I - D^2u$ are given by
\begin{equation}\label{2.1}
\lambda(\Delta uI-D^2 u)
=\left\{
\begin{aligned}
&\left( (n-1)\frac{\varphi^{\prime}}{r}, \varphi^{\prime\prime}+(n-2)\frac{\varphi^{\prime}}{r}, \ldots, \varphi^{\prime\prime}+(n-2)\frac{\varphi^{\prime}}{r}\right),&& r>0,\\
&\left( (n-1)\varphi^{\prime\prime}(0), (n-1)\varphi^{\prime\prime}(0), \ldots, (n-1)\varphi^{\prime\prime}(0)\right),&& r=0.
\end{aligned}
\right.
\end{equation}
Consequently,
\begin{equation}\label{2.2}
\begin{split}
&S_k[\Delta u I - D^2 u]\\
&=\left\{
\begin{aligned}
&C_{n-1}^{k-1} (n-1)\frac{\varphi^{\prime}}{r}  \left[\varphi^{\prime\prime}+ (n-2)\frac{\varphi^{\prime}}{r} \right]^{k-1} + C_{n-1}^k \left[\varphi^{\prime\prime}+ (n-2)\frac{\varphi^{\prime}}{r}\right]^k,&& r >0,\\
&C_n^k \left( (n-1)\varphi^{\prime\prime}(0)\right)^k,&& r=0,
\end{aligned}
\right.
\end{split}
\end{equation}
where $C_n^k=\frac{n!}{(n-k)!k!}$.
\end{lemma}

Taking the $k$th root of both sides of \eqref{2.2}, we find that the radial function $u(x)=\varphi(r)$ satisfies the following second-order ordinary differential equation:
\begin{equation}\label{sODE}
\left[\varphi^{\prime\prime}(r)+c_1\frac{\varphi^{\prime}(r)}r\right]^{\frac{k-1}{k}}
\left[c_3\varphi^{\prime\prime}(r)+c_2\frac{\varphi^{\prime}(r)}r\right]^{\frac1k}
=
C_0\left(f(\varphi(r))+g(\varphi(r))|\varphi^{\prime}(r)|^q\right),
\quad r>0,
\end{equation}
where
\begin{equation}\label{constants}
c_1=n-2,\quad c_2=n^2-2n+k,\quad c_3=n-k,\quad C_0=\left(\frac n{C_n^k}\right)^{\frac1k}.
\end{equation}
Since $1\leq k<n$, we have $c_2>c_1\geq0$, $c_3>0$, and $C_0>0$.
For $k\geq2$, equation \eqref{sODE} is fully nonlinear and implicit in $\varphi^{\prime\prime}$ and $\varphi^\prime/r$. When $k=1$, the first factor on the left-hand side is omitted, and the equation is linear in $\varphi^{\prime\prime}$.

For each $a>0$ and $R>0$, we consider the following nonlinear Cauchy problem associated with \eqref{sODE}:
\begin{equation}\label{CP}
\left\{
\begin{aligned}
&\left[\varphi^{\prime\prime}(r)+c_1\frac{\varphi^\prime(r)}{r}\right]^{\frac{k-1}{k}}\left[c_3\varphi^{\prime\prime}(r)+c_2\frac{\varphi^\prime(r)}{r}\right]^{\frac1k}=C_0\left[f(\varphi(r))+g(\varphi(r))|\varphi^\prime(r)|^q\right],\quad r\in(0,R],\\
&\varphi(0)=a,\quad \varphi^\prime(0)=0.
\end{aligned}
\right.
\end{equation}

The following lemma establishes the regularity of radial solutions at the origin. In the settings considered in \cite{JB2010,BF2022,JTX2026,JJL2026}, integral representations for $\varphi^\prime$ provide a direct approach to the analysis of origin regularity. For $2\leq k<n$, our analysis is based on the implicit radial equation \eqref{sODE}, rather than an explicit integral representation for $\varphi^\prime$.
We employ a scaling argument to determine the asymptotic behavior of $\varphi^{\prime\prime}(r)$ and $\varphi^\prime(r)/r$ as $r\to0^+$.

\begin{lemma}\label{Lem2.2}
Assume that $f,g\in C(\mathbb R)$ and $q\in\mathbb R$. Let $R>0$ and let $\varphi\in C^1[0,R]\cap C^2(0,R]$ be a solution of \eqref{CP}.
Suppose that there exists $R_0\in (0,R)$ such that
\begin{equation}\label{elliptic}
\varphi^\prime(r)>0,\quad\varphi^{\prime\prime}(r)+c_1\frac{\varphi^\prime(r)}r>0,\quad c_3\varphi^{\prime\prime}(r)+c_2\frac{\varphi^\prime(r)}r>0,\quad r\in(0,R_0].
\end{equation}
Then the following assertions hold.
\begin{enumerate}[label=(\roman*)]
\item Let $q>0$ and $f(a)>0$. Suppose that $\varphi^\prime(r)/r=O(1)$ as $r\to0^+$, then $\varphi\in C^2[0,R]$ and
$\varphi^{\prime\prime}(0)=\frac{f(a)}{(n-1)(C_n^k)^{\frac1k}}$.

\item Let $q=0$ and $f(a)+g(a)>0$. Suppose that $\varphi^\prime(r)/r=O(1)$ as $r\to0^+$, then $\varphi\in C^2[0,R]$ and
$\varphi^{\prime\prime}(0)=\frac{f(a)+g(a)}{(n-1)(C_n^k)^{\frac1k}}$.

\item Let $0<q<1$, $f(a)=0$. Suppose that $\varphi^{\prime\prime}(r)\geq0$ for $r\in(0,R_0]$, then $\varphi\in C^2[0,R]$ and
$\varphi^{\prime\prime}(0)=0$.

\item Let $q<0$. Suppose that $\varphi^{\prime\prime}(r)\geq0$ for $r\in(0,R_0]$, then $\varphi\in W_{loc}^{2,\delta}[0,R]$ for $\delta\in[1,1-\frac1q)$.
\end{enumerate}
\end{lemma}

\begin{proof}
We first prove (i)--(ii). Set
$$
h(r):=f(\varphi(r))+g(\varphi(r))(\varphi^\prime(r))^q.
$$
In each case, $\lim\limits_{r\to0^+}h(r)=h_0>0$, where $h_0=f(a)$ and $h_0=f(a)+g(a)$, respectively.

Let $w(r):=\frac{\varphi^\prime(r)}r$. Since $w(r)=O(1)$ and $w(r)>0$ near the origin, there exist $R_1\in(0,R_0]$ and $M>0$ such that
$$
0<w(r)\leq M,\quad h(r)>0,\quad r\in(0,R_1].
$$
Let $r=e^t$ and write $w(t):=w(e^t)$ for $t\in(-\infty,\log R_1]$. Using a dot to denote differentiation with respect to $t$, we have
$$
\dot w(t)=rw^\prime(r),\quad \varphi^{\prime\prime}(r)=w(t)+\dot w(t).
$$
Thus, \eqref{sODE} becomes
\begin{equation}\label{w_equation}
\left[(1+c_1)w(t)+\dot w(t)\right]^{\frac{k-1}{k}}\left[(c_3+c_2)w(t)+c_3\dot w(t)\right]^{\frac1k}=C_0h(e^t).
\end{equation}
For fixed $w\geq0$, define
$$
\Psi_w(y):=\left[(1+c_1)w+y\right]^{\frac{k-1}{k}}\left[(c_3+c_2)w+c_3y\right]^{\frac1k}
$$
on the interval where its factors are positive, with the first factor omitted when $k=1$.
The function $\Psi_w$ is strictly increasing in $y$. Hence, for every $w\geq0$ and $\eta>0$, the equation $\Psi_w(y)=\eta$ has a unique solution $y=G(w,\eta)$, and $G$ is continuous. Let $\eta(t):=C_0h(e^t)$ and $\eta_0:=C_0h_0$,
we may rewrite \eqref{w_equation} as
\begin{equation}\label{eq_G}
\dot w(t)=G(w(t),\eta(t)).
\end{equation}
Since $\eta(t)\to\eta_0$ as $t\to-\infty$, the function $G(\cdot,\eta(t))$ converges uniformly on $[0,M]$ to
$$
G_0(w):=G(w,\eta_0).
$$

Set
$K:=(1+c_1)^{\frac{k-1}{k}}(c_3+c_2)^{\frac1k}$ and $L:=\frac{\eta_0}{K}$.
Since $\Psi_w(0)=Kw$ and $\Psi_w$ is strictly increasing, we have
$$
G_0(w)>0\quad\text{for }0\leq w<L,\quad G_0(L)=0,\quad G_0(w)<0\quad\text{for }w>L.
$$
Thus, $L$ is the unique nonnegative equilibrium of the limiting equation $\dot w=G(w,\eta_0)$.

We claim that
\begin{equation}\label{L}
\lim\limits_{t\to-\infty}w(t)=L.
\end{equation}
First, $L\leq M$. Indeed, if $L>M$, then $G_0>0$ on $[0,M]$. The uniform convergence gives $T<\log R_1$ and $c>0$ such that $\dot w(t)\geq c$ for $t\leq T$. Hence,
$$
w(t)=w(T)-\int_t^T\dot w(s)ds\leq w(T)-c(T-t)<0
$$
for all sufficiently negative $t$, contradicting $w(t)>0$.
Fix $\varepsilon\in(0,L)$. By the uniform convergence of $G(\cdot,\eta(t))$ to $G_0$, there exist $T_\varepsilon<\log R_1$ and $c_\varepsilon>0$ such that
$$
G(w,\eta(t))\geq c_\varepsilon\quad\text{for }0\leq w\leq L-\varepsilon,\quad G(w,\eta(t))\leq-c_\varepsilon\quad\text{for }L+\varepsilon\leq w\leq M,
$$
whenever $t\leq T_\varepsilon$. Suppose that $w(t_0)\leq L-\varepsilon$ for some $t_0\leq T_\varepsilon$, the sign of the vector field at $L-\varepsilon$ implies that $w(t)\leq L-\varepsilon$ for $t\leq t_0$. Therefore,
$$
w(t)=w(t_0)-\int_t^{t_0}\dot w(s)ds\leq w(t_0)-c_\varepsilon(t_0-t),\quad t<t_0,
$$
which contradicts $w(t)>0$ for all sufficiently negative $t$. Similarly, if $w(t_0)\geq L+\varepsilon$, then $w(t)\geq L+\varepsilon$ for $t\leq t_0$ and
$$
w(t)=w(t_0)-\int_t^{t_0}\dot w(s)ds\geq w(t_0)+c_\varepsilon(t_0-t),\quad t<t_0,
$$
which contradicts $w(t)\leq M$. Consequently,
$$
L-\varepsilon<w(t)<L+\varepsilon
$$
for all sufficiently negative $t$. Since $\varepsilon>0$ is arbitrary, \eqref{L} follows.

It follows from \eqref{eq_G}, \eqref{L}, and the convergence of $\eta(t)$ that
$$
\lim\limits_{t\to-\infty}\dot w(t)=G(L,\eta_0)=0.
$$
Returning to $r=e^t$, we obtain
$$
\lim\limits_{r\to0^+}\frac{\varphi^\prime(r)}r=L,\quad \lim\limits_{r\to0^+}\varphi^{\prime\prime}(r)=L.
$$
Moreover, since $\varphi^\prime(0)=0$,
$$
\varphi^{\prime\prime}(0):=\lim\limits_{r\to0^+}\frac{\varphi^\prime(r)-\varphi^\prime(0)}r=L.
$$
Therefore, $\varphi^{\prime\prime}$ extends continuously to the origin and $\varphi\in C^2[0,R]$. Finally, \eqref{constants} gives
$$
K=(n-1)n^{\frac1k},\quad L=\frac{C_0h_0}{K}=\frac{h_0}{(n-1)(C_n^k)^{\frac1k}}.
$$
Substituting the two corresponding values of $h_0$ proves (i)--(ii).

We next prove (iii). Since $\varphi\in C^1[0,R]$, $\varphi^\prime(0)=0$, and $q>0$, we have $\lim\limits_{r\to0^+}\varphi(r)=a$ and $\lim\limits_{r\to0^+}\varphi^\prime(r)=0$.
Since $f(a)=0$ and $g$ is continuous, it follows that
$$
\lim\limits_{r\to0^+}h(r)=\lim\limits_{r\to0^+}\left[f(\varphi(r))+g(\varphi(r))(\varphi^\prime(r))^q\right]=0.
$$
Moreover, since $\varphi^{\prime\prime}\geq0$ and $\varphi^\prime>0$ on $(0,R_0]$, we have
$$
\left(\varphi^{\prime\prime}+c_1\frac{\varphi^\prime}{r}\right)^{\frac{k-1}{k}}
\left(c_3\varphi^{\prime\prime}+c_2\frac{\varphi^\prime}{r}\right)^{\frac1k}\geq c_3^{\frac1k}\varphi^{\prime\prime}\quad \text{on }(0,R_0].
$$
It follows from \eqref{sODE} that
$$
0\leq\varphi^{\prime\prime}(r)\leq C_0c_3^{-\frac1k}h(r)\to0,\quad \text{as }r\to0^+.
$$
Therefore,
$$
\lim\limits_{r\to0^+}\varphi^{\prime\prime}(r)=0.
$$
For any sufficiently small $r>0$, the mean value theorem gives some $\xi_r\in(0,r)$ such that
$$
\lim\limits_{r\to0^+}\frac{\varphi^\prime(r)-\varphi^\prime(0)}r
=\lim\limits_{r\to0^+}\varphi^{\prime\prime}(\xi_r)=0.
$$
Thus, $\varphi^\prime$ is differentiable at the origin and $\varphi^{\prime\prime}(0)=0$.
Combining this with $\lim\limits_{r\to0^+}\varphi^{\prime\prime}(r)=0$, we have $\varphi\in C^2[0,R]$, which proves (iii).

We next prove (iv). Since $\varphi^\prime(0)=0$ and $f$ and $g$ are continuous, there exist $R_*\in(0,R_0]$ and $M>0$ such that
$$
0<\varphi^\prime(r)\leq1,\quad |f(\varphi(r))|+|g(\varphi(r))|\leq M,\quad r\in(0,R_*].
$$
Since $q<0$, we have $(\varphi^\prime(r))^q\geq1$ on $(0,R_*]$. Using $\varphi^{\prime\prime}\geq0$ in \eqref{sODE}, we obtain
$$
0\leq\varphi^{\prime\prime}(r)\leq C_0c_3^{-\frac1k}\left[f(\varphi(r))+g(\varphi(r))(\varphi^\prime(r))^q\right]\leq C(\varphi^\prime(r))^q,\quad r\in(0,R_*],
$$
where $C:=C_0c_3^{-\frac1k}M$. Let $\delta\in[1,1-\frac1q)$. Then $1+q(\delta-1)>0$, and, for any $\varepsilon\in(0,R_*)$,
\begin{align*}
\int_\varepsilon^{R_*}|\varphi^{\prime\prime}(r)|^\delta dr
&\leq C^{\delta-1}\int_\varepsilon^{R_*}(\varphi^\prime(r))^{q(\delta-1)}\varphi^{\prime\prime}(r)dr\\
&\leq\frac{C^{\delta-1}}{1+q(\delta-1)}(\varphi^\prime(R_*))^{1+q(\delta-1)}.
\end{align*}
Letting $\varepsilon\to0^+$ and using $\varphi\in C^2(0,R]$ proves (iv).
\end{proof}

\begin{remark}\label{Rem2.2-iv}
The Sobolev regularity established in Lemma \ref{Lem2.2}~(iv) does not require any additional assumption on $g(a)$ or the existence of a finite limit of $g(\varphi(r))(\varphi^\prime(r))^q$ as $r\to0^+$. Under the hypotheses of that lemma, the following refinements hold.
\begin{enumerate}[label=(\roman*)]
\item Let $q<0$ and $g(a)=0$. Suppose that $\ell:=\lim\limits_{r\to0^+}g(\varphi(r))(\varphi^\prime(r))^q\in\mathbb R$, $f(a)+\ell>0$, and $\varphi^\prime(r)/r=O(1)$ as $r\to0^+$. Then $\varphi\in C^2[0,R]$ and $\varphi^{\prime\prime}(0)=\frac{f(a)+\ell}{(n-1)(C_n^k)^{\frac1k}}$.

\item Let $q<0$ and $g(a)>0$. Suppose that there exist $R_1\in(0,R_0]$ and $m,M>0$ such that $m\leq\varphi^\prime(r)/r^{\frac1{1-q}}\leq M$ for $r\in(0,R_1]$. Then
\begin{equation}\label{O}
\varphi^\prime(r)=Dr^{\frac1{1-q}}(1+o(1)),\quad \varphi^{\prime\prime}(r)=\frac{D}{1-q}r^{\frac q{1-q}}(1+o(1)),\quad\text{as }r\to0^+,
\end{equation}
where $D:=\left[C_0g(a)\left(\frac1{1-q}+c_1\right)^{-\frac{k-1}{k}}\left(\frac{c_3}{1-q}+c_2\right)^{-\frac1k}\right]^{\frac1{1-q}}$.
\end{enumerate}
Both assertions follow from scaling arguments analogous to those in the proof of Lemma \ref{Lem2.2}~(i)--(ii), applied to $\varphi^\prime(r)/r$ in (i) and $\varphi^\prime(r)/r^{\frac1{1-q}}$ in (ii).
\end{remark}
We next establish several properties of solutions to \eqref{sODE}. These properties will be used in the construction of entire subsolutions and in verifying the admissibility of the corresponding radial solutions.

\begin{lemma}\label{Lem2.3}
Assume that $f,g\in C(\mathbb R)$ and $q\in\mathbb R$. If $q>0$, assume that $f(t)>0$ for all $t>0$. If $q=0$, assume that $f(t)+g(t)>0$ for all $t>0$. Let $R\in(0,+\infty]$ and let $\varphi\in C^1[0,R)\cap C^2(0,R)$ solve \eqref{CP}. Suppose that there exists $r_0\in(0,R)$ such that
\begin{equation}\label{elliptic_2}
\varphi^\prime(r)>0,\quad \varphi^{\prime\prime}(r)\geq0,\quad r\in(0,r_0].
\end{equation}
Then the following assertions hold.
\begin{enumerate}[label=(\roman*)]
\item $\varphi^\prime(r)>0$ for all $r\in(0,R)$.

\item If $f$ and $g$ are non-decreasing, then $\varphi^{\prime\prime}(r)\geq0$ for all $r\in(0,R)$.

\item If $q>0$ and $f$ and $g$ are non-decreasing, then
\begin{equation}\label{iii}
\varphi^\prime(r)<\left(\frac{f(\varphi(r))}{g^-(\varphi(r))}\right)^{\frac1q},\quad r\in(0,R),
\end{equation}
where $g^-(t):=\max\{-g(t),0\}$ and the right-hand side is understood as $+\infty$ when $g^-(\varphi(r))=0$.
\end{enumerate}
\end{lemma}

\begin{proof}
We first prove (i). Suppose that $q\geq0$ and, for contradiction, that $\rho\in(r_0,R)$ is the first point at which $\varphi^\prime(\rho)=0$. Then $\varphi^\prime(r)>0$ for $r\in(0,\rho)$, $\varphi^{\prime\prime}(\rho)\leq0$, and $\varphi(\rho)>a>0$. By the assumptions on $f$ and $g$, we have $h(\rho)=f(\varphi(\rho))>0$ if $q>0$, and $h(\rho)=f(\varphi(\rho))+g(\varphi(\rho))>0$ if $q=0$. By continuity, $h(r)>0$ for $r<\rho$ sufficiently close to $\rho$, and \eqref{sODE} therefore gives $c_3\varphi^{\prime\prime}(r)+c_2\varphi^\prime(r)/r>0$ on this interval.
Letting $r\to\rho^-$ yields $\varphi^{\prime\prime}(\rho)\geq0$, and hence $\varphi^{\prime\prime}(\rho)=0$. The left-hand side of \eqref{sODE} then vanishes at $\rho$, contradicting $C_0h(\rho)>0$.

If $q<0$, the well-definedness of $|\varphi^\prime(r)|^q$ requires $\varphi^\prime(r)\neq0$ for every $r\in(0,R)$. Since $\varphi^\prime$ is continuous and positive near the origin, it remains positive throughout $(0,R)$. This proves (i).

We next prove (ii). Suppose, for contradiction, that $\varphi^{\prime\prime}$ is negative somewhere, and set
$$
r_1:=\inf\{r\in(r_0,R):\varphi^{\prime\prime}(r)<0\}.
$$
By \eqref{elliptic_2} and continuity, $r_1\in[r_0,R)$, $\varphi^{\prime\prime}(r_1)=0$, and $\varphi^{\prime\prime}(r)\geq0$ for $r\in(0,r_1]$.
By (i), $\varphi^\prime(r_1)>0$. Since $c_2>0$ and $c_1>0$ when $k\geq2$, the factors occurring on the left-hand side of \eqref{sODE} are strictly positive at $r_1$. By continuity and the definition of $r_1$, we may choose $\tau\in(r_1,R)$ sufficiently close to $r_1$ such that $\varphi^{\prime\prime}(\tau)<0$ and these factors remain strictly positive at $\tau$.

Since $\varphi^{\prime\prime}(\tau)<0$, we have $\varphi^\prime(r)>\varphi^\prime(\tau)$ for $r<\tau$ sufficiently close to $\tau$. Together with $\varphi^\prime(0)=0<\varphi^\prime(\tau)$, this shows that
$$
\sigma:=\min\{r\in(0,\tau):\varphi^\prime(r)=\varphi^\prime(\tau)\}
$$
is well defined and satisfies $\varphi^{\prime\prime}(\sigma)\geq0$. Thus $0<\sigma<\tau$, $\varphi^\prime(\sigma)=\varphi^\prime(\tau)>0$, and
\begin{equation}\label{2.14}
\varphi^{\prime\prime}(\sigma)\geq0>\varphi^{\prime\prime}(\tau),\quad \frac{\varphi^\prime(\sigma)}{\sigma}>\frac{\varphi^\prime(\tau)}{\tau}.
\end{equation}
For $k\geq2$, define
$$
\mathcal F(y,z):=(y+c_1z)^{\frac{k-1}{k}}(c_3y+c_2z)^{\frac1k}
$$
on the region where both factors are positive; for $k=1$, set $\mathcal F(y,z):=c_3y+c_2z$. By the choice of $\tau$ and the properties of $\sigma$, both points under consideration belong to the domain of $\mathcal F$. Since $c_1\geq0$, $c_2>0$, and $c_3>0$, \eqref{2.14} gives
\begin{equation}\label{eq2.1}
\mathcal F\left(\varphi^{\prime\prime}(\tau),\frac{\varphi^\prime(\tau)}{\tau}\right)<\mathcal F\left(\varphi^{\prime\prime}(\sigma),\frac{\varphi^\prime(\sigma)}{\sigma}\right).
\end{equation}
On the other hand, (i) gives $\varphi(\tau)>\varphi(\sigma)$. Since $f$ and $g$ are non-decreasing and $\varphi^\prime(\sigma)=\varphi^\prime(\tau)$,
$$
f(\varphi(\tau))+g(\varphi(\tau))(\varphi^\prime(\tau))^q\geq f(\varphi(\sigma))+g(\varphi(\sigma))(\varphi^\prime(\sigma))^q.
$$
In view of \eqref{sODE}, this contradicts \eqref{eq2.1}. Therefore, $\varphi^{\prime\prime}(r)\geq0$ for all $r\in(0,R)$.

Finally, we prove (iii). By (i) and (ii), $\varphi^\prime(r)>0$ and $\varphi^{\prime\prime}(r)\geq0$ for all $r\in(0,R)$. Since $c_2,c_3>0$ and $c_1>0$ when $k\geq2$, the left-hand side of \eqref{sODE} is strictly positive. Hence
$$
f(\varphi(r))+g(\varphi(r))(\varphi^\prime(r))^q>0,\quad r\in(0,R).
$$
If $g(\varphi(r))<0$, this yields
$$
g^-(\varphi(r))(\varphi^\prime(r))^q<f(\varphi(r)),
$$
which proves \eqref{iii}. If $g(\varphi(r))\geq0$, the right-hand side of \eqref{iii} is $+\infty$, so the inequality follows from the stated convention. This completes the proof.
\end{proof}
\begin{remark}
In the degenerate case $0<q<1$ and $f(a)=0$, conclusions (i) and (ii) remain valid if $f$ and $g$ are non-decreasing and \eqref{elliptic_2} holds. Indeed, the proof of (ii), applied on the maximal interval where $\varphi^\prime>0$, gives $\varphi^{\prime\prime}\geq0$ without using strict positivity of $f$. Consequently, $\varphi^\prime$ cannot vanish at a finite endpoint of this interval. Thus every locally increasing convex branch remains increasing and convex throughout its interval of existence. In particular, this applies to the degenerate branches constructed in Section \ref{Section3}, including the pure-gradient case $f\equiv0$.
\end{remark}

\begin{remark}\label{Rem2.2}
Let $R>0$ and let $\varphi\in C^1[0,R]\cap C^2(0,R]$ solve \eqref{CP} and satisfy the hypotheses of Lemma \ref{Lem2.3}. Set $u(x):=\varphi(|x|)$ in $B_R$. Applying that lemma on $(0,R]$ gives $\varphi^\prime>0$ and $\varphi^{\prime\prime}\geq0$. If $n\geq3$, all eigenvalues in \eqref{2.1} are positive for $r>0$. If $n=2$, then $k=1$ and
$\sigma_1(\lambda(\Delta uI-D^2u))=\frac{\varphi^\prime(r)}r+\varphi^{\prime\prime}(r)>0$ for $r>0$. Thus
\begin{equation*}
\lambda(\Delta uI-D^2u)\in\Gamma_k,\quad x\in B_R\setminus\{0\}.
\end{equation*}
If Lemma \ref{Lem2.2}~(i), (ii), or (iii) applies, then $u\in C^2(B_R)$ and is admissible at the origin as well.

For $q<0$, Lemma \ref{Lem2.2}~(iv) gives the one-dimensional Sobolev estimate. Since $\varphi^\prime(0)=0$, the identity $\varphi^\prime(r)=\int_0^r\varphi^{\prime\prime}(s)ds$ and Jensen's inequality give
$$
\int_0^R\left|\frac{\varphi^\prime(r)}r\right|^\delta r^{n-1}dr\leq\frac{R^{n-1}}{n-1}\int_0^R|\varphi^{\prime\prime}(r)|^\delta dr.
$$
The radial Hessian formula therefore yields
$$
u\in C^1(B_R)\cap C^2(B_R\setminus\{0\})\cap W_{loc}^{2,\delta}(B_R),\quad 1\leq\delta<1-\frac1q.
$$
If, in addition, the hypotheses of Remark \ref{Rem2.2-iv}~(ii) hold, then \eqref{O} gives $|D^2u(x)|\leq C|x|^{\frac q{1-q}}$ near the origin and improves the multidimensional range to $1\leq\delta<n(1-\frac1q)$.

The preceding admissibility conclusion does not require $g\geq0$. In particular, under condition (B), the right-hand side of \eqref{sODE} remains strictly positive along the increasing convex solutions.
\end{remark}

\begin{remark}\label{Rem_condition}
The local radial solutions constructed in Section \ref{Section3} satisfy $\varphi^\prime(r)>0$ and $\varphi^{\prime\prime}(r)\geq0$ for all sufficiently small $r>0$. Consequently, the conditions \eqref{elliptic_2} and \eqref{elliptic} are automatically satisfied and do not need to be imposed as extra assumptions on these solutions.
\end{remark}

The weighted integral comparisons in \cite[Lemma 2.4]{JJL2026} apply with $F=f$ and $\theta=\frac12$. Under condition (B), the function $P$ defined in \eqref{P} is nonnegative and non-increasing for $0<q\leq2$. Set
\begin{equation}\label{weighted_J_k}
J_{\lambda,b}^-(t):=\int_b^t f(s)e^{-\lambda\int_s^tP(\tau)d\tau}ds,\quad \lambda>0,\quad b\geq0,\quad t>b.
\end{equation}
Part (i) of \cite[Lemma 2.4]{JJL2026} yields
$$
J_{\lambda,b}^-(t)\leq J_{\mu,b}^-(t)\leq\frac{\lambda}{\mu}J_{\lambda,b}^-(t),\quad t>b,\quad 0<\mu<\lambda.
$$
Consequently, the convergence or divergence of
$$
\int_A^{+\infty}\frac{dt}{\sqrt{J_{\lambda,b}^-(t)}},\quad A>b,
$$
is independent of the choice of $\lambda>0$.
For positively weighted integrals, part (ii) of \cite[Lemma 2.4]{JJL2026} applies with $G=g$ under condition (A) and $G=g^+$ under condition (B), provided that $f(a)>0$. These comparisons establish the parameter independence of $\Psi_k^{(\lambda)}$ and $\widetilde\Psi_k^{(\lambda)}$ used below.

\section{Local Existence for the Radial Cauchy Problem}\label{Section3}

We establish local solvability of the Cauchy Problem \eqref{CP} under conditions (A) and (B). In the fully nonlinear case $k\geq2$, the radial equation in \eqref{CP} is nonlinear and implicit in $\varphi^{\prime\prime}$ and does not admit an explicit integral representation.
Moreover, the vector field of the associated first-order system need not be locally Lipschitz at $\varphi^\prime=0$ when $0<q<1$ and is singular there when $q<0$. We address these difficulties by constructing suitable sub- and supersolutions and applying an existence theorem for cooperative systems.

We first recall a local existence and regularity result for a class of Cauchy problems for first-order ordinary differential equations established by Euler's method.

\begin{lemma}[{\cite[Theorem 1.1]{JTX2026}}]\label{Lem3.1}
Assume that $g:\mathbb{R}\to(0,+\infty)$ is a continuous non-decreasing function. Let $q\geq0$, $\tau>0$, $\theta>0$, and $C>0$ be constants satisfying $\tau>q\theta$. For any initial value $ a \in \mathbb{R} $, there exists $ R > 0 $ such that the Cauchy problem
\begin{equation*}
\left\{
\begin{aligned}
\varphi^\prime(r)&=Cr^{-q}\left(\int_0^r s^{\tau-1}g(\varphi(s))ds\right)^{\frac{1}{\theta}}, \quad r\in(0,R),\\
\varphi(0)&=a,
\end{aligned}
\right.
\end{equation*}
admits a solution $\varphi\in C^1[0,R)\cap C^2(0,R)$ satisfying $\varphi^\prime(0)=0$, $\varphi^\prime(r)>0$, and $\varphi^{\prime\prime}(r)>0$ for all $r\in(0,R)$. Moreover, the following properties hold:
\begin{enumerate}
\item If $ \tau > \theta(q+1) $, then $ \varphi \in C^2[0, R) $ with $ \varphi^{\prime\prime}(0) = 0 $.
\item If $ \tau = \theta(q+1) $, then $ \varphi \in C^2[0, R) $ with $ \varphi^{\prime\prime}(0) = C \left( \frac{g(a)}{\tau} \right)^{\frac{1}{\theta}} $.
\item If $ \tau < \theta(q+1) $, then $ \varphi \in C^2(0, R) \cap W^{2,\delta}(0, R) $ for any $ \delta \in \left( 1, \frac{\theta}{\theta(q+1)-\tau} \right) $.
\end{enumerate}
\end{lemma}

We next recall the notion of a cooperative system; see, for example, \cite{Smith1995,Walter1998}.

\begin{definition}[Cooperative system]\label{Def_Cooperative}
A system of first-order ordinary differential equations
\begin{equation*}
\mathbf{u}^\prime(t)=\mathbf{F}(t,\mathbf{u}(t)),
\end{equation*}
where $\mathbf{F}=(F_1,F_2)^T:[0,T]\times\mathbb{R}^2\to\mathbb{R}^2$, is said to be cooperative (or quasimonotone non-decreasing), if $F_1(t,u_1,u_2)$ is non-decreasing in $u_2$ and $F_2(t,u_1,u_2)$ is non-decreasing in $u_1$.
\end{definition}

The following sub- and supersolution result for continuous cooperative systems follows by truncation and Schauder's fixed-point theorem. We include the argument to specify the continuity assumptions; see also \cite{Walter1998} and the scalar result \cite[Theorem 1.2.3]{GK1988}.

\begin{lemma}\label{Lem3.2}
Let $T>0$, and suppose that $\mathbf{F}:[0,T]\times\mathbb{R}^2\to\mathbb{R}^2$ is continuous and cooperative. Let $\mathbf u_0\in\mathbb R^2$. Let $\underline{\mathbf{u}}=(\underline{u}_1,\underline{u}_2)^T$ and $\bar{\mathbf{u}}=(\bar{u}_1,\bar{u}_2)^T$ belong to $C^1([0,T];\mathbb{R}^2)$ and be a subsolution and a supersolution, respectively, such that $\underline{\mathbf{u}}(t) \leq \bar{\mathbf{u}}(t)$ for all $t \in [0, T]$, and
\begin{align*}
\underline{\mathbf{u}}^\prime(t)\leq\mathbf{F}(t,\underline{\mathbf{u}}(t)), \quad \bar{\mathbf{u}}^\prime(t)\geq\mathbf{F}(t,\bar{\mathbf{u}}(t))
\end{align*}
with initial conditions $\underline{\mathbf{u}}(0) \leq \mathbf{u}_0 \leq \bar{\mathbf{u}}(0)$.
Then the initial value problem
\begin{equation}\label{osystem}
\left\{
\begin{aligned}
\mathbf{u}^\prime(t)&=\mathbf{F}(t,\mathbf{u}(t)), \quad t\in[0,T],\\
\mathbf{u}(0)&=\mathbf{u}_0,
\end{aligned}
\right.
\end{equation}
admits at least one solution $\mathbf{u}\in C^1([0,T];\mathbb{R}^2)$ satisfying
\begin{equation*}
\underline{\mathbf{u}}(t)\leq\mathbf{u}(t)\leq\overline{\mathbf{u}}(t), \quad t\in[0,T].
\end{equation*}
\end{lemma}

\begin{proof}
For $i=1,2$, define the continuous truncation mapping
\begin{equation*}
\mathbf{p}(t,\mathbf{u})=\bigl(p_1(t,u_1),p_2(t,u_2)\bigr)^T,
\end{equation*}
where $p_i(t,u_i):=\max\left\{\underline{u}_i(t),\min\left\{u_i,\bar{u}_i(t)\right\}\right\}$.
Consider the modified initial value problem
\begin{equation}\label{msystem}
\left\{
\begin{aligned}
\mathbf{u}^\prime(t)&=\mathbf{F}\bigl(t,\mathbf{p}(t,\mathbf{u}(t))\bigr)=:\widetilde{\mathbf{F}}(t,\mathbf{u}(t)), \quad t\in[0,T],\\
\mathbf{u}(0)&=\mathbf{u}_0.
\end{aligned}
\right.
\end{equation}
Set $D:=\left\{(t,\mathbf{v})\in[0,T]\times\mathbb{R}^2:\underline{\mathbf{u}}(t)\leq\mathbf{v}\leq\overline{\mathbf{u}}(t)\right\}$.
Since $\underline{\mathbf{u}}$ and $\overline{\mathbf{u}}$ are continuous, $D$ is compact. Moreover, $(t,\mathbf{p}(t,\mathbf{u}))\in D$ for any $(t,\mathbf{u})\in[0,T]\times\mathbb{R}^2$. The continuity of $\mathbf{F}$ implies that the modified vector field $\widetilde{\mathbf{F}}$ is continuous and globally bounded on $[0,T]\times\mathbb{R}^2$. Hence, there exists a constant $M>0$ such that
\begin{equation*}
\|\widetilde{\mathbf{F}}(t,\mathbf{u})\|_{\mathbb{R}^2}\leq M, \quad (t,\mathbf{u})\in[0,T]\times\mathbb{R}^2.
\end{equation*}
Let $X=C([0,T];\mathbb{R}^2)$ be equipped with the uniform norm $\|\mathbf{u}\|_\infty:=\max\limits_{t\in[0,T]}\|\mathbf{u}(t)\|_{\mathbb{R}^2}$.
The problem \eqref{msystem} is equivalent to the integral equation
$$
(\mathcal{T}\mathbf{u})(t):=\mathbf{u}_0+\int_0^t\widetilde{\mathbf{F}}(s,\mathbf{u}(s))ds.
$$
Define
$B:=\left\{\mathbf{u}\in X:\|\mathbf{u}-\mathbf{u}_0\|_\infty\leq MT\right\}$.
Clearly, $B$ is nonempty, closed, bounded, and convex. For any $\mathbf{u}\in B$ and $t\in[0,T]$, we have
$$
\|(\mathcal{T}\mathbf{u})(t)-\mathbf{u}_0\|_{\mathbb{R}^2}\leq\int_0^t\|\widetilde{\mathbf{F}}(s,\mathbf{u}(s))\|_{\mathbb{R}^2}ds\leq MT,
$$
and hence $\mathcal{T}(B)\subseteq B$. Furthermore, the uniform boundedness of $\widetilde{\mathbf{F}}$ implies that $\mathcal{T}(B)$ is uniformly bounded and equicontinuous. If $\mathbf{u}_j\to\mathbf{u}$ in $X$, then the uniform continuity of $\widetilde{\mathbf{F}}$ on a suitable compact set yields $\mathcal{T}\mathbf{u}_j\to\mathcal{T}\mathbf{u}$ in $X$. Thus, $\mathcal{T}$ is continuous. By the Arzel\`a--Ascoli theorem, $\mathcal{T}(B)$ is relatively compact in $X$, and hence $\mathcal{T}:B\to B$ is completely continuous. Consequently, Schauder's fixed-point theorem yields a fixed point $\mathbf{u}\in B$.
Since $\mathbf{u}$ is a fixed point of $\mathcal{T}$ and $t\mapsto\widetilde{\mathbf{F}}(t,\mathbf{u}(t))$ is continuous, the fundamental theorem of calculus implies that $\mathbf{u}\in C^1([0,T];\mathbb{R}^2)$ and solves \eqref{msystem}.

It remains to show that the solution of \eqref{msystem} satisfies
\begin{equation}\label{sub-super}
\underline{\mathbf{u}}(t)\leq\mathbf{u}(t)\leq\overline{\mathbf{u}}(t), \quad t\in[0,T].
\end{equation}
Suppose that the upper bound fails. Without loss of generality, there exist $t_0\in[0,T)$ and $\delta>0$ such that $w(t):=u_1(t)-\overline{u}_1(t)$ satisfies $w(t_0)=0$ and $w(t)>0$ for $t\in(t_0,t_0+\delta]$. For any $t\in(t_0,t_0+\delta]$, the definition of $\mathbf{p}$ gives $p_1(t,u_1(t))=\overline{u}_1(t)$ and $p_2(t,u_2(t))\leq\overline{u}_2(t)$. Hence, by the cooperativity of $\mathbf{F}$ and the supersolution property of $\overline{\mathbf{u}}$, we deduce that
$$
u_1^\prime(t)=F_1\left(t,\overline{u}_1(t),p_2(t,u_2(t))\right)\leq F_1\left(t,\overline{u}_1(t),\overline{u}_2(t)\right)\leq\overline{u}_1^\prime(t).
$$
Thus, $w^\prime(t)\leq0$ on $(t_0,t_0+\delta]$. Integrating from $t_0$ to $t$ gives $w(t) \le w(t_0) = 0$, contradicting $w(t) > 0$. Similar arguments apply to the lower bound and the second component. Therefore, \eqref{sub-super} holds, and $\mathbf{p}(t,\mathbf{u}(t))=\mathbf{u}(t)$ on $[0,T]$. Hence, $\mathbf{u}$ solves the original problem \eqref{osystem}.
\end{proof}

Based on the preceding sub- and supersolution method, we now establish local existence for problem \eqref{CP} under condition (A).

\begin{theorem}[Local Existence under condition (A)]\label{Th3.1}
Assume that $f$ and $g$ satisfy condition (A), and let $q\in\mathbb{R}$. Suppose further that either $f(t)>0$ for all $t>0$, or $q<1$. Then, for any initial value $a>0$, there exists $R>0$ such that \eqref{CP}
admits a positive solution $\varphi\in C^1[0,R]\cap C^2(0,R]$ satisfying $\varphi^\prime(r)>0$ and $\varphi^{\prime\prime}(r)\geq0$ for all $r\in(0,R]$. Moreover,
\begin{itemize}
\item[(i)] If $q\geq0$, then $\varphi\in C^2[0,R]$.
\item[(ii)] If $q<0$, then $\varphi\in C^2(0,R]\cap W^{2,\delta}[0,R]$ for $\delta\in[1,1-\frac{1}{q})$.
\end{itemize}
\end{theorem}

\begin{proof}
We construct subsolutions and supersolutions of \eqref{sODE} satisfying $\varphi(0)=a$ and $\varphi^\prime(0)=0$. Since $f(a)+g(a)>0$, we distinguish the following three cases.

\textbf{Case I (Regular case): $q=0$, or $q>0$ and $f(a)>0$.} Set
$$
\bar C_0:=C_0\left(\frac{2(1+c_1)}{c_3+c_2}\right)^{\frac1k},\quad \underline C_0:=C_0\left(\frac{1+c_1}{2(c_3+c_2)}\right)^{\frac1k}.
$$
Define
$$
\widehat g(t):=
\begin{cases}
g(t),&q=0,\\
(1-q)g(t)^{\frac1{1-q}},&0<q<1,\\
f(t),&q\geq1,
\end{cases}
\quad
\widetilde g(t):=f(t)+\widehat g(t).
$$
To construct a supersolution, by Lemma \ref{Lem3.1}, there exists $R_1^{\mathrm{super}}>0$ such that
\begin{equation}\label{eq2.16}
\left\{
\begin{aligned}
\overline\varphi^\prime(r)&=2\bar C_0r^{-c_1}\int_0^r s^{c_1}\widetilde g(\overline\varphi(s))ds,\quad r>0,\\
\overline\varphi(0)&=a
\end{aligned}
\right.
\end{equation}
admits a positive solution $\overline\varphi\in C^2[0,R_1^{\mathrm{super}}]$ satisfying $\overline\varphi^\prime(0)=0$, $\overline\varphi^\prime(r)>0$, and $\overline\varphi^{\prime\prime}(r)>0$ on $(0,R_1^{\mathrm{super}}]$. Differentiating \eqref{eq2.16}, we obtain
$$
\overline\varphi^{\prime\prime}+c_1\frac{\overline\varphi^\prime}{r}
=2\bar C_0\widetilde g(\overline\varphi).
$$
Since $\overline\varphi\in C^2[0,R_1^{\mathrm{super}}]$, $\overline\varphi^\prime(0)=0$, and $\overline\varphi^{\prime\prime}(0)>0$, we have $\lim\limits_{r\to0^+}\frac{\overline\varphi^\prime(r)}r=\overline\varphi^{\prime\prime}(0)$, and hence
$$
\lim\limits_{r\to0^+}
\frac{c_3\overline\varphi^{\prime\prime}+c_2\frac{\overline\varphi^\prime}{r}}{\overline\varphi^{\prime\prime}+c_1\frac{\overline\varphi^\prime}{r}}=\frac{c_3+c_2}{1+c_1}.
$$
There exists $\delta_0\in(0,R_1^{\mathrm{super}}]$ such that
\begin{align*}
\left(\overline\varphi^{\prime\prime}+c_1\frac{\overline\varphi^\prime}{r}\right)^{\frac{k-1}{k}}\left(c_3\overline\varphi^{\prime\prime}+c_2\frac{\overline\varphi^\prime}{r}\right)^{\frac1k}
\geq\left(\frac{c_3+c_2}{2(1+c_1)}\right)^{\frac1k}\left(\overline\varphi^{\prime\prime}+c_1\frac{\overline\varphi^\prime}{r}\right)
=2C_0\widetilde g(\overline\varphi),\quad r\in(0,\delta_0].
\end{align*}
We next estimate the right-hand side of \eqref{sODE}. If $q=0$, then $f(\overline\varphi)+g(\overline\varphi)=\widetilde g(\overline\varphi)$.
If $0<q<1$, Young's inequality gives
$$
g(\overline\varphi)(\overline\varphi^\prime)^q
\leq q\overline\varphi^\prime+(1-q)g(\overline\varphi)^{\frac1{1-q}}
=q\overline\varphi^\prime+\widehat g(\overline\varphi).
$$
Since $\overline\varphi^\prime(r)\to0$ and $\widetilde g(\overline\varphi(r))\to\widetilde g(a)>0$ as $r\to0^+$, we have $\lim\limits_{r\to0^+}\frac{q\overline\varphi^\prime(r)}{\widetilde g(\overline\varphi(r))}=0$. There exists $\delta_1\in(0,R_1^{\mathrm{super}}]$ such that $q\overline\varphi^\prime(r)\leq\widetilde g(\overline\varphi(r))$ on $(0,\delta_1]$, and
$$
f(\overline\varphi)+g(\overline\varphi)(\overline\varphi^\prime)^q
\leq\widetilde g(\overline\varphi)+q\overline\varphi^\prime
\leq2\widetilde g(\overline\varphi),
\quad r\in(0,\delta_1].
$$
If $q\geq1$, then $f(a)>0$ and $\lim\limits_{r\to0^+}\frac{g(\overline\varphi(r))(\overline\varphi^\prime(r))^q}{f(\overline\varphi(r))}=0$.
There exists $\delta_1\in(0,R_1^{\mathrm{super}}]$ such that $g(\overline\varphi)(\overline\varphi^\prime)^q\leq f(\overline\varphi)$ on $(0,\delta_1]$. Since $\widehat g=f$ in this case, it follows that
$$
f(\overline\varphi)+g(\overline\varphi)(\overline\varphi^\prime)^q
\leq2f(\overline\varphi)
=\widetilde g(\overline\varphi),
\quad r\in(0,\delta_1].
$$
For $q=0$, set $\delta_1=R_1^{\mathrm{super}}$, and let $\bar r:=\min\{\delta_0,\delta_1\}$. Combining these estimates, we obtain
\begin{align*}
\left(\overline\varphi^{\prime\prime}+c_1\frac{\overline\varphi^\prime}{r}\right)^{\frac{k-1}{k}}
\left(c_3\overline\varphi^{\prime\prime}+c_2\frac{\overline\varphi^\prime}{r}\right)^{\frac1k}
\geq2C_0\widetilde g(\overline\varphi)
\geq C_0\left[f(\overline\varphi)+g(\overline\varphi)(\overline\varphi^\prime)^q\right],
\quad r\in(0,\bar r].
\end{align*}
Therefore, $\overline\varphi$ is a supersolution of \eqref{CP} on $[0,\bar r]$.

To construct a subsolution, define
$$
\widetilde f(t):=
\begin{cases}
f(t)+g(t),&q=0,\\
f(t),&q\neq0.
\end{cases}
$$
By Lemma \ref{Lem3.1}, there exists $R_1^{\mathrm{sub}}>0$ such that
\begin{equation}\label{eq2.21}
\left\{
\begin{aligned}
\underline\varphi^\prime(r)&=\underline C_0r^{-c_1}\int_0^r s^{c_1}\widetilde f(\underline\varphi(s))ds,\quad r>0,\\
\underline\varphi(0)&=a
\end{aligned}
\right.
\end{equation}
admits a positive solution $\underline\varphi\in C^2[0,R_1^{\mathrm{sub}}]$ satisfying $\underline\varphi^\prime(0)=0$, $\underline\varphi^\prime(r)>0$, and $\underline\varphi^{\prime\prime}(r)>0$ on $(0,R_1^{\mathrm{sub}}]$. Differentiating \eqref{eq2.21}, we obtain
$$
\underline\varphi^{\prime\prime}+c_1\frac{\underline\varphi^\prime}{r}
=\underline C_0\widetilde f(\underline\varphi).
$$
Since $\lim\limits_{r\to0^+}\frac{\underline\varphi^\prime(r)}r=\underline\varphi^{\prime\prime}(0)>0$, we have
$$
\lim\limits_{r\to0^+}
\frac{c_3\underline\varphi^{\prime\prime}+c_2\frac{\underline\varphi^\prime}{r}}
{\underline\varphi^{\prime\prime}+c_1\frac{\underline\varphi^\prime}{r}}
=\frac{c_3+c_2}{1+c_1}.
$$
Therefore, there exists $\delta_2\in(0,R_1^{\mathrm{sub}}]$ such that
\begin{align*}
\left(\underline\varphi^{\prime\prime}+c_1\frac{\underline\varphi^\prime}{r}\right)^{\frac{k-1}{k}}
\left(c_3\underline\varphi^{\prime\prime}+c_2\frac{\underline\varphi^\prime}{r}\right)^{\frac1k}
&\leq\left(\frac{2(c_3+c_2)}{1+c_1}\right)^{\frac1k}
\left(\underline\varphi^{\prime\prime}+c_1\frac{\underline\varphi^\prime}{r}\right)\\
&=C_0\widetilde f(\underline\varphi)\\
&\leq C_0\left[f(\underline\varphi)+g(\underline\varphi)(\underline\varphi^\prime)^q\right],
\quad r\in(0,\delta_2].
\end{align*}
Thus, $\underline\varphi$ is a subsolution of \eqref{CP} on $[0,\delta_2]$. Moreover,
$$
\overline\varphi^{\prime\prime}(0)
=\frac{2\bar C_0\widetilde g(a)}{1+c_1},
\quad
\underline\varphi^{\prime\prime}(0)
=\frac{\underline C_0\widetilde f(a)}{1+c_1}.
$$
Since $\widetilde g(a)\geq\widetilde f(a)>0$ and $2\bar C_0>\underline C_0$, we have $\overline\varphi^{\prime\prime}(0)>\underline\varphi^{\prime\prime}(0)$. Choose $\delta_3>0$ sufficiently small that $\delta_3\leq\min\{R_1^{\mathrm{super}},R_1^{\mathrm{sub}}\}$ and $\underline\varphi^\prime(r)\leq\overline\varphi^\prime(r)$ on $[0,\delta_3]$. Since $\underline\varphi(0)=\overline\varphi(0)=a$, integration gives $\underline\varphi(r)\leq\overline\varphi(r)$ on $[0,\delta_3]$. Set
$$
R_1:=\min\{\bar r,\delta_2,\delta_3\}.
$$
Then $\underline\varphi$ and $\overline\varphi$ are the required subsolution and supersolution on $[0,R_1]$.

\textbf{Case II (Degenerate case): $0<q<1$ and $f(a)=0$.} Since $f(a)+g(a)>0$, we have $g(a)>0$. Set
$$
\bar C_0:=C_0\left(\frac{2(1+c_1)}{c_3+c_2}\right)^{\frac1k},\quad
\underline C_0:=C_0\left[\frac{\frac1{1-q}+c_1}{2\left(\frac{c_3}{1-q}+c_2\right)}\right]^{\frac1k}.
$$
Define
$$
\widehat g(t):=(1-q)g(t)^{\frac1{1-q}},\quad
\widetilde g(t):=f(t)+\widehat g(t).
$$
Since $\widetilde g(a)>0$, the supersolution argument in Case I applies. More precisely, let $\overline\varphi\in C^2[0,R_1^{\mathrm{super}}]$ be the corresponding solution of \eqref{eq2.16} given by Lemma \ref{Lem3.1}. Then there exists $\bar r_0\in(0,R_1^{\mathrm{super}}]$ such that
$$
\left(\overline\varphi^{\prime\prime}+c_1\frac{\overline\varphi^\prime}{r}\right)^{\frac{k-1}{k}}
\left(c_3\overline\varphi^{\prime\prime}+c_2\frac{\overline\varphi^\prime}{r}\right)^{\frac1k}
\geq C_0\left[f(\overline\varphi)+g(\overline\varphi)(\overline\varphi^\prime)^q\right],
\quad r\in(0,\bar r_0].
$$
Thus, $\overline\varphi$ is a supersolution of \eqref{CP} on $[0,\bar r_0]$.

To construct a subsolution, by Lemma \ref{Lem3.1}, there exists $R_2^{\mathrm{sub}}>0$ such that
\begin{equation}\label{eq2.21b}
\left\{
\begin{aligned}
\underline\varphi^\prime(r)&=\left[\underline C_0(1-q)r^{-c_1(1-q)}
\int_0^r s^{c_1(1-q)}g(\underline\varphi(s))ds\right]^{\frac1{1-q}},\quad r>0,\\
\underline\varphi(0)&=a
\end{aligned}
\right.
\end{equation}
admits a positive solution $\underline\varphi\in C^1[0,R_2^{\mathrm{sub}}]\cap C^2(0,R_2^{\mathrm{sub}}]$ satisfying $\underline\varphi^\prime(0)=0$, $\underline\varphi^\prime(r)>0$, and $\underline\varphi^{\prime\prime}(r)>0$ on $(0,R_2^{\mathrm{sub}}]$. Differentiating \eqref{eq2.21b}, we obtain
$$
\underline\varphi^{\prime\prime}+c_1\frac{\underline\varphi^\prime}{r}
=\underline C_0g(\underline\varphi)(\underline\varphi^\prime)^q.
$$
Moreover, \eqref{eq2.21b} gives
$$
\lim\limits_{r\to0^+}\frac{\underline\varphi^\prime(r)}{r^{\frac1{1-q}}}
=\left[\frac{\underline C_0(1-q)g(a)}{c_1(1-q)+1}\right]^{\frac1{1-q}}>0,
\quad
\lim\limits_{r\to0^+}
\frac{\underline\varphi^{\prime\prime}(r)}{\underline\varphi^\prime(r)/r}
=\frac1{1-q}.
$$
Then
$$
\lim\limits_{r\to0^+}\frac{c_3\underline\varphi^{\prime\prime}+c_2\frac{\underline\varphi^\prime}{r}}{\underline\varphi^{\prime\prime}+c_1\frac{\underline\varphi^\prime}{r}}
=\frac{\frac{c_3}{1-q}+c_2}{\frac1{1-q}+c_1}.
$$
Therefore, there exists $\delta_1\in(0,R_2^{\mathrm{sub}}]$ such that
\begin{align*}
\left(\underline\varphi^{\prime\prime}+c_1\frac{\underline\varphi^\prime}{r}\right)^{\frac{k-1}{k}}
\left(c_3\underline\varphi^{\prime\prime}+c_2\frac{\underline\varphi^\prime}{r}\right)^{\frac1k}
\leq C_0g(\underline\varphi)(\underline\varphi^\prime)^q
\leq C_0\left[f(\underline\varphi)+g(\underline\varphi)(\underline\varphi^\prime)^q\right],
\quad r\in(0,\delta_1].
\end{align*}
Thus, $\underline\varphi$ is a subsolution of \eqref{CP} on $[0,\delta_1]$.
Since $\overline\varphi^{\prime\prime}(0)=\frac{2\bar C_0\widetilde g(a)}{1+c_1}>0$ and $0<q<1$, we have
$$
\lim\limits_{r\to0^+}
\frac{\underline\varphi^\prime(r)}{\overline\varphi^\prime(r)}
=\lim\limits_{r\to0^+}\frac{\underline\varphi^\prime(r)}{r^{\frac1{1-q}}}\frac{r}{\overline\varphi^\prime(r)}r^{\frac{q}{1-q}}
=0.
$$
Hence, there exists $\delta_2\in(0,\min\{\bar r_0,R_2^{\mathrm{sub}}\}]$ such that $\underline\varphi^\prime(r)\leq\overline\varphi^\prime(r)$ on $[0,\delta_2]$. Since $\underline\varphi(0)=\overline\varphi(0)=a$, integration gives $\underline\varphi(r)\leq\overline\varphi(r)$ on $[0,\delta_2]$. Set
$$
R_2:=\min\{\bar r_0,\delta_1,\delta_2\}.
$$
Then $\underline\varphi$ and $\overline\varphi$ are the required subsolution and supersolution on $[0,R_2]$.

\textbf{Case III (Singular case): $q<0$.} To construct a supersolution, set
$$
\bar C_0:=C_0\left[\frac{2\left(\frac1{1-q}+c_1\right)}{\frac{c_3}{1-q}+c_2}\right]^{\frac1k}.
$$
Since $f(a)+g(a)>0$, by Lemma \ref{Lem3.1}, there exists $R_2^{\mathrm{super}}>0$ such that
\begin{equation}\label{eq2.17}
\left\{
\begin{aligned}
\overline\varphi^\prime(r)&=\left[2\bar C_0(1-q)r^{-c_1(1-q)}
\int_0^r s^{c_1(1-q)}\left(f(\overline\varphi(s))+g(\overline\varphi(s))\right)ds\right]^{\frac1{1-q}},\quad r>0,\\
\overline\varphi(0)&=a
\end{aligned}
\right.
\end{equation}
admits a positive solution $\overline\varphi\in C^1[0,R_2^{\mathrm{super}}]\cap C^2(0,R_2^{\mathrm{super}}]$ satisfying $\overline\varphi^\prime(0)=0$, $\overline\varphi^\prime(r)>0$, and $\overline\varphi^{\prime\prime}(r)>0$ on $(0,R_2^{\mathrm{super}}]$.
Since $\lim\limits_{r\to0^+}\overline\varphi^\prime(r)=0$, there exists $\delta_1\in(0,R_2^{\mathrm{super}}]$ such that $\overline\varphi^\prime(r)\leq1$ on $(0,\delta_1]$. Since $q<0$ and $f\geq0$, we have
$$
f(\overline\varphi)+g(\overline\varphi)(\overline\varphi^\prime)^q
\leq\left(f(\overline\varphi)+g(\overline\varphi)\right)(\overline\varphi^\prime)^q,
\quad r\in(0,\delta_1].
$$
Differentiating \eqref{eq2.17}, we obtain
$$
\overline\varphi^{\prime\prime}+c_1\frac{\overline\varphi^\prime}{r}
=2\bar C_0\left(f(\overline\varphi)+g(\overline\varphi)\right)(\overline\varphi^\prime)^q.
$$
Moreover, \eqref{eq2.17} gives
$$
\lim\limits_{r\to0^+}
\frac{\overline\varphi^\prime(r)}{r^{\frac1{1-q}}}
=\left[\frac{2\bar C_0(1-q)(f(a)+g(a))}{c_1(1-q)+1}\right]^{\frac1{1-q}}>0,
\quad
\lim\limits_{r\to0^+}
\frac{\overline\varphi^{\prime\prime}(r)}{\overline\varphi^\prime(r)/r}
=\frac1{1-q}.
$$
Then
$$
\lim\limits_{r\to0^+}\frac{c_3\overline\varphi^{\prime\prime}+c_2\frac{\overline\varphi^\prime}{r}}{\overline\varphi^{\prime\prime}+c_1\frac{\overline\varphi^\prime}{r}}
=\frac{\frac{c_3}{1-q}+c_2}{\frac1{1-q}+c_1}.
$$
Thus, there exists $\delta_2\in(0,R_2^{\mathrm{super}}]$ such that
\begin{align*}
\left(\overline\varphi^{\prime\prime}+c_1\frac{\overline\varphi^\prime}{r}\right)^{\frac{k-1}{k}}
\left(c_3\overline\varphi^{\prime\prime}+c_2\frac{\overline\varphi^\prime}{r}\right)^{\frac1k}
\geq2C_0\left(f(\overline\varphi)+g(\overline\varphi)\right)(\overline\varphi^\prime)^q
\geq C_0\left[f(\overline\varphi)+g(\overline\varphi)(\overline\varphi^\prime)^q\right]
\end{align*}
for $r\in(0,\min\{\delta_1,\delta_2\}]$.
Therefore, $\overline\varphi$ is a supersolution of \eqref{CP} on $[0,\min\{\delta_1,\delta_2\}]$.

We next construct a subsolution according to the value of $g(a)$.
If $g(a)>0$, set
$$
\underline C_0:=C_0\left[\frac{\frac1{1-q}+c_1}{2\left(\frac{c_3}{1-q}+c_2\right)}\right]^{\frac1k},
$$
and let $\underline\varphi$ be the solution of \eqref{eq2.21b} given by Lemma \ref{Lem3.1} on $[0,R_2^{\mathrm{sub}}]$. Thus, there exists $\delta_0\in(0,R_2^{\mathrm{sub}}]$ such that $\underline\varphi$ is a subsolution of \eqref{CP} on $[0,\delta_0]$. Moreover, \eqref{eq2.17} and \eqref{eq2.21b} give $\lim\limits_{r\to0^+}\frac{\underline\varphi^\prime(r)}{\overline\varphi^\prime(r)}=\left[\frac{\underline C_0g(a)}{2\bar C_0(f(a)+g(a))}\right]^{\frac1{1-q}}<1$.
Therefore, there exists $\delta_3\in(0,\min\{R_2^{\mathrm{sub}},R_2^{\mathrm{super}}\}]$ such that $\underline\varphi^\prime(r)\leq\overline\varphi^\prime(r)$ for $r\in[0,\delta_3]$.

If $g(a)=0$, then $f(a)>0$. Set
$$
\underline C_0:=C_0\left(\frac{1+c_1}{2(c_3+c_2)}\right)^{\frac1k},
$$
and let $\underline\varphi\in C^2[0,R_1^{\mathrm{sub}}]$ be the solution of \eqref{eq2.21} given by Lemma \ref{Lem3.1}. Here $\widetilde f=f$, since $q<0$. By the subsolution argument in Case I, there exists $\delta_0\in(0,R_1^{\mathrm{sub}}]$ such that
$$
\left(\underline\varphi^{\prime\prime}+c_1\frac{\underline\varphi^\prime}{r}\right)^{\frac{k-1}{k}}
\left(c_3\underline\varphi^{\prime\prime}+c_2\frac{\underline\varphi^\prime}{r}\right)^{\frac1k}
\leq C_0f(\underline\varphi)
\leq C_0\left[f(\underline\varphi)+g(\underline\varphi)(\underline\varphi^\prime)^q\right],
\quad r\in(0,\delta_0].
$$
Thus, $\underline\varphi$ is a subsolution of \eqref{CP} on $[0,\delta_0]$. Since $\lim\limits_{r\to0^+}\frac{\underline\varphi^\prime(r)}r
=\frac{\underline C_0f(a)}{1+c_1}>0$ and $\frac1{1-q}<1$, the limit for $\overline\varphi^\prime$ gives $\lim\limits_{r\to0^+}\frac{\underline\varphi^\prime(r)}{\overline\varphi^\prime(r)}=0$.
Therefore, there exists $\delta_3\in(0,\min\{R_1^{\mathrm{sub}},R_2^{\mathrm{super}}\}]$ such that $\underline\varphi^\prime(r)\leq\overline\varphi^\prime(r)$ on $[0,\delta_3]$.
In both cases, $\underline\varphi^\prime\leq\overline\varphi^\prime$ on $[0,\delta_3]$. Since $\underline\varphi(0)=\overline\varphi(0)=a$, integration gives $\underline\varphi\leq\overline\varphi$ on $[0,\delta_3]$. Set
$$
R_3:=\min\{\delta_0,\delta_1,\delta_2,\delta_3\}.
$$
Then $\underline\varphi$ and $\overline\varphi$ are the required subsolution and supersolution on $[0,R_3]$.

In each of the above cases, let $R$ denote $R_1$, $R_2$, or $R_3$, respectively. Let $\phi(r):=\varphi^\prime(r)$ and $\mathbf{u}(r):=(\varphi(r),\phi(r))^T$. For fixed $r>0$ and $\phi>0$, define
$$
H_{r,\phi}(z):=\left(z+c_1\frac{\phi}{r}\right)^{\frac{k-1}{k}}\left(c_3z+c_2\frac{\phi}{r}\right)^{\frac1k}
$$
on the set where $z+c_1\frac{\phi}{r}>0$ and $c_3z+c_2\frac{\phi}{r}>0$, with the first factor and its domain restriction omitted when $k=1$.
The function $H_{r,\phi}$ is continuous and strictly increasing in $z$ on this set. Hence, its inverse with respect to $z$ is well-defined and unique. We can rewrite \eqref{sODE} as
$$
\varphi^{\prime\prime}(r)
=H_{r,\phi}^{-1}\left(C_0\left[f(\varphi)+g(\varphi)\phi^q\right]\right)
=:F(r,\varphi,\phi).
$$
For fixed $r>0$ and $\phi>0$, since $f$ and $g$ are non-decreasing, the function $F(r,\varphi,\phi)$ is non-decreasing with respect to $\varphi$. Thus, \eqref{CP} is equivalent to the first-order cooperative system
\begin{equation}\label{1order}
\mathbf{u}^\prime(r)
=
\begin{pmatrix}
\phi(r)\\
F(r,\varphi(r),\phi(r))
\end{pmatrix}
=:\mathbf{F}(r,\varphi(r),\phi(r)),
\quad r\in(0,R],
\end{equation}
with the initial condition $\mathbf{u}(0)=(a,0)^T$.
Set
$$
\underline{\mathbf{u}}:=(\underline\varphi,\underline\varphi^\prime)^T,\quad \overline{\mathbf{u}}:=(\overline\varphi,\overline\varphi^\prime)^T.
$$
By the preceding constructions,
$$
\underline{\mathbf{u}}(0)=\overline{\mathbf{u}}(0)=(a,0)^T,
\quad
\underline{\mathbf{u}}(r)\leq\overline{\mathbf{u}}(r),\quad r\in[0,R].
$$
For $\rho\in(0,R)$, define
$$
D_\rho:=\left\{(r,v,w):r\in[\rho,R],\ \underline\varphi(r)\leq v\leq\overline\varphi(r),\ \underline\varphi^\prime(r)\leq w\leq\overline\varphi^\prime(r)\right\}.
$$
Since $r\geq\rho>0$ and $\phi\geq \underline\varphi^\prime>0$ on $[\rho,R]$, the vector field $\mathbf F$ is continuous and cooperative on $D_\rho$. Applying Lemma \ref{Lem3.2} on $[\rho,R]$ with
$\mathbf u_\rho(\rho)=\underline{\mathbf u}(\rho)$, we obtain a solution $\mathbf u_\rho=(\varphi_\rho,\phi_\rho)^T$ of \eqref{1order} satisfying
$$
\underline{\mathbf u}(r)\leq\mathbf u_\rho(r)\leq\overline{\mathbf u}(r),\quad r\in[\rho,R].
$$
Since $\underline{\mathbf{u}}$ is a subsolution and $\mathbf{u}_\rho(\rho)=\underline{\mathbf{u}}(\rho)$, we have
\begin{align*}
H_{\rho,\underline\varphi^\prime(\rho)}\bigl(\underline\varphi^{\prime\prime}(\rho)\bigr)
\leq C_0\left[f(\underline\varphi(\rho))+g(\underline\varphi(\rho))	\bigl(\underline\varphi^\prime(\rho)\bigr)^q\right]=
H_{\rho,\underline\varphi^\prime(\rho)}
\bigl(\varphi_\rho^{\prime\prime}(\rho)\bigr).
\end{align*}
The strict monotonicity of $H_{\rho,\underline\varphi^\prime(\rho)}$ gives
$\varphi_\rho^{\prime\prime}(\rho)\geq\underline\varphi^{\prime\prime}(\rho)>0$.
By continuity, $\varphi_\rho^{\prime\prime}>0$ in a right neighborhood of $\rho$. Moreover, $\varphi_\rho^\prime\geq\underline\varphi^\prime>0$ on $[\rho,R]$.
We therefore obtain
$$
\varphi_\rho^{\prime\prime}(r)\geq0,
\quad r\in[\rho,R].
$$

Let $\rho_j\downarrow0$ and fix $\varepsilon\in(0,R)$. For any $\rho_j<\varepsilon$, we have $(r,\mathbf{u}_{\rho_j}(r))\in D_\varepsilon$ for $r\in[\varepsilon,R]$. Since $D_\varepsilon$ is compact and $\mathbf{F}$ is bounded on $D_\varepsilon$, the families $\{\mathbf{u}_{\rho_j}\}_{\rho_j<\varepsilon}$ and $\{\mathbf{u}_{\rho_j}^\prime\}_{\rho_j<\varepsilon}$ are uniformly bounded on $[\varepsilon,R]$. By the Arzel\`a--Ascoli theorem and a diagonal argument, a subsequence converges locally uniformly on $(0,R]$ to a function $\mathbf{u}=(\varphi,\phi)^T$. Passing to the limit in the integral form of \eqref{1order}, we obtain $\mathbf{u}\in C^1((0,R];\mathbb{R}^2)$ satisfying \eqref{1order}. In particular, $\phi=\varphi^\prime$ on $(0,R]$. Moreover,
\begin{equation}\label{sub-super-comparision}
\underline\varphi(r)\leq\varphi(r)\leq\bar\varphi(r),\quad\underline\varphi^\prime(r)\leq\varphi^\prime(r)\leq\bar\varphi^\prime(r),\quad r\in(0,R].
\end{equation}
Since each $\varphi_{\rho_j}^\prime$ is non-decreasing, the limit $\varphi^\prime$ is also non-decreasing. Therefore,
$$
\varphi^{\prime\prime}(r)\geq0,\quad r\in(0,R].
$$

Finally, we verify the regularity at the origin. By \eqref{sub-super-comparision}, $\varphi(r)\to a$ and $\varphi^\prime(r)\to0$ as $r\to0^+$. Thus, setting $\varphi(0):=a$ and $\varphi^\prime(0):=0$, we obtain $\varphi\in C^1[0,R]\cap C^2(0,R]$. Moreover, $\varphi>0$ on $[0,R]$, and $\varphi^\prime>0$ and $\varphi^{\prime\prime}\geq0$ on $(0,R]$.
Suppose first that $q\geq0$. The supersolution constructed in Case I or Case II belongs to $C^2[0,R]$ and satisfies $\overline\varphi^\prime(0)=0$. Since $0<\varphi^\prime\leq\overline\varphi^\prime$, it follows that
\begin{equation*}
\frac{\varphi^\prime(r)}r=O(1),\quad\text{as }r\to0^+.
\end{equation*}
Lemma \ref{Lem2.2}~(i) applies when $q>0$ and $f(a)>0$, and Lemma \ref{Lem2.2}~(ii) applies when $q=0$. If $0<q<1$ and $f(a)=0$, the convexity established above allows us to apply Lemma \ref{Lem2.2}~(iii). Hence $\varphi\in C^2[0,R]$ in all three cases.
If $q<0$, the positivity of $\varphi^\prime$ and the convexity of $\varphi$, together with Lemma \ref{Lem2.2}~(iv), yield $\varphi\in W^{2,\delta}[0,R]$ for $\delta\in [1,1-\frac1q)$.
This completes the proof.
\end{proof}

Next, we establish local existence for problem \eqref{CP} under Condition (B).
\begin{theorem}[Local Existence under Condition (B)]\label{Th3.1-2}
Assume that $f$ and $g$ satisfy condition (B), and let $q>0$. Then, for any $a>0$, there exists $R>0$ such that problem \eqref{CP} admits a positive solution $\varphi\in C^2[0,R]$ satisfying $\varphi^\prime(r)>0$ and $\varphi^{\prime\prime}(r)\geq0$ for all $r\in(0,R]$.
\end{theorem}
\begin{proof}
Set
$$
\bar C_0:=C_0\left(\frac{2(1+c_1)}{c_3+c_2}\right)^{\frac1k},\quad \underline C_0:=C_0\left(\frac{1+c_1}{2(c_3+c_2)}\right)^{\frac1k},
$$
and let $g^+(t):=\max\{g(t),0\}$ and $g^-(t):=\max\{-g(t),0\}$.

We first construct a supersolution. Define
$$
\widehat g^+(t):=
\begin{cases}
(1-q)(g^+(t))^{\frac1{1-q}},&0<q<1,\\
f(t),&q\geq1,
\end{cases}
\quad \widetilde g(t):=f(t)+\widehat g^+(t).
$$
By Lemma \ref{Lem3.1}, there exists $R_1^{\mathrm{super}}>0$ such that
\begin{equation*}
\left\{
\begin{aligned}
\overline\varphi^\prime(r)&=2\bar C_0r^{-c_1}\int_0^r s^{c_1}\widetilde g(\overline\varphi(s))ds,\quad r>0,\\
\overline\varphi(0)&=a,
\end{aligned}
\right.
\end{equation*}
admits a positive solution $\overline\varphi\in C^2[0,R_1^{\mathrm{super}}]$ satisfying $\overline\varphi^\prime(0)=0$ and $\overline\varphi^\prime,\overline\varphi^{\prime\prime}>0$ on $(0,R_1^{\mathrm{super}}]$. Repeating the supersolution estimates in Case I of the proof of Theorem \ref{Th3.1}, with $g$ replaced by $g^+$, we obtain $R_0^{\mathrm{super}}\in(0,R_1^{\mathrm{super}}]$ such that
\begin{align*}
\left(\overline\varphi^{\prime\prime}+c_1\frac{\overline\varphi^\prime}{r}\right)^{\frac{k-1}{k}}\left(c_3\overline\varphi^{\prime\prime}+c_2\frac{\overline\varphi^\prime}{r}\right)^{\frac1k}
&\geq C_0\left[f(\overline\varphi)+g^+(\overline\varphi)(\overline\varphi^\prime)^q\right]\\
&\geq C_0\left[f(\overline\varphi)+g(\overline\varphi)(\overline\varphi^\prime)^q\right],\quad r\in(0,R_0^{\mathrm{super}}].
\end{align*}
Thus, $\overline\varphi$ is a supersolution of \eqref{CP} on $[0,R_0^{\mathrm{super}}]$.

We next construct a subsolution. By Lemma \ref{Lem3.1}, there exists $R_1^{\mathrm{sub}}>0$ such that
\begin{equation*}
\left\{
\begin{aligned}
\underline\varphi^\prime(r)&=\frac12\underline C_0r^{-c_1}\int_0^r s^{c_1}f(\underline\varphi(s))ds,\quad r>0,\\
\underline\varphi(0)&=a,
\end{aligned}
\right.
\end{equation*}
admits a positive solution $\underline\varphi\in C^2[0,R_1^{\mathrm{sub}}]$ satisfying $\underline\varphi^\prime(0)=0$ and $\underline\varphi^\prime,\underline\varphi^{\prime\prime}>0$ on $(0,R_1^{\mathrm{sub}}]$.
As in the subsolution construction in Case I of the proof of Theorem \ref{Th3.1}, there exists $R_0^{\mathrm{sub}}\in(0,R_1^{\mathrm{sub}}]$ such that
$$
\left(\underline\varphi^{\prime\prime}+c_1\frac{\underline\varphi^\prime}{r}\right)^{\frac{k-1}{k}}\left(c_3\underline\varphi^{\prime\prime}+c_2\frac{\underline\varphi^\prime}{r}\right)^{\frac1k}
\leq\frac{C_0}{2}f(\underline\varphi),\quad r\in(0,R_0^{\mathrm{sub}}].
$$
Moreover, since $q>0$ and $\lim\limits_{r\to0^+}\frac{g^-(\underline\varphi(r))(\underline\varphi^\prime(r))^q}{f(\underline\varphi(r))}=0$, there exists $R_2^{\mathrm{sub}}\in(0,R_1^{\mathrm{sub}}]$ such that
$$
g^-(\underline\varphi)(\underline\varphi^\prime)^q
\leq\frac12f(\underline\varphi),\quad r\in(0,R_2^{\mathrm{sub}}].
$$
Since $g\geq-g^-$, it follows that
$$
f(\underline\varphi)+g(\underline\varphi)(\underline\varphi^\prime)^q
\geq\frac12f(\underline\varphi),\quad r\in(0,R_2^{\mathrm{sub}}].
$$
Therefore, $\underline\varphi$ is a subsolution of \eqref{CP} on
$\left[0,\min\left\{R_0^{\mathrm{sub}},R_2^{\mathrm{sub}}\right\}\right]$.
At the origin,
$$
\overline\varphi^{\prime\prime}(0)=\frac{2\bar C_0\widetilde g(a)}{1+c_1}>\frac{\underline C_0f(a)}{2(1+c_1)}=\underline\varphi^{\prime\prime}(0).
$$
Thus, there exists $R_0\in\left(0,\min\left\{R_0^{\mathrm{super}},R_0^{\mathrm{sub}},R_2^{\mathrm{sub}}\right\}\right]$ such that $\underline\varphi^\prime\leq\overline\varphi^\prime$ and $\underline\varphi\leq\overline\varphi$ on $[0,R_0]$.

For $r\in(0,R]$ and $\underline\varphi(r)\leq v\leq\overline\varphi(r)$, $\underline\varphi^\prime(r)\leq w\leq\overline\varphi^\prime(r)$,
we have $v\geq a>0$. Since $\lim\limits_{r\to0^+}\overline\varphi^\prime(r)=0$ and $q>0$, there exists $R\in(0,R_0]$ such that
$$
g^-(a)(\overline\varphi^\prime(r))^q\leq\frac12f(a),\quad r\in(0,R].
$$
Since $f$ and $g$ are non-decreasing,
\begin{align*}
f(v)+g(v)w^q\geq f(a)-g^-(a)w^q \geq f(a)-g^-(a)(\overline\varphi^\prime(r))^q \geq\frac12f(a)>0.
\end{align*}
Hence, $\mathbf F$ defined in \eqref{1order} is continuous and cooperative in $D_\rho$.
The construction, convexity, and compactness arguments in the proof of Theorem \ref{Th3.1} then yield a positive solution satisfying $\varphi^\prime>0$ and $\varphi^{\prime\prime}\geq0$ on $(0,R]$, with
$$
\frac{\varphi^\prime(r)}r
\leq
\frac{\overline\varphi^\prime(r)}r
=O(1)
\quad\text{as }r\to0^+.
$$
Since $f(a)>0$ in condition (B), Lemma \ref{Lem2.2}~(i) yields $\varphi\in C^2[0,R]$.
This completes the proof.
\end{proof}

\begin{remark}\label{Rem3.1}
Let $\varphi$ be a local solution obtained in either Theorem \ref{Th3.1} or Theorem \ref{Th3.1-2}. The continuation theorem for continuous first-order systems gives a maximal extension on $[0,\mathcal R)$ for some $0<\mathcal R\leq+\infty$. Lemma \ref{Lem2.3} gives $\varphi^\prime>0$ and $\varphi^{\prime\prime}\geq0$ throughout this interval.

Suppose that $\mathcal R<+\infty$ and both $\varphi$ and $\varphi^\prime$ remain bounded as $r\to\mathcal R^-$. Their monotonicity gives finite limits $L_1:=\lim\limits_{r\to\mathcal R^-}\varphi(r)\geq a$ and $L_2:=\lim\limits_{r\to\mathcal R^-}\varphi^\prime(r)>0$. By convexity and \eqref{sODE},
$$
C_0\left[f(L_1)+g(L_1)L_2^q\right]\geq\left(c_1\frac{L_2}{\mathcal R}\right)^{\frac{k-1}{k}}\left(c_2\frac{L_2}{\mathcal R}\right)^{\frac1k}>0,
$$
where the first factor is omitted when $k=1$. Thus the inverse defining the vector field in \eqref{1order} remains continuous near $(\mathcal R,L_1,L_2)$ under either condition (A) or condition (B). Peano's theorem therefore gives a continuation beyond $\mathcal R$, contradicting maximality. Consequently, if $\mathcal R<+\infty$, then
\begin{equation}\label{blow-up-alternative}
\lim\limits_{r\to\mathcal R^-}\varphi(r)=+\infty\quad\text{or}\quad\lim\limits_{r\to\mathcal R^-}\varphi^\prime(r)=+\infty.
\end{equation}
\end{remark}


\section{Global Existence for the Radial Cauchy Problem}\label{Section4}
We derive two-sided derivative estimates for increasing convex radial solutions and use them to characterize global existence for \eqref{CP}. We treat conditions (A) and (B) separately.
Under condition (B), an exponentially weighted Keller–Osserman condition is required for the negative part of $g$.

\subsection{Global Existence under Condition (A) }\label{sec4.1}

We begin with the nonnegative coefficients in condition (A).

\begin{lemma}\label{Lemma4.1}
Assume that $f$ and $g$ satisfy condition (A), and that $1 \le k < n$ and $q \le 2$. Suppose that $R\in(0,+\infty]$ and $\varphi\in C^1[0,R)\cap C^2(0,R)$ is a solution of \eqref{sODE} satisfying $\varphi(0) = a > 0$, $\varphi^\prime(0) = 0$, $\varphi^\prime(r)>0$, and $\varphi^{\prime\prime}(r)\geq0$ for $r\in(0,R)$.
Then for all $r\in(0,R)$,
\begin{equation}\label{1prime_k_upper}
\varphi^{\prime}(r) \leq
\left\{
\begin{aligned}
& C_1 \left[\left(\int_a^{\varphi(r)} f(t) d t\right)^{\frac{1}{2}} + \left(\int_a^{\varphi(r)} g(t) d t\right)^{\frac{1}{2-q}} \right],&& q<2, \\
&\left(2 c_3^{-\frac{1}{k}} C_0 \int_a^{\varphi(r)} e^{2 c_3^{-\frac{1}{k}} C_0 \int_t^{\varphi(r)} g(\tau) d \tau} f(t) d t\right)^{\frac{1}{2}}, && q=2,
\end{aligned}
\right.
\end{equation}
and
\begin{equation}\label{1prime_k_lower}
\varphi^{\prime}(r) \geq
\left\{
\begin{aligned}
&C_2 \left[\left(\int_{a}^{\varphi(r)} f(t) d t\right)^{\frac{1}{2}} + \left(\int_{a}^{\varphi(r)} g(t) d t\right)^{\frac{1}{2-q}} \right],&& q<2, \\
&\left(\frac{1}{2}\right)^{n+k-2} \left( \frac{2k C_0}{n-1}\int_{\varphi(\frac{r}{2})}^{\varphi(r)} e^{\frac{2k C_0}{n-1} \int_t^{\varphi(r)} g(\tau) d\tau} f(t) dt \right)^{\frac{1}{2}},&& q=2.
\end{aligned}
\right.
\end{equation}
Here, $C_0, c_3$ are defined in \eqref{constants}, and $C_1,C_2>0$ depend only on $n$, $k$ and $q$.
\end{lemma}
\begin{proof}
We first prove  \eqref{1prime_k_upper}.
Since $c_1 \frac{\varphi^{\prime}}{r}\ge0$, $c_2 \frac{\varphi^{\prime}}{r}>0$ and $c_3>0$, then \eqref{sODE} yields
\begin{equation}\label{2prime_upper}
\varphi^{\prime\prime}(r) \le c_3^{-\frac1k}C_0 \left[ f(\varphi(r)) + g(\varphi(r))(\varphi^{\prime}(r))^q \right], \quad r\in(0,R).
\end{equation}

{\bfseries\boldmath Case 1: $q=0$.}
Multiplying \eqref{2prime_upper} by $\varphi^\prime(r)$ and integrating over $[0,r]$, we obtain
$$
\varphi^\prime(r)\leq2\left(c_3^{-\frac1k}C_0\right)^{\frac12}\left[\left(\int_a^{\varphi(r)}f(t)dt\right)^{\frac12}+\left(\int_a^{\varphi(r)}g(t)dt\right)^{\frac12}\right].
$$
This proves \eqref{1prime_k_upper} for $q=0$.

{\bfseries\boldmath Case 2: $0<q < 2$.}
Multiplying \eqref{2prime_upper} by $\varphi^\prime(r)$ and integrating from $0$ to $r$, we obtain
\begin{align*}
\frac{1}{2}(\varphi^{\prime}(r))^2
&\le c_3^{-\frac{1}{k}}C_0 \int_a^{\varphi(r)} f(t) dt + c_3^{-\frac{1}{k}}C_0 \int_0^r g(\varphi(s)) (\varphi^\prime(s))^q \varphi^\prime(s) ds\\
&\le c_3^{-\frac{1}{k}}C_0 \int_a^{\varphi(r)} f(t) dt + c_3^{-\frac{1}{k}}C_0 (\varphi^{\prime}(r))^q \int_a^{\varphi(r)} g(t) dt,
\end{align*}
where we used the monotonicity of $\varphi^\prime$, $\varphi(0)=a$, and $\varphi^\prime(0) = 0$.
Applying the $\varepsilon$-Young's inequality with $\varepsilon=\frac{1}{2q}$ to the second term on the right-hand side, we obtain
\begin{align*}
(\varphi^{\prime}(r))^2
\leq 4 c_3^{-\frac{1}{k}}C_0 \int_a^{\varphi(r)} f(t) dt+2(2-q) (2q)^{\frac{q}{2-q}} (c_3^{-\frac{1}{k}}C_0)^{\frac{2}{2-q}} \left( \int_a^{\varphi(r)} g(t) dt \right)^{\frac{2}{2-q}}.
\end{align*}
Taking square roots yields
\begin{align*}
\varphi^{\prime}(r)
&\leq 2(c_3^{-\frac{1}{k}}C_0)^{\frac{1}{2}} \left( \int_a^{\varphi(r)} f(t) dt \right)^{\frac{1}{2}}+\sqrt{2(2-q)} (2q)^{\frac{q}{2(2-q)}} (c_3^{-\frac{1}{k}}C_0)^{\frac{1}{2-q}} \left( \int_a^{\varphi(r)} g(t) dt \right)^{\frac{1}{2-q}}\\
&\leq \max \left\{ 2, \sqrt{2(2-q)} (2q)^{\frac{q}{2(2-q)}} \right\}  (c_3^{-\frac{1}{k}}C_0)^{\frac{1}{2}} \left[\left(\int_a^{\varphi(r)} f(t) d t\right)^{\frac{1}{2}} + \left(\int_a^{\varphi(r)} g(t) d t\right)^{\frac{1}{2-q}} \right],
\end{align*}
which gives \eqref{1prime_k_upper} for $0<q<2$.

{\bfseries\boldmath Case 3: $q < 0$.}
Multiplying \eqref{2prime_upper} by $(\varphi^\prime(r))^{1-q}$ and integrating from $0$ to $r$, we get
\begin{align*}
\frac{1}{2-q}(\varphi^{\prime}(r))^{2-q}
&\le c_3^{-\frac{1}{k}}C_0 \int_0^r f(\varphi(s)) (\varphi^\prime(s))^{-q}\varphi^\prime(s) ds + c_3^{-\frac{1}{k}}C_0 \int_a^{\varphi(r)} g(t) dt\\
&\leq c_3^{-\frac{1}{k}}C_0 (\varphi^\prime(r))^{-q} \int_a^{\varphi(r)} f(t) dt + c_3^{-\frac{1}{k}}C_0 \int_a^{\varphi(r)} g(t) dt,
\end{align*}
where we used $\varphi^\prime(0) = 0$ and the fact that $(\varphi^\prime(r))^{-q}$ is increasing.
Applying the $\varepsilon$-Young’s inequality with $\varepsilon = (-2q)^{\frac{q}{2-q}}$ to the first term on the right-hand side, we obtain
\begin{align*}
(\varphi^{\prime}(r))^{2-q}
\leq 4 (-2q)^{-\frac{q}{2}} \left( c_3^{-\frac{1}{k}}C_0 \int_a^{\varphi(r)} f(t) dt \right)^{\frac{2-q}{2}}+2(2-q) c_3^{-\frac{1}{k}}C_0 \int_a^{\varphi(r)} g(t) dt.
\end{align*}
Raising both sides to the power of $\frac{1}{2-q}\in (0,1)$, we deduce
\begin{align*}
\varphi^\prime(r)
\leq& \left(4(-2q)^{-\frac q2}\right)^{\frac1{2-q}}\left( c_3^{-\frac{1}{k}}C_0 \int_a^{\varphi(r)}f(t)dt\right)^{\frac12}+\left(2(2-q)\right)^{\frac1{2-q}}\left( c_3^{-\frac{1}{k}}C_0 \int_a^{\varphi(r)}g(t)dt\right)^{\frac1{2-q}}\\
\leq& \max \left\{ \left( 4 (-2q)^{-\frac{q}{2}} \right)^{\frac{1}{2-q}}, \left( 2(2-q) \right)^{\frac{1}{2-q}} \right\}  ( c_3^{-\frac{1}{k}}C_0)^{\frac{1}{2-q}}
\left[\left(\int_a^{\varphi(r)} f(t) d t\right)^{\frac{1}{2}} + \left(\int_a^{\varphi(r)} g(t) d t\right)^{\frac{1}{2-q}} \right],
\end{align*}
which gives \eqref{1prime_k_upper} for $q<0$.

{\bfseries\boldmath Case 4: $q = 2$.}
Multiplying \eqref{2prime_upper} by $2 e^{-2 c_3^{-\frac{1}{k}} C_0 \int_a^{\varphi(r)} g(\tau) d \tau} \varphi^{\prime}(r)$, we get
\begin{equation*}
\left( (\varphi^{\prime}(r))^2 e^{-2 c_3^{-\frac{1}{k}} C_0 \int_a^{\varphi(r)} g(\tau) d \tau} \right)^{\prime} \le 2 c_3^{-\frac{1}{k}} C_0 e^{-2 c_3^{-\frac{1}{k}} C_0 \int_a^{\varphi(r)} g(\tau) d \tau} \varphi^{\prime}(r) f(\varphi(r)).
\end{equation*}
Integrating from $0$ to $r$ and using $\varphi^{\prime}(0) = 0$ yields
\begin{equation*}
(\varphi^{\prime}(r))^2 e^{-2 c_3^{-\frac{1}{k}} C_0 \int_a^{\varphi(r)} g(\tau) d \tau} \le 2 c_3^{-\frac{1}{k}} C_0 \int_0^r e^{-2 c_3^{-\frac{1}{k}} C_0 \int_a^{\varphi(s)} g(\tau) d \tau} \varphi^{\prime}(s) f(\varphi(s)) ds.
\end{equation*}
Multiplying both sides by $e^{2 c_3^{-\frac{1}{k}} C_0 \int_a^{\varphi(r)} g(\tau) d \tau}$, we obtain
\begin{align*}
(\varphi^{\prime}(r))^2
\leq 2 c_3^{-\frac{1}{k}} C_0 \int_0^r e^{2 c_3^{-\frac{1}{k}} C_0 \int_{\varphi(s)}^{\varphi(r)} g(\tau) d \tau }\varphi^{\prime}(s) f(\varphi(s)) ds
= 2 c_3^{-\frac{1}{k}} C_0 \int_a^{\varphi(r)} e^{2 c_3^{-\frac{1}{k}} C_0 \int_t^{\varphi(r)} g(\tau) d \tau} f(t) dt.
\end{align*}
Taking the square root gives \eqref{1prime_k_upper} for $q=2$.

Next, we prove  \eqref{1prime_k_lower}. Since $\varphi^{\prime}>0$, $\varphi^{\prime\prime}\geq0$, $c_3\geq 1$, and $c_2>c_1\geq 0$, applying Young's inequality to the left-hand side of \eqref{sODE} gives
\begin{align*}
\left[\varphi^{\prime\prime} + c_1 \frac{\varphi^{\prime}}{r}\right]^{\frac{k-1}{k}} \left[c_3 \varphi^{\prime\prime} + c_2 \frac{\varphi^{\prime}}{r}\right]^{\frac{1}{k}}
&\le \frac{k-1}{k} \left(\varphi^{\prime\prime} + c_1 \frac{\varphi^{\prime}}{r}\right) + \frac{1}{k} \left(c_3 \varphi^{\prime\prime} + c_2 \frac{\varphi^{\prime}}{r}\right)\\
&=\frac{n-1}{k}\varphi^{\prime\prime} +\frac{(n-1)(n+k-2)}{k}\frac{\varphi^{\prime}}{r}
,\quad r\in(0,R).
\end{align*}
Set $\gamma := n+k-2\geq 1$. Substituting this inequality into \eqref{sODE}, we obtain
\begin{equation}\label{2prime_lower_qneg}
\varphi^{\prime\prime}(r) +\gamma \frac{\varphi^{\prime}(r)}{r} \ge \frac{k C_0}{n-1} \left[ f(\varphi(r)) + g(\varphi(r)) (\varphi^\prime(r))^q \right], \quad r\in(0,R).
\end{equation}
Since both terms on the right-hand side are nonnegative, \eqref{2prime_lower_qneg} implies the following two separate inequalities:
\begin{align}
\varphi^{\prime\prime}(r) + \gamma \frac{\varphi^\prime(r)}{r} &\ge \frac{k C_0}{n-1}f(\varphi(r)), \quad r\in(0,R),\label{eq3.3_qneg}\\
\varphi^{\prime\prime}(r) + \gamma \frac{\varphi^\prime(r)}{r} &\ge \frac{k C_0}{n-1}g(\varphi(r)) (\varphi^\prime(r))^q, \quad r\in(0,R). \label{3.3_qneg}
\end{align}
We divide the subsequent argument into two cases.

{\bfseries\boldmath Case 1: $q < 2$.}
Multiplying \eqref{eq3.3_qneg} by $2 r^{2\gamma} \varphi^\prime(r)$ and integrating over $[\frac{r}{2},r]$, we have
\begin{align*}
r^{2\gamma} (\varphi^\prime(r))^2
&\ge \left(\frac{r}{2}\right)^{2\gamma} \left(\varphi^\prime\left(\frac{r}{2}\right)\right)^2 + \frac{2k C_0}{n-1}\left(\frac{r}{2}\right)^{2\gamma} \int_{\frac{r}{2}}^r f(\varphi(s)) \varphi^{\prime}(s) ds \\
&\ge \frac{2k C_0}{n-1} \left(\frac{1}{2}\right)^{2\gamma} r^{2\gamma} \int_{\varphi(\frac{r}{2})}^{\varphi(r)} f(t) dt.
\end{align*}
Dividing both sides by $r^{2\gamma}$ and taking the square root gives
\begin{equation}\label{lower_f}
\varphi^\prime(r) \ge \sqrt{\frac{2k C_0}{n-1}} \left(\frac{1}{2}\right)^{\gamma} \left( \int_{\varphi(\frac{r}{2})}^{\varphi(r)} f(t) dt \right)^{\frac{1}{2}}.
\end{equation}
Using the fact that $\varphi^\prime$ and $f$ are non-decreasing, we have
\begin{align*}
\int_a^{\varphi(r)} f(t) dt= \int_a^{\varphi(\frac{r}{2})} f(t) dt + \int_{\varphi(\frac{r}{2})}^{\varphi(r)} f(t) dt\leq 2 \int_{\varphi(\frac{r}{2})}^{\varphi(r)} f(t) dt,\quad r\in(0,R).
\end{align*}
Substituting this into \eqref{lower_f}, we get
\begin{equation}\label{lower_f_}
\varphi^\prime(r) \ge \sqrt{\frac{2k C_0}{n-1}} \left(\frac{1}{2}\right)^{\gamma+\frac{1}{2}} \left( \int_{a}^{\varphi(r)} f(t) dt \right)^{\frac{1}{2}}.
\end{equation}
Similarly, multiplying \eqref{3.3_qneg} by $(2-q) r^{\gamma(2-q)} (\varphi^\prime(r))^{1-q}$ and integrating over $[\frac{r}{2},r]$, we have
\begin{align*}
r^{\gamma(2-q)} (\varphi^\prime(r))^{2-q}
&\ge \left(\frac{r}{2}\right)^{\gamma(2-q)} \left(\varphi^\prime\left(\frac{r}{2}\right)\right)^{2-q}+(2-q)\frac{k C_0}{n-1}\left(\frac{r}{2}\right)^{\gamma(2-q)} \int_{\varphi(\frac{r}{2})}^{\varphi(r)} g(t) dt \\
&\ge (2-q)\frac{k C_0}{n-1} \left(\frac{1}{2}\right)^{\gamma(2-q)} r^{\gamma(2-q)} \int_{\varphi(\frac{r}{2})}^{\varphi(r)} g(t) dt.
\end{align*}
Raising both sides to the positive power $\frac1{2-q}$ and using the monotonicity of $g$, we obtain
\begin{equation}\label{lower_g}
\varphi^\prime(r) \geq \left( (2-q)\frac{k C_0}{n-1}\right)^{\frac{1}{2-q}} \left(\frac{1}{2}\right)^{\gamma+\frac{1}{2-q}} \left( \int_{a}^{\varphi(r)} g(t) dt \right)^{\frac{1}{2-q}}.
\end{equation}
Combining \eqref{lower_f_} and \eqref{lower_g}, we get \eqref{1prime_k_lower} with
\begin{equation*}
C_2:=  \min\left\{ \left(\frac{k C_0}{n-1}\right)^{\frac{1}{2}}, \left( \frac{k C_0(2-q)}{2(n-1)}\right)^{\frac{1}{2-q}} \right\} \left(\frac{1}{2}\right)^{n+k-1}, \quad q<2.
\end{equation*}

{\bfseries\boldmath Case 2: $q = 2$.}
Rewrite \eqref{2prime_lower_qneg} with $q=2$ as
\begin{equation}\label{nevprime_lower}
\varphi^{\prime\prime}(r) + \gamma \frac{\varphi^\prime(r)}{r} - \frac{k C_0}{n-1}g(\varphi(r)) (\varphi^\prime(r))^2 \ge \frac{k C_0}{n-1}f(\varphi(r)), \quad r\in(0,R).
\end{equation}
Multiplying \eqref{nevprime_lower} by $2 r^{2\gamma} e^{- \frac{2k C_0}{n-1}\int_{a}^{\varphi(r)} g(\tau) d\tau} \varphi^\prime(r)$ yields
\begin{equation*}
\left( r^{2\gamma} e^{-\frac{2k C_0}{n-1}\int_{a}^{\varphi(r)} g(\tau) d\tau} (\varphi^\prime(r))^2 \right)^\prime \ge \frac{2k C_0}{n-1}r^{2\gamma} e^{-\frac{2k C_0}{n-1}\int_{a}^{\varphi(r)} g(\tau) d\tau} f(\varphi(r)) \varphi^\prime(r).
\end{equation*}
Integrating over $[\frac{r}{2},r]$, and using $s^{2\gamma} \ge (\frac{r}{2})^{2\gamma}$ for $s \in [\frac{r}{2},r]$, we obtain
\begin{align*}
r^{2\gamma} e^{-\frac{2k C_0}{n-1}\int_{a}^{\varphi(r)} g(\tau) d\tau} (\varphi^\prime(r))^2
\ge \frac{2k C_0}{n-1}\left(\frac{r}{2}\right)^{2\gamma} \int_{\varphi(\frac{r}{2})}^{\varphi(r)}  e^{-\frac{2k C_0}{n-1}\int_{a}^{t} g(\tau) d\tau} f(t)  dt.
\end{align*}
Multiplying both sides by $e^{\frac{2k C_0}{n-1}\int_{a}^{\varphi(r)} g(\tau) d\tau}$ and taking square roots yields \eqref{1prime_k_lower} for $q=2$.
\end{proof}

With the upper and lower derivative estimates established in Lemma \ref{Lemma4.1}, we now characterize the global solvability of \eqref{CP} under condition (A)  in terms of the Keller--Osserman condition \eqref{KO_1}.

\begin{theorem}[Global Existence under Condition (A)]\label{Th4.1}
Assume that $f$ and $g$ satisfy condition (A), and suppose further that either $f(t)>0$ for all $t>0$, or $q<1$.
Let $1 \le k < n$ and $q<2$. For any initial value $a>0$, the Cauchy problem \eqref{CP} admits a positive global solution $\varphi \in C^1[0,+\infty) \cap C^2(0,+\infty)$ satisfying $\varphi^\prime(r)>0$ and $\varphi^{\prime\prime}(r)\geq0$ for all $r>0$
if and only if condition \eqref{KO_1} holds.
Furthermore,
\begin{itemize}
\item[(i)] if $0 \le q < 2$, then $\varphi \in C^2[0,+\infty)$;
\item[(ii)] if $q<0$, then $\varphi \in W_{loc}^{2,\delta}[0,+\infty)$ for $\delta \in [1,1 - \frac{1}{q})$.
\end{itemize}
\end{theorem}

\begin{proof}
\underline{\textit{Sufficiency.}} Assume that \eqref{KO_1} holds and fix $a>0$. By Theorem \ref{Th3.1} and Remark \ref{Rem3.1}, let $\varphi$ be a positive increasing convex solution of \eqref{CP} on its maximal interval $[0,R)$, where $0<R\leq+\infty$. Suppose that $R<+\infty$. The estimate \eqref{1prime_k_upper} shows that $\varphi^\prime$ is bounded whenever $\varphi$ is bounded on $[0,R)$. Thus \eqref{blow-up-alternative} gives $\varphi(r)\to+\infty$ as $r\to R^-$. Since $f,g\geq0$, integrating \eqref{1prime_k_upper} and letting $r\to R^-$ yield
$$
\int_a^{+\infty}\frac{dt}{\left(\int_0^t f(s)ds\right)^{\frac12}+\left(\int_0^t g(s)ds\right)^{\frac1{2-q}}}\leq C_1R<+\infty,
$$
which contradicts \eqref{KO_1}. Consequently,  \eqref{CP} admits a positive global solution for any $a>0$.

\underline{\textit{Necessity.}} Suppose on the contrary that there exists $A_0>0$ such that
\begin{equation*}
\int_{A_0}^{+\infty}\frac{dt}{\left(\int_0^t f(s)ds\right)^{\frac12}+\left(\int_0^t g(s)ds\right)^{\frac1{2-q}}}<+\infty.
\end{equation*}
Since $\varphi^\prime$ is positive and non-decreasing, $\varphi(r)\to+\infty$ as $r\to+\infty$. Choose $r_1>0$ such that $A:=\varphi(r_1)>2\max\{a,A_0\}$. For $t\geq A$, the monotonicity of $f$ gives $\int_a^t f(s)ds\geq\int_{\frac t2}^t f(s)ds\geq\int_0^{\frac t2}f(s)ds$, and the same comparison holds for $g$. Integrating \eqref{1prime_k_lower} and changing variables therefore yield
\begin{equation*}
\begin{aligned}
C_2(r-r_1)&\leq\int_A^{\varphi(r)}\frac{dt}{\left(\int_a^t f(s)ds\right)^{\frac12}+\left(\int_a^t g(s)ds\right)^{\frac1{2-q}}}\\
&\leq2\int_{\frac A2}^{+\infty}\frac{dt}{\left(\int_0^t f(s)ds\right)^{\frac12}+\left(\int_0^t g(s)ds\right)^{\frac1{2-q}}}<+\infty,\quad r>r_1.
\end{aligned}
\end{equation*}
The right-hand side is independent of $r$. Letting $r\to+\infty$ gives a contradiction and proves the necessity of \eqref{KO_1}.

Finally, the regularity properties in (i) and (ii) at the origin follow directly from Theorem \ref{Th3.1}.
\end{proof}

\subsection{Global Existence under Condition (B)}\label{sec4.2}
In this subsection, we consider condition (B), where the function $g$ is allowed to change sign.
In this setting, the negative part of $g$ requires additional weighted estimates.

\begin{lemma}\label{Lemma4.1-2}
Assume that $f$ and $g$ satisfy condition (B), and that $1\leq k<n$ and $0<q\leq2$.
Suppose that $R\in(0,+\infty]$ and $\varphi\in C^2[0,R)$ is a solution of \eqref{sODE} satisfying $\varphi(0) = a>0$, $\varphi^\prime(0) = 0$, $\varphi^\prime(r)>0$, and $\varphi^{\prime\prime}(r)\geq0$ for $r\in(0,R)$.
Then for all $r\in(0,R)$,
\begin{equation}\label{1prime_k_upper_2}
\varphi^{\prime}(r) \leq
\left\{
\begin{aligned}
& C_1 \left[\left(\int_a^{\varphi(r)} f(t) d t\right)^{\frac{1}{2}} + \left(\int_a^{\varphi(r)} g^+(t) d t\right)^{\frac{1}{2-q}} \right],&& 0<q<2, \\
&\left(2 c_3^{-\frac{1}{k}} C_0 \int_a^{\varphi(r)}  f(t)e^{2 c_3^{-\frac{1}{k}} C_0 \int_t^{\varphi(r)} g^+(\tau) d \tau} d t\right)^{\frac{1}{2}}, && q=2.
\end{aligned}
\right.
\end{equation}
If we further assume $g(a)\geq 0$, then for all $r\in(0,R)$,
\begin{equation}\label{1prime_k_lower_2}
\varphi^{\prime}(r) \geq
\left\{
\begin{aligned}
&C_2 \left[\left(\int_{a}^{\varphi(r)} f(t) d t\right)^{\frac{1}{2}} + \left(\int_{a}^{\varphi(r)} g^+(t) d t\right)^{\frac{1}{2-q}} \right],&& 0<q<2, \\
&\left(\frac{1}{2}\right)^{n+k-2} \left( \frac{2k C_0}{n-1}\int_{\varphi(\frac{r}{2})}^{\varphi(r)}  f(t)e^{\frac{2k C_0}{n-1} \int_t^{\varphi(r)} g^+(\tau) d\tau}  dt \right)^{\frac{1}{2}},&& q=2.
\end{aligned}
\right.
\end{equation}
Moreover, for $0 < q \le 2$,
\begin{equation}\label{1prime_k_lower_uni}
\varphi^\prime(r)\geq \left(\frac{1}{2}\right)^{n+k-2}\left( \frac{q k C_0}{2(n-1)} \int_{a}^{\varphi(r)} f(t) e^{- \frac{qk C_0}{n-1} \int_{t}^{\varphi(r)} P(\tau)d\tau} dt\right)^{\frac{1}{2}},\quad r\in(0,R).
\end{equation}
If we further assume $g(t)\leq 0$ for all $t\in\mathbb{R}$, for $0 < q \le 2$, we have
\begin{equation}\label{1prime_k_upper_uni}
\varphi^\prime(r)\leq \left(2 c_3^{-\frac1k}C_0  \int_a^{\varphi(r)}  f(t) e^{-2 c_3^{-\frac1k}C_0 \int_t^{\varphi(r)} P(\tau) d\tau} dt\right)^{\frac{1}{2}},\quad r\in(0,R),
\end{equation}
where $P(t) :=  \frac{(g^-(t))^{\frac{2}{q}}}{(f(t))^{\frac{2}{q}-1}}$ for $t>0$. Here, $C_0, c_3$ are defined in \eqref{constants}, and $C_1,C_2>0$ depend only on $n$, $k$ and $q$.
\end{lemma}

\begin{proof}
Estimate \eqref{1prime_k_upper_2} follows from the proof of \eqref{1prime_k_upper} with $g$ replaced by $g^+$. If $g(a)\geq0$, then $g(\varphi)\geq0$, and the proof of \eqref{1prime_k_lower} also gives \eqref{1prime_k_lower_2}.

We prove \eqref{1prime_k_lower_uni} for $0<q\leq2$. Set $M_0:=\frac{kC_0}{n-1}$ and $\gamma:=n+k-2$. By \eqref{2prime_lower_qneg} and $g\geq-g^-$,
\begin{equation}\label{4.21}
\varphi^{\prime\prime}(r)+\gamma\frac{\varphi^\prime(r)}r+M_0g^-(\varphi(r))(\varphi^\prime(r))^q\geq M_0f(\varphi(r)),\quad r\in(0,R).
\end{equation}
By Young's inequality for $0<q<2$, and the identity $P=g^-$ for $q=2$, we have
$$
g^-(\varphi)(\varphi^\prime)^q\leq\frac q2P(\varphi)(\varphi^\prime)^2+\left(1-\frac q2\right)f(\varphi).
$$
Substituting this into \eqref{4.21} yields
\begin{equation*}
\varphi^{\prime\prime}(r)+\gamma\frac{\varphi^\prime(r)}r+\frac q2M_0P(\varphi(r))(\varphi^\prime(r))^2\geq\frac q2M_0f(\varphi(r)).
\end{equation*}
Multiplying by $2r^{2\gamma}\varphi^\prime(r)e^{qM_0\int_a^{\varphi(r)}P(t)dt}$ and integrating over $[\frac r2,r]$ yield
\begin{equation}\label{4.25}
(\varphi^\prime(r))^2\geq qM_02^{-2\gamma}\int_{\varphi(\frac r2)}^{\varphi(r)}f(t)e^{-qM_0\int_t^{\varphi(r)}P(\tau)d\tau}dt.
\end{equation}
For any fixed $r>0$, define
$$
\mathcal{I}(t):=f(t)e^{-qM_0\int_t^{\varphi(r)}P(\tau)d\tau}, \quad a\leq t\leq\varphi(r).
$$
Since $f$ is non-decreasing and $P\geq0$, the function $\mathcal{I}$ is non-decreasing.
Furthermore, the monotonicity of $\varphi^\prime$ and $\mathcal{I}$ implies
\begin{equation*}
\int_a^{\varphi(r)} \mathcal{I}(t) dt = \int_a^{\varphi(\frac{r}{2})} \mathcal{I}(t) dt + \int_{\varphi(\frac{r}{2})}^{\varphi(r)} \mathcal{I}(t) dt \le 2 \int_{\varphi(\frac{r}{2})}^{\varphi(r)} \mathcal{I}(t) dt,\quad r\in(0,R).
\end{equation*}
Combining this inequality with \eqref{4.25}, we obtain
\begin{equation*}
(\varphi^\prime(r))^2 \geq q M_0 \left(\frac{1}{2}\right)^{2\gamma+1} \int_{a}^{\varphi(r)}f(t) e^{-q M_0\int_t^{\varphi(r)}  P(\tau) d\tau}  dt.
\end{equation*}
Taking square roots gives \eqref{1prime_k_lower_uni}.

We next prove \eqref{1prime_k_upper_uni}. Since $g=-g^-$, estimate \eqref{2prime_upper} gives
\begin{equation*}
\varphi^{\prime\prime}(r)+c_3^{-\frac1k}C_0g^-(\varphi(r))(\varphi^\prime(r))^q\leq c_3^{-\frac1k}C_0f(\varphi(r)),\quad r\in(0,R).
\end{equation*}
Where $g^-(\varphi)>0$, Lemma \ref{Lem2.3}~(iii) and $0<q\leq2$ imply
$$
g^-(\varphi)(\varphi^\prime)^q=\frac{g^-(\varphi)}{(\varphi^\prime)^{2-q}}(\varphi^\prime)^2\geq g^-(\varphi)\left(\frac{g^-(\varphi)}{f(\varphi)}\right)^{\frac{2-q}{q}}(\varphi^\prime)^2=P(\varphi)(\varphi^\prime)^2.
$$
The same inequality is immediate when $g^-(\varphi)=0$. Hence
\begin{equation*}
\varphi^{\prime\prime}(r)+c_3^{-\frac1k}C_0P(\varphi(r))(\varphi^\prime(r))^2\leq c_3^{-\frac1k}C_0f(\varphi(r)).
\end{equation*}
Multiplying by $2\varphi^\prime(r)e^{2c_3^{-\frac1k}C_0\int_a^{\varphi(r)}P(t)dt}$, integrating from $0$ to $r$, and using $\varphi^\prime(0)=0$, we obtain
$$
(\varphi^\prime(r))^2\leq2c_3^{-\frac1k}C_0\int_a^{\varphi(r)}f(t)e^{-2c_3^{-\frac1k}C_0\int_t^{\varphi(r)}P(\tau)d\tau}dt.
$$
Taking square roots proves \eqref{1prime_k_upper_uni} and completes the proof.
\end{proof}

Based on the preceding derivative estimates, we now characterize the global solvability of \eqref{CP} under condition (B) in terms of the corresponding Keller--Osserman conditions.

\begin{theorem}[Global Existence under Condition (B)]\label{Th4.1-2}
Assume that $f$ and $g$ satisfy condition (B), and let $1\leq k<n$.
Then the following conclusions hold.
\begin{itemize}
\item[(i)] Let $0<q<2$ and assume that $\lim\limits_{t\to+\infty}g(t)>0$. For any initial value $a>0$, the Cauchy problem \eqref{CP} admits a positive global solution $\varphi \in C^2[0,+\infty)$ satisfying $\varphi^\prime(r)>0$ and $\varphi^{\prime\prime}(r)\geq0$ for all $r>0$
if and only if condition \eqref{KO_2} holds.

\item[(ii)] Let $0<q\leq 2$ and assume that $\lim\limits_{t\to+\infty}g(t)\leq 0$. For any initial value $a>0$, the Cauchy problem \eqref{CP} admits a positive global solution $\varphi \in C^2[0,+\infty)$ satisfying $\varphi^\prime(r)>0$ and $\varphi^{\prime\prime}(r)\geq0$ for all $r>0$ if and only if condition \eqref{KO_3} holds.
\end{itemize}
\end{theorem}

\begin{proof}
We first prove (i). The sufficiency follows from \eqref{1prime_k_upper_2} by the argument in Theorem \ref{Th4.1}, with $g$ replaced by $g^+$.
For the necessity, let $\varphi$ be a positive increasing convex global solution. Then $\varphi(r)\to+\infty$ as $r\to+\infty$. Choose $r_0>0$ such that $g(\varphi(r))>0$ for $r\geq r_0$, and set $A:=\varphi(r_0)$. Repeating the argument leading to \eqref{1prime_k_lower_2} for sufficiently large $r$, with $a$ replaced by $A$ and $g$ by $g^+$, and then applying the necessity argument in Theorem \ref{Th4.1}, we obtain \eqref{KO_2}.

We next prove (ii). The assumptions imply that $g\leq0$, so $g^-=-g$ is non-increasing. Since $f$ is positive and non-decreasing on $(0,+\infty)$ and $0<q\leq2$, the function $P$ in \eqref{P} is nonnegative and non-increasing. For $\lambda>0$, $b\geq0$, and $t>b$, set
$$
J_{\lambda,b}^-(t):=\int_b^t f(s)e^{-\lambda\int_s^tP(\tau)d\tau}ds.
$$

\underline{\textit{Sufficiency.}} Assume that \eqref{KO_3} holds and fix $a>0$. By Theorem \ref{Th3.1-2} and Remark \ref{Rem3.1}, let $\varphi$ be a positive increasing convex solution on its maximal interval $[0,R)$. Suppose that $R<+\infty$. Estimate \eqref{1prime_k_upper_uni} bounds $\varphi^\prime$ whenever $\varphi$ is bounded, so \eqref{blow-up-alternative} gives $\varphi(r)\to+\infty$ as $r\to R^-$. Set $\lambda_1:=2c_3^{-\frac1k}C_0$. Integrating \eqref{1prime_k_upper_uni}, using $J_{\lambda_1,0}^-\geq J_{\lambda_1,a}^-$, and letting $r\to R^-$ yield
$$
\int_a^{+\infty}\frac{dt}{\sqrt{J_{\lambda_1,0}^-(t)}}\leq\sqrt{\lambda_1}R<+\infty.
$$
By \cite[Lemma 2.4~(i)]{JJL2026}, this contradicts \eqref{KO_3}. Hence $R=+\infty$.

\underline{\textit{Necessity.}} Suppose that a positive increasing convex global solution exists but \eqref{KO_3} fails. Set $\lambda_0:=\frac{qkC_0}{n-1}$. By \cite[Lemma 2.4~(i)]{JJL2026}, there exists $A_0>0$ such that
\begin{equation*}
\int_{A_0}^{+\infty}\frac{dt}{\sqrt{J_{\lambda_0,0}^-(t)}}<+\infty.
\end{equation*}
Since $\varphi(r)\to+\infty$ as $r\to+\infty$, we can choose $r_1>0$ such that $A:=\varphi(r_1)>2\max\{a,A_0\}$. The monotonicity of $f$ and the non-increasing property of $P$ give
$$
J_{\lambda_0,a}^-(t)\geq\int_{\frac t2}^t f(s)e^{-\lambda_0\int_s^tP(\tau)d\tau}ds\geq J_{\lambda_0,0}^-\left(\frac t2\right),\quad t\geq A.
$$
Integrating \eqref{1prime_k_lower_uni} therefore yields
\begin{equation*}
\begin{aligned}
2^{-(n+k-\frac32)}\sqrt{\lambda_0}(r-r_1)&\leq\int_A^{\varphi(r)}\frac{dt}{\sqrt{J_{\lambda_0,a}^-(t)}}\\
&\leq2\int_{\frac A2}^{+\infty}\frac{dt}{\sqrt{J_{\lambda_0,0}^-(t)}}<+\infty,\quad r>r_1.
\end{aligned}
\end{equation*}
Letting $r\to+\infty$ gives a contradiction and proves \eqref{KO_3}. Finally, the regularity at the origin in both cases follows from Theorem \ref{Th3.1-2}.
\end{proof}

\begin{remark}\label{Rem4.1}
Let $1\leq k<n$ and $0<q<2$. Under condition \textup{(B)}, if $\lim\limits_{t\to+\infty}g(t)>0$, then $q\leq1$ is necessary for the existence of a positive increasing convex global solution of \eqref{CP}.
Indeed, choose $t_0>0$ such that $g(t_0)>0$. The monotonicity of $g$ gives
$$
\int_0^t g^+(s)ds\geq\int_{t_0}^t g(s)ds\geq g(t_0)(t-t_0)\geq\frac{g(t_0)}2t,\quad t\geq2t_0.
$$
Consequently,
\begin{equation*}
\int_{2t_0}^{+\infty}\frac{dt}{\left(\int_0^t f(s)ds\right)^{\frac12}+\left(\int_0^t g^+(s)ds\right)^{\frac1{2-q}}}\leq\left(\frac2{g(t_0)}\right)^{\frac1{2-q}}\int_{2t_0}^{+\infty}\frac{dt}{t^{\frac1{2-q}}}.
\end{equation*}
For $1<q<2$, the right-hand side is finite since $\frac1{2-q}>1$. Thus \eqref{KO_2} fails, and Theorem \ref{Th4.1-2}~(i) implies that \eqref{CP} admits no positive increasing convex global solution for any initial value $a>0$. Under the hypotheses of Theorem \ref{Th4.1}, the same argument applies to \eqref{KO_1}, since $g^+=g$, and yields the corresponding nonexistence conclusion under condition \textup{(A)}.
\end{remark}

\vspace{3mm}

\section{Proofs of the Main Results}\label{Section5}

We prove Theorems \ref{Th1.1} and \ref{Th1.2} using the comparison and equivalence principles established for the determinant case in \cite[Lemma 5.1 and Theorem 5.1]{JJL2026}. We state their counterparts for $1\leq k<n$ and record the required modification to the linearization.

\subsection{Comparison and Equivalence Principles}

\begin{lemma}[Local comparison principle]\label{Lem6.1}
Let $u$ be a positive admissible entire subsolution of \eqref{eq1.1}, and let $\varphi$ be a positive solution of \eqref{CP} on $[0,R)$, where $0<R<+\infty$, satisfying $\varphi^\prime(r)>0$, $\varphi^{\prime\prime}(r)\geq0$ for $r\in(0,R)$, and $\lim\limits_{r\to R^-}\varphi(r)=+\infty$.
Set $v(x):=\varphi(|x|)$ in $B_R$. Suppose that one of the following conditions holds:
\begin{enumerate}[label=(\roman*)]
\item $q=0$, condition \textup{(A)} holds, and $u\in C^2(\mathbb R^n)$;
\item $q>0$, condition \textup{(A)} or \textup{(B)} holds, $f(t)>0$ for $t>0$, and $u\in C^2(\mathbb R^n)$;
\item $q<0$, condition \textup{(A)} holds, and $u$ is a strong subsolution in the sense of Definition \ref{def:entire_admissible_subsolutions};
\item $0<q<1$, condition \textup{(A)} holds, $u\in C^2(\mathbb R^n)$, and $Du(0)\neq0$.
\end{enumerate}
Then
\begin{equation}\label{eq6.1}
u(x)\leq v(x),\quad x\in B_R.
\end{equation}
\end{lemma}

\begin{proof}

We follow the proof and notation of \cite[Lemma 5.1]{JJL2026}. Set $\mathcal L[w]:=S_k^{\frac1k}[\Delta wI-D^2w]$, $w:=u-v$, $u_\theta:=\theta u+(1-\theta)v$, and $M_\theta:=\Delta u_\theta I-D^2u_\theta$. At the contact points considered there, admissibility and the strictly positive radial inhomogeneous term place both transformed Hessian matrices in $\Gamma_k$. By continuity and the convexity of $\Gamma_k$, we have $\lambda(M_\theta)\in\Gamma_k$ for $\theta\in[0,1]$ in a neighborhood of each such point. Hence
$$
\mathcal L[u]-\mathcal L[v]=\sum_{i,j=1}^n a^{ij}D_{ij}w,
$$
where
$$
a^{ij}(x):=\frac1k\int_0^1 S_k[M_\theta]^{\frac1k-1}\left[\left(\sum_{\ell=1}^n S_k^{\ell\ell}[M_\theta]\right)\delta_{ij}-S_k^{ij}[M_\theta]\right]d\theta,
$$
and $S_k^{ij}[M]$ denotes the $(i,j)$ entry of the matrix derivative $D_MS_k[M]$. Since $(S_k^{ij}[M])$ is positive definite for $\lambda(M)\in\Gamma_k$, the coefficient matrix $(a^{ij})$ is locally uniformly positive definite. The remaining gradient estimates, translation argument, and application of the strong maximum principle proceed as in \cite[Lemma 5.1]{JJL2026}, yielding \eqref{eq6.1}.

\end{proof}

\begin{theorem}[Equivalence theorem]\label{Th6.1}
Let $1\leq k<n$ and $q\leq2$, and assume that one of the following sets of hypotheses holds.
\begin{enumerate}[label=(\alph*)]
\item Condition (A) holds and either $f(t)>0$ for all $t>0$, or $q<1$.
\item Condition (B) holds and $q>0$.
\end{enumerate}
Then the following assertions hold.
\begin{enumerate}[label=(\roman*)]
\item If $q\geq0$, equation \eqref{eq1.1} admits a positive nontrivial admissible classical entire subsolution if and only if, for some $a>0$, problem \eqref{CP} admits a positive global solution $\varphi\in C^2[0,+\infty)$ satisfying $\varphi^\prime(r)>0$ and $\varphi^{\prime\prime}(r)\geq0$ for all $r>0$.

\item If $q<0$, equation \eqref{eq1.1} admits a positive nontrivial admissible strong entire subsolution $u\in W_{loc}^{2,\delta}(\mathbb R^n)$ for $\delta\in[1,1-\frac1q)$ if and only if, for some $a>0$, problem \eqref{CP} admits a positive global solution $\varphi\in C^1[0,+\infty)\cap C^2(0,+\infty)\cap W_{loc}^{2,\delta}[0,+\infty)$ for $\delta\in[1,1-\frac1q)$ satisfying $\varphi^\prime(r)>0$ and $\varphi^{\prime\prime}(r)\geq0$ for all $r>0$.
\end{enumerate}
\end{theorem}

\begin{proof}
The necessity follows from the argument in \cite[Theorem 5.1]{JJL2026}, using the local existence results in Section \ref{Section3}, the derivative estimates in Section \ref{Section4}, and Lemma \ref{Lem6.1}. Conversely, a positive increasing convex global solution $\varphi$ yields the admissible entire solution $u(x):=\varphi(|x|)$. The Sobolev regularity in (ii) follows from Lemma \ref{Lem2.2}~(iv) and Remark \ref{Rem2.2}.
\end{proof}

\subsection{Proofs of Theorems \ref{Th1.1} and \ref{Th1.2}}

Theorem \ref{Th1.1} follows from the equivalence in Theorem \ref{Th6.1} and the global existence result in Theorem \ref{Th4.1}. The regularity at the origin follows from Theorem \ref{Th3.1} and Remark \ref{Rem2.2}.
Similarly, Theorem \ref{Th6.1}~(i), together with Theorem \ref{Th4.1-2}, establishes Theorem \ref{Th1.2}.

\begin{remark}\label{Rem6.2}
If $f$ is strictly positive on $\mathbb R$, with the remaining coefficient assumptions unchanged, Lemma \ref{Lem6.1} and Theorem \ref{Th6.1} remain valid without positivity restrictions on $u$ and $\varphi$, and the initial value may be any $a\in\mathbb R$. The Keller--Osserman conditions are unchanged, since their convergence at infinity is independent of finite lower integration limits.
\end{remark}

\section{Boundary Blow-Up Solutions in Balls}\label{Section6}

We prove Theorems \ref{Th1.3} and \ref{Th1.4}, establishing necessary and sufficient conditions for radial boundary blow-up solutions and two-sided estimates for their blow-up radii. We subsequently derive upper bounds for the corresponding profiles. The Keller--Osserman functions are defined in Definition \ref{Def1.boundary_KO}.

\begin{definition}\label{Def1.radial_solutions}
Let $a\in\mathbb R$ and $0<\mathcal R<+\infty$. Suppose that $\varphi\in C^1[0,\mathcal R)\cap C^2(0,\mathcal R)$ solves \eqref{CP} with $\varphi(0)=a$ and satisfies $\varphi^\prime(r)>0$ and $\varphi^{\prime\prime}(r)\geq0$ for $r\in(0,\mathcal R)$. The associated function $u_a(x):=\varphi(|x|)$ is called a \emph{radial boundary blow-up solution} of \eqref{eq1.2} in $B_{\mathcal R}$ if
$$
\lim\limits_{|x|\to\mathcal R^-}u_a(x)=+\infty.
$$
The value $a=u_a(0)$ is its center value, and $\mathcal R$ is its blow-up radius.
\end{definition}

\subsection{Boundary Blow-Up Conditions and Radius Estimates}

\begin{proof}[Proof of Theorem \ref{Th1.3}]
By Definition \ref{Def1.radial_solutions}, a radial boundary blow-up solution has the form $u_a(x)=\varphi_a(|x|)$, where $\varphi_a$ is a positive solution of \eqref{CP} satisfying $\varphi_a^\prime(r)>0$, $\varphi_a^{\prime\prime}(r)\geq0$, and $\lim\limits_{r\to R^-}\varphi_a(r)=+\infty$.

We first prove (i), where $q<2$.

\underline{\textit{Sufficiency.}} Assume that \eqref{KO_4} holds. Let $\varphi_a$ be the local solution given by Theorem \ref{Th3.1}, and let $[0,R)$ be its maximal interval of existence. Since $\Psi_k(a)<+\infty$, condition \eqref{KO_1} fails, and Theorem \ref{Th4.1} gives $R<+\infty$. If $\varphi_a$ remained bounded, \eqref{1prime_k_upper} would also bound $\varphi_a^\prime$, contradicting Remark \ref{Rem3.1}. Hence, $\lim\limits_{r\to R^-}\varphi_a(r)=+\infty$
and $u_a(x)=\varphi_a(|x|)$ is the required solution in $B_R$.

\underline{\textit{Necessity.}} Let $u_a(x)=\varphi_a(|x|)$ be a positive radial admissible boundary blow-up solution in $B_R$ with center value $a$. The estimates \eqref{1prime_k_upper} and \eqref{1prime_k_lower} give
$$
C_2\leq\frac{\varphi_a^\prime(r)}{\left(\int_a^{\varphi_a(r)}f(t)dt\right)^{\frac12}+\left(\int_a^{\varphi_a(r)}g(t)dt\right)^{\frac1{2-q}}}\leq C_1,\quad 0<r<R.
$$
Integrating from $0$ to $r$, we obtain
$$
C_2r\leq\int_a^{\varphi_a(r)}\frac{dt}{\left(\int_a^t f(s)ds\right)^{\frac12}+\left(\int_a^t g(s)ds\right)^{\frac1{2-q}}}\leq C_1r.
$$
Letting $r\to R^-$ yields $C_2R\leq\Psi_k(a)\leq C_1R$. Thus \eqref{KO_4} holds, and
$$
\frac{\Psi_k(a)}{C_1}\leq R\leq\frac{\Psi_k(a)}{C_2}.
$$
This proves (i).

We next prove (ii), where $q=2$. In this case, the hypotheses imply that $f(t)>0$ for all $t>0$.

\underline{\textit{Necessity.}} Let $u_a(x)=\varphi_a(|x|)$ be a positive radial admissible boundary blow-up solution in $B_R$. Set
$$
\lambda_1:=2c_3^{-\frac1k}C_0.
$$
The upper estimate \eqref{1prime_k_upper} gives
$$
\frac{\varphi_a^\prime(r)}{\left(\int_a^{\varphi_a(r)}f(t)e^{\lambda_1\int_t^{\varphi_a(r)}g(\tau)d\tau}dt\right)^{\frac12}}\leq\sqrt{\lambda_1},\quad 0<r<R.
$$
Integrating from $0$ to $r$ and letting $r\to R^-$, we obtain
\begin{equation}\label{critical_A_lower_radius}
\Psi_k^{(\lambda_1)}(a)\leq\sqrt{\lambda_1}R<+\infty.
\end{equation}
The parameter comparison in \cite[Lemma 2.4~(ii)]{JJL2026} then yields \eqref{KO_5}.

Assume additionally that $\Psi_k^{(0)}(a)<+\infty$. Convexity of $\varphi_a$ and monotonicity of $f$ give
$\int_{\varphi_a(\frac r2)}^{\varphi_a(r)}f(t)dt\geq\frac12\int_a^{\varphi_a(r)}f(t)dt$.
Since $g\geq0$, the lower estimate \eqref{1prime_k_lower} implies
\begin{equation}\label{6.1}
\varphi_a^\prime(r)\geq\gamma_1\left(\int_a^{\varphi_a(r)}f(t)dt\right)^{\frac12},\quad 0<r<R,
\end{equation}
where $\gamma_1:=2^{-(n+k-2)}\left(\frac{kC_0}{n-1}\right)^{\frac12}$. Integrating from $0$ to $r$ yields
\begin{equation}\label{6.2}
\gamma_1r\leq\int_a^{\varphi_a(r)}\frac{dt}{\left(\int_a^t f(s)ds\right)^{\frac12}},\quad 0<r<R.
\end{equation}
Letting $r\to R^-$, combining with \eqref{critical_A_lower_radius}, we obtain
$$
\frac{\Psi_k^{(\lambda_1)}(a)}{\sqrt{\lambda_1}}\leq R\leq\frac{\Psi_k^{(0)}(a)}{\gamma_1}.
$$

\underline{\textit{Sufficiency.}} Assume that \eqref{KO_5} holds. Let $\varphi_a$ be the local solution given by Theorem \ref{Th3.1}, with maximal interval $[0,R_*)$, where $R_*\in(0,+\infty]$.
If $g\equiv0$ on $[a,+\infty)$, then $\Psi_k^{(\lambda)}(a)=\Psi_k^{(0)}(a)<+\infty$. Applying \eqref{6.2} on the interval of existence gives
$$
\gamma_1r\leq\Psi_k^{(0)}(a),\quad 0<r<R_*,
$$
and hence $R_*<+\infty$.

Suppose instead that $g\not\equiv0$ on $[a,+\infty)$. Choose $T>a$ such that $m:=g(T)>0$. If $R_*=+\infty$, positivity and monotonicity of $\varphi_a^\prime$ imply that $\varphi_a(r)\to+\infty$ as $r\to+\infty$. We may therefore choose $r_0>0$ such that
$$
\varphi_a\left(\frac{r_0}{2}\right)\geq T\quad\text{and}\quad\varphi_a(r_0)\geq a+\frac{(n-1)\log2}{kmC_0}.
$$
For $r\geq r_0$ and $t\in[\varphi_a(\frac r2),\varphi_a(r)]$, we have $f(t)\geq f(T)>0$ and $g(\tau)\geq m$ for $\tau\in[t,\varphi_a(r)]$. Thus \eqref{1prime_k_lower} gives
\begin{align*}
\varphi_a^\prime(r)&\geq2^{-(n+k-2)}\left(\frac{2kC_0}{n-1}\int_{\varphi_a(\frac r2)}^{\varphi_a(r)}f(t)e^{\frac{2kC_0}{n-1}\int_t^{\varphi_a(r)}g(\tau)d\tau}dt\right)^{\frac12}\\
&\geq2^{-(n+k-2)}\left(\frac{2kC_0f(T)}{n-1}\int_{\varphi_a(\frac r2)}^{\varphi_a(r)}e^{\frac{2kmC_0}{n-1}(\varphi_a(r)-t)}dt\right)^{\frac12}\\
&=2^{-(n+k-2)}\left(\frac{f(T)}m\right)^{\frac12}\left(e^{\frac{2kmC_0}{n-1}\left(\varphi_a(r)-\varphi_a(\frac r2)\right)}-1\right)^{\frac12}.
\end{align*}
Convexity gives $\varphi_a(r)-\varphi_a(r/2)\geq(\varphi_a(r)-a)/2$, and the choice of $r_0$ ensures that $e^{\frac{kmC_0}{n-1}(\varphi_a(r)-a)}\geq2$ for $r\geq r_0$. Consequently,
\begin{align*}
\varphi_a^\prime(r)\geq2^{-(n+k-\frac32)}\left(\frac{f(T)}m\right)^{\frac12}e^{\frac{kmC_0}{2(n-1)}(\varphi_a(r)-a)}=:A_1e^{A_2\varphi_a(r)},
\end{align*}
where $A_1:=2^{-(n+k-\frac32)}\left(\frac{f(T)}m\right)^{\frac12}e^{-\frac{kmC_0}{2(n-1)}a}>0$ and $A_2:=\frac{kmC_0}{2(n-1)}>0$. Integrating from $r_0$ to $r>r_0$, we obtain
$$
r-r_0\leq\int_{\varphi_a(r_0)}^{\varphi_a(r)}\frac{dt}{A_1e^{A_2t}}\leq\int_{\varphi_a(r_0)}^{+\infty}\frac{dt}{A_1e^{A_2t}}=\frac{e^{-A_2\varphi_a(r_0)}}{A_1A_2}<+\infty.
$$
Letting $r\to+\infty$ gives a contradiction. Hence $R_*<+\infty$.

In either case, the upper estimate \eqref{1prime_k_upper} and Remark \ref{Rem3.1} imply $\lim\limits_{r\to R_*^-}\varphi_a(r)=+\infty$.
Thus $u_a(x)=\varphi_a(|x|)$ is the required solution in $B_{R_*}$. Setting $R:=R_*$ completes the proof of (ii).

Finally, Theorem \ref{Th3.1} and Remark \ref{Rem2.2} give $u_a\in C^2(B_R)$ when $q\geq0$. When $q<0$, they yield an admissible strong solution satisfying $u_a\in  W_{\mathrm{loc}}^{2,\delta}(B_R)$ for $\delta\in[1,1-\frac1q)$. The proof is complete.
\end{proof}

\begin{remark}\label{Rem6.criticalA}
Under the hypotheses of Theorem \ref{Th1.3} with $q=2$, we have $f(t)>0$ for all $t>0$. For each $a>0$, condition \eqref{KO_5} is equivalent to
\begin{equation*}
\Psi_k^{(0)}(a)=\int_a^{+\infty}\frac{dt}{\left(\int_a^t f(s)ds\right)^{\frac12}}<+\infty\quad\text{or}\quad\lim\limits_{t\to+\infty}g(t)>0.
\end{equation*}
Indeed, nonnegativity and monotonicity of $g$ imply that either $g\equiv0$ on $[a,+\infty)$ or $\lim\limits_{t\to+\infty}g(t)>0$. In the first case, $\Psi_k^{(\lambda)}(a)=\Psi_k^{(0)}(a)$ for every $\lambda>0$.
In the second case, choose $T>a$ with $g(T)>0$. For $t\geq T$ and $\lambda>0$,
$$
\int_a^t f(s)e^{\lambda\int_s^t g(\tau)d\tau}ds\geq e^{\lambda g(T)(t-T)}\int_a^T f(s)ds.
$$
Consequently,
$$
\int_T^{+\infty}\frac{dt}{\left(\int_a^t f(s)e^{\lambda\int_s^t g(\tau)d\tau}ds\right)^{\frac12}}\leq\left(\int_a^T f(s)ds\right)^{-\frac12}\int_T^{+\infty}e^{-\frac{\lambda g(T)}2(t-T)}dt<+\infty.
$$
Together with local integrability at $t=a$, this gives $\Psi_k^{(\lambda)}(a)<+\infty$ for every $\lambda>0$.
\end{remark}

\begin{proof}[Proof of Theorem \ref{Th1.4}]
We first prove (i). If $\widetilde\Psi_k(a)<+\infty$, condition \eqref{KO_2} fails. Theorem \ref{Th4.1-2}~(i) therefore implies that a maximal extension of the local solution from Theorem \ref{Th3.1-2} has finite radius $R$. The upper estimate \eqref{1prime_k_upper_2} and Remark \ref{Rem3.1} then give $\varphi_a(r)\to+\infty$ as $r\to R^-$. Conversely, integration of \eqref{1prime_k_upper_2} along any such solution yields $\widetilde\Psi_k(a)\leq C_1R<+\infty$. If $g(a)\geq0$, then $g=g^+$ along the solution, and \eqref{1prime_k_lower_2} similarly gives $C_2R\leq\widetilde\Psi_k(a)$. This proves (i), including the radius estimates.

We next prove (ii). Let $q=2$ and $\lim\limits_{t\to+\infty}g(t)>0$. Let $\varphi_a$ be a maximal extension of the local solution given by Theorem \ref{Th3.1-2}, and suppose that it is global. Then $\varphi_a(r)\to+\infty$ as $r\to+\infty$. Choose $T>a$ with $g(T)>0$ and $r_0>0$ such that $A:=\varphi_a(r_0)\geq T$. For $r>r_0$, set $m:=\frac{r_0+r}2$. On $[m,r]$, we have $f(\varphi_a)\geq f(T)>0$ and $g(\varphi_a)\geq g(T)>0$. Thus \eqref{1prime_k_lower_2} gives
$$
\varphi_a^\prime(r)\geq2^{-(n+k-2)}\left(\frac{f(T)}{g(T)}\right)^{\frac12}\left[e^{\frac{2kC_0g(T)}{n-1}(\varphi_a(r)-\varphi_a(m))}-1\right]^{\frac12}.
$$
Convexity implies $\varphi_a(r)-\varphi_a(m)\geq(\varphi_a(r)-A)/2$. Hence there exist $c,C>0$ and $r_1>r_0$ such that
$$
\varphi_a^\prime(r)\geq Ce^{c\varphi_a(r)},\quad r\geq r_1.
$$
Integration yields $r-r_1\leq C^{-1}\int_{\varphi_a(r_1)}^{+\infty}e^{-ct}dt$ for all $r\geq r_1$, a contradiction as $r\to+\infty$. The maximal radius is therefore finite. Since \eqref{1prime_k_upper_2} bounds $\varphi_a^\prime$ whenever $\varphi_a$ is bounded, Remark \ref{Rem3.1} gives boundary blow-up.

Assume additionally that $g(a)\geq0$ and $\Psi_k^{(0)}(a)<+\infty$, and let $u_a$ be any such solution in $B_R$. Monotonicity gives $g(t)=g^+(t)$ for $t\geq a$. Applying \eqref{1prime_k_upper_2} and \eqref{1prime_k_lower_2} as in the proof of Theorem \ref{Th1.3}~(ii), we obtain
$$
\frac{\widetilde\Psi_k^{(\lambda_1)}(a)}{\sqrt{\lambda_1}}\leq R\leq\frac{\Psi_k^{(0)}(a)}{\gamma_1},
$$
where $\lambda_1$ and $\gamma_1$ are defined in Theorem \ref{Th1.3}~(ii). This proves (ii).

Finally, we prove (iii). Let $0<q\leq2$ and $\lim\limits_{t\to+\infty}g(t)\leq0$. Monotonicity implies that $g(t)\leq0$ for all $t\in\mathbb R$.

\underline{\textit{Necessity.}} Let $u_a(x)=\varphi_a(|x|)$ be a positive radial admissible classical boundary blow-up solution in $B_R$. With $\lambda_1:=2c_3^{-\frac1k}C_0$, the upper estimate \eqref{1prime_k_upper_uni} gives
$$
\frac{\varphi_a^\prime(r)}{\left(\int_a^{\varphi_a(r)}f(t)e^{-\lambda_1\int_t^{\varphi_a(r)}P(\tau)d\tau}dt\right)^{\frac12}}\leq\sqrt{\lambda_1},\quad 0<r<R.
$$
Integrating from $0$ to $r$ and letting $r\to R^-$, we obtain
\begin{equation}\label{6.3}
\widetilde\Psi_k^{(-\lambda_1)}(a)\leq\sqrt{\lambda_1}R<+\infty.
\end{equation}
The parameter comparison in \cite[Lemma 2.4~(i)]{JJL2026} then yields \eqref{KO_7}.

\underline{\textit{Sufficiency.}} Assume that \eqref{KO_7} holds. Let $\varphi_a$ be the local solution given by Theorem \ref{Th3.1-2}, with maximal interval $[0,R_*)$, where $R_*\in(0,+\infty]$. The parameter comparisons imply that \eqref{KO_3} fails, and Theorem \ref{Th4.1-2}~(ii) gives $R_*<+\infty$. If $\varphi_a$ remained bounded, \eqref{1prime_k_upper_uni} would also bound $\varphi_a^\prime$, contradicting Remark \ref{Rem3.1}. Hence, $\lim\limits_{r\to R_*^-}\varphi_a(r)=+\infty$ and $u_a(x)=\varphi_a(|x|)$ is the required boundary blow-up solution in $B_{R_*}$.

To establish the radius estimates, let $u_a$ be any such solution in $B_R$. Set $\lambda_q:=\frac{qkC_0}{n-1}$ and $\gamma_q:=2^{-(n+k-2)}\left(\frac{qkC_0}{2(n-1)}\right)^{\frac12}$.
Since $\lambda_q\leq\lambda_1$ and $\widetilde\Psi_k^{(-\lambda)}(a)$ is non-decreasing in $\lambda$, the finiteness of $\widetilde\Psi_k^{(-\lambda_1)}(a)$ implies that of $\widetilde\Psi_k^{(-\lambda_q)}(a)$. The lower estimate \eqref{1prime_k_lower_uni} gives
$$
\frac{\varphi_a^\prime(r)}{\left(\int_a^{\varphi_a(r)}f(t)e^{-\lambda_q\int_t^{\varphi_a(r)}P(\tau)d\tau}dt\right)^{\frac12}}\geq\gamma_q,\quad 0<r<R.
$$
Integrating from $0$ to $r$ and letting $r\to R^-$, combined with \eqref{6.3}, we obtain
$$
\frac{\widetilde\Psi_k^{(-\lambda_1)}(a)}{\sqrt{\lambda_1}}\leq R\leq\frac{\widetilde\Psi_k^{(-\lambda_q)}(a)}{\gamma_q}.
$$
This completes the proof of (iii).

In all three cases, $q>0$ and $f(a)>0$. Theorem \ref{Th3.1-2}, Lemma \ref{Lem2.2}~(i), and the radial regularity statement in Remark \ref{Rem2.2} therefore give $u_a\in C^2(B_R)$.
\end{proof}

\begin{remark}\label{Rem6.criticalB}
Under the hypotheses of Theorem \ref{Th1.4}~(ii), we have $\lim\limits_{t\to+\infty}g^+(t)>0$. Applying Remark \ref{Rem6.criticalA} with $g^+$ in place of $g$ gives
$$
\widetilde\Psi_k^{(\lambda)}(a)<+\infty,\quad\text{for any }\lambda>0.
$$
Thus the weighted integrability condition holds automatically and need not be imposed separately.
\end{remark}

\subsection{Upper Bounds for Blow-Up Profiles}

We now prove the pointwise upper bounds stated in Theorems \ref{Th6.profileA} and \ref{Th6.profileB}.
We first establish the properties of the Keller--Osserman functions needed for the inverse functions used in the blow-up profile estimates.

\begin{lemma}\label{Lem6.KO_inverse}
Fix $a>0$. The following assertions hold.

\begin{enumerate}[label=(\roman*)]
\item Assume that $f$ and $g$ satisfy condition (A), let $q<2$, and suppose that either $f(t)>0$ for all $t>0$ or $q<1$. If condition (C1-A) holds and $\Psi_k(a)<+\infty$, then $\Psi_k:(0,+\infty)\to(0,+\infty)$ is continuous and strictly decreasing and satisfies
$
\lim\limits_{b\to+\infty}\Psi_k(b)=0.
$

\item Assume that $f$ and $g$ satisfy condition (B) and let $0<q<2$. If condition (C1-B) holds and
$\widetilde\Psi_k(a)<+\infty$, then $\widetilde\Psi_k:(0,+\infty)\to(0,+\infty)$ is continuous and strictly decreasing and satisfies
$
\lim\limits_{b\to+\infty}\widetilde\Psi_k(b)=0.
$

\item Suppose that $f$ is nonnegative on $[0,+\infty)$ and that condition (C2) holds. If $\Psi_k^{(0)}(a)<+\infty$,
then $\Psi_k^{(0)}:(0,+\infty)\to(0,+\infty)$ is continuous and strictly decreasing and satisfies
$
\lim\limits_{b\to+\infty}\Psi_k^{(0)}(b)=0.
$
Since $\widetilde\Psi_k^{(0)}(b)=\Psi_k^{(0)}(b)$ for all $b>0$, the same conclusions hold for $\widetilde\Psi_k^{(0)}$.

\item Assume that $f$ and $g$ satisfy condition (B), let $0<q\leq2$, and fix $\lambda>0$. If condition (C2) holds and $\widetilde\Psi_k^{(-\lambda)}(a)<+\infty$, then $\widetilde\Psi_k^{(-\lambda)}:(0,+\infty)\to(0,+\infty)$
is continuous and strictly decreasing and satisfies
$
\lim\limits_{b\to+\infty}\widetilde\Psi_k^{(-\lambda)}(b)=0.
$
\end{enumerate}
\end{lemma}
The proof is trivial and is therefore omitted.

\begin{proof}[Proof of Theorem \ref{Th6.profileA}]
We first prove (i). For $q<2$, the lower estimate \eqref{1prime_k_lower} gives
$$
\varphi_a^\prime(\rho)\geq C_2\left[\left(\int_a^{\varphi_a(\rho)}f(t)dt\right)^{\frac12}+\left(\int_a^{\varphi_a(\rho)}g(t)dt\right)^{\frac1{2-q}}\right],\quad r<\rho<R.
$$
Integrating from $r$ to $R$ and using $\varphi_a(\rho)\to+\infty$ as $\rho\to R^-$, we obtain
\begin{equation}\label{7.1}
\begin{aligned}
C_2d(x)&\leq\int_{\varphi_a(r)}^{+\infty}\frac{dt}{\left(\int_a^t f(s)ds\right)^{\frac12}+\left(\int_a^t g(s)ds\right)^{\frac1{2-q}}}\\
&\leq\int_{\varphi_a(r)}^{+\infty}\frac{dt}{\left(\int_{\varphi_a(r)}^t f(s)ds\right)^{\frac12}+\left(\int_{\varphi_a(r)}^t g(s)ds\right)^{\frac1{2-q}}}=\Psi_k(\varphi_a(r)).
\end{aligned}
\end{equation}
Theorem \ref{Th1.3} gives $\Psi_k(a)<+\infty$. Applying Lemma \ref{Lem6.KO_inverse}~(i) and inverting \eqref{7.1} yields
$$
u_a(x)=\varphi_a(r)\leq\Psi_k^{-1}\bigl(C_2d(x)\bigr),
$$
which proves (i).

For (ii), let $q=2$. Since $\varphi_a^\prime$ is non-decreasing on $(0,R)$, \eqref{6.1} yields
$$
\varphi_a^\prime(\rho)\geq\gamma_1\left(\int_a^{\varphi_a(\rho)}f(t)dt\right)^{\frac12},\quad r<\rho<R.
$$
Integrating from $r$ to $R$, we obtain
\begin{equation*}
\begin{aligned}
\gamma_1d(x)\leq\int_{\varphi_a(r)}^{+\infty}\frac{dt}{\left(\int_a^t f(s)ds\right)^{\frac12}}
\leq\int_{\varphi_a(r)}^{+\infty}\frac{dt}{\left(\int_{\varphi_a(r)}^t f(s)ds\right)^{\frac12}}=\Psi_k^{(0)}(\varphi_a(r)).
\end{aligned}
\end{equation*}
Lemma \ref{Lem6.KO_inverse}~(iii) therefore yields
$$
u_a(x)=\varphi_a(r)\leq\left(\Psi_k^{(0)}\right)^{-1}\bigl(\gamma_1d(x)\bigr),
$$
proving (ii).
\end{proof}

\begin{proof}[Proof of Theorem \ref{Th6.profileB}]
For (i) and (ii), set $t_0:=\inf\{t>0:g(t)\geq0\}$. Since $g$ is non-decreasing and $\lim\limits_{t\to+\infty}g(t)>0$, we have $t_0<+\infty$ and
$$
g(t)=g^+(t)\geq0,\quad t>t_0.
$$
If $\varphi_a(r)\leq t_0$, then (i)  and (ii) are immediately proved.
Suppose that $\varphi_a(r)>t_0$. Then
$$
g(\varphi_a(\rho))=g^+(\varphi_a(\rho))\geq0,\quad r\leq\rho<R.
$$

For (i), apply \eqref{1prime_k_lower_2} on $[r,R)$ to obtain
$$
\varphi_a^\prime(\rho)\geq C_2\left[\left(\int_{\varphi_a(r)}^{\varphi_a(\rho)}f(t)dt\right)^{\frac12}+\left(\int_{\varphi_a(r)}^{\varphi_a(\rho)}g^+(t)dt\right)^{\frac1{2-q}}\right],\quad r<\rho<R.
$$
Integrating from $r$ to $R$ gives
$$
C_2d(x)\leq\int_{\varphi_a(r)}^{+\infty}\frac{dt}{\left(\int_{\varphi_a(r)}^t f(s)ds\right)^{\frac12}+\left(\int_{\varphi_a(r)}^t g^+(s)ds\right)^{\frac1{2-q}}}=\widetilde\Psi_k(\varphi_a(r)).
$$
By Theorem \ref{Th1.4}~(i), $\widetilde\Psi_k(a)<+\infty$. Lemma \ref{Lem6.KO_inverse}~(ii) then yields
$$
\varphi_a(r)\leq\widetilde\Psi_k^{-1}\bigl(C_2d(x)\bigr).
$$
Together with the case $\varphi_a(r)\leq t_0$, this proves (i).

For (ii), the estimate \eqref{1prime_k_lower_2} for $q=2$ gives
$$
\varphi_a^\prime(\rho)\geq\gamma_1\left(\int_{\varphi_a(r)}^{\varphi_a(\rho)}f(t)dt\right)^{\frac12},\quad r<\rho<R.
$$
Hence,
$$
\gamma_1d(x)\leq\int_{\varphi_a(r)}^{+\infty}\frac{dt}{\left(\int_{\varphi_a(r)}^t f(s)ds\right)^{\frac12}}=\Psi_k^{(0)}(\varphi_a(r)).
$$
By Lemma \ref{Lem6.KO_inverse}~(iii),
$$
\varphi_a(r)\leq\left(\Psi_k^{(0)}\right)^{-1}\bigl(\gamma_1d(x)\bigr).
$$
Combining the two cases proves (ii).

Finally, let $0<q\leq2$ and $\lim\limits_{t\to+\infty}g(t)\leq0$. The lower estimate \eqref{1prime_k_lower_uni} gives
$$
\varphi_a^\prime(\rho)\geq\gamma_q\left(\int_a^{\varphi_a(\rho)}f(t)e^{-\lambda_q\int_t^{\varphi_a(\rho)}P(\tau)d\tau}dt\right)^{\frac12},\quad r<\rho<R.
$$
Integrating from $r$ to $R$, we obtain
\begin{align*}
\gamma_qd(x)\leq\int_{\varphi_a(r)}^{+\infty}\frac{dt}{\left(\int_{\varphi_a(r)}^t f(s)e^{-\lambda_q\int_s^tP(\tau)d\tau}ds\right)^{\frac12}}=\widetilde\Psi_k^{(-\lambda_q)}(\varphi_a(r)).
\end{align*}
By Theorem \ref{Th1.4}~(iii), $\widetilde\Psi_k^{(-\lambda_q)}(a)<+\infty$. Applying Lemma \ref{Lem6.KO_inverse}~(iv), we conclude that
$$
u_a(x)=\varphi_a(r)\leq\left(\widetilde\Psi_k^{(-\lambda_q)}\right)^{-1}\bigl(\gamma_qd(x)\bigr).
$$
This proves (iii).
\end{proof}

\section{Further Remarks}\label{Section7}

We conclude by discussing equivalent formulations of the Keller--Osserman conditions, their distinction from the determinant case, and the influence of the gradient term on global solvability and regularity at the origin. We also clarify the methodological significance of the implicit radial analysis.

\begin{remark}[Equivalent integral conditions]\label{Rem7.KO_equivalent}
Fix $a>0$. Since the following conditions concern convergence at infinity, their finite lower limits may be replaced by any fixed positive number.

\begin{enumerate}[label=(\roman*)]
\item Let $q<2$. If $h$ is nonnegative and non-decreasing, then, for $t\geq2a$, $$\frac t2h\left(\frac t2\right)\leq\int_a^t h(s)ds\leq th(t).$$
Consequently, under condition (A),
\begin{equation}\label{KO_pointwise_A}
\eqref{KO_1}\text{ holds}\iff\int_a^{+\infty}\frac{dt}{(tf(t))^{\frac12}+(tg(t))^{\frac1{2-q}}}=+\infty,
\end{equation}
and $\Psi_k(a)<+\infty$ if and only if the integral in \eqref{KO_pointwise_A} is finite. Under condition (B), if $0<q<2$ and $\lim\limits_{t\to+\infty}g(t)>0$, then
\begin{equation}\label{KO_pointwise_B_positive}
\eqref{KO_2}\text{ holds}\iff\int_a^{+\infty}\frac{dt}{(tf(t))^{\frac12}+(tg^+(t))^{\frac1{2-q}}}=+\infty,
\end{equation}
and $\widetilde\Psi_k(a)<+\infty$ if and only if the integral in \eqref{KO_pointwise_B_positive} is finite.

\item Let $q=2$ and suppose that $f(t)>0$ for all $t>0$. First,
\begin{equation*}
\Psi_k^{(0)}(a)<+\infty\iff\int_a^{+\infty}\frac{dt}{(tf(t))^{\frac12}}<+\infty.
\end{equation*}
Under condition (A), \cite[Lemma 2.4~(ii)]{JJL2026} and Remark \ref{Rem6.criticalA} give
\begin{align*}
\Psi_k^{(\lambda)}(a)<+\infty \text{ for some }\lambda>0&\iff\Psi_k^{(\lambda)}(a)<+\infty \text{ for any }\lambda>0\\
&\iff\Psi_k^{(0)}(a)<+\infty\text{ or }\lim\limits_{t\to+\infty}g(t)>0.
\end{align*}

\item Assume condition (B), let $0<q\leq2$, and suppose that $\lim\limits_{t\to+\infty}g(t)\leq0$. Fix $\lambda>0$.
Since $f$ is positive and non-decreasing and $P$ is nonnegative and non-increasing, for $t\geq2a$,
\begin{equation*}
\frac12e^{-\frac\lambda2}f\left(\frac t2\right)\min\left\{t,\frac1{P(\frac t2)}\right\}\leq J_{\lambda,a}^-(t)\leq\max\{1,\lambda^{-1}\}f(t)\min\left\{t,\frac1{P(t)}\right\},
\end{equation*}
where $J_{\lambda,a}^-$ is defined by \eqref{weighted_J_k} with $b=a$, and $1/P(t)=+\infty$ when $P(t)=0$. Since
$$
\frac1{\left[f(t)\min\{t,P(t)^{-1}\}\right]^{\frac12}}=\max\left\{\frac1{(tf(t))^{\frac12}},\left(\frac{g^-(t)}{f(t)}\right)^{\frac1q}\right\},
$$
a change of variables and comparison of the maximum with the sum yield
\begin{equation}\label{KO_negative_divergence_equivalent}
\eqref{KO_3}\text{ holds}\iff\int_a^{+\infty}\left[\frac1{(tf(t))^{\frac12}}+\left(\frac{g^-(t)}{f(t)}\right)^{\frac1q}\right]dt=+\infty.
\end{equation}
The same comparison shows that $\widetilde\Psi_k^{(-\lambda)}(a)<+\infty$ if and only if the integral in \eqref{KO_negative_divergence_equivalent} is finite. In particular, the convergence or divergence is independent of $\lambda>0$; see also \cite[Lemma 2.4~(i)]{JJL2026}.

If, in addition, $\lim\limits_{t\to+\infty}g(t)<0$, then $g^-(t)$ is bounded above and below by positive constants for all sufficiently large $t$, and hence
\begin{equation}\label{KO_strictly_negative_equivalent}
\eqref{KO_3}\text{ holds}\iff\int_a^{+\infty}\left[\frac1{(tf(t))^{\frac12}}+\frac1{(f(t))^{\frac1q}}\right]dt=+\infty.
\end{equation}
In particular, if $q=2$, condition \eqref{KO_strictly_negative_equivalent} reduces to
$$
\int_a^{+\infty}\frac{dt}{(f(t))^{\frac12}}=+\infty.
$$
\end{enumerate}
\end{remark}

\begin{remark}[An approach to implicit radial equations]\label{Rem7.5}
For $2\leq k<n$, the argument treats the implicit radial equation directly. Strict monotonicity with respect to $\varphi^{\prime\prime}$ permits inversion, and the monotonicity of $f$ and $g$ yields a cooperative first-order system. The method of sub- and supersolutions then establishes the local existence of immediately increasing solutions without requiring local Lipschitz continuity of the coefficients. The scaling arguments in Lemma \ref{Lem2.2}~(i)--(ii) determine the limits at the origin in the regular cases; analogous arguments yield the refined singular behavior under the additional hypotheses of Remark \ref{Rem2.2-iv}.

The derivative estimates establish a connection with the radial Laplacian. For an increasing convex solution, the positivity of $c_3$ and Young's inequality applied to the two factors in \eqref{sODE} yield
$$
c_3^{\frac1k}\varphi^{\prime\prime}(r)\leq C_0h(r)\leq\frac{n-1}{k}\left[\varphi^{\prime\prime}(r)+(n+k-2)\frac{\varphi^\prime(r)}r\right],
$$
as in \eqref{2prime_upper} and \eqref{2prime_lower_qneg}. Integration of these differential inequalities leads to the powers $\frac12$ and $\frac1{2-q}$ in the Keller--Osserman conditions for $q<2$. The equivalence principle in Theorem \ref{Th6.1} then transfers the radial existence conditions to entire subsolutions without a radial symmetry assumption. These arguments provide a framework for deriving necessary and sufficient growth conditions for the present nondivergence-form operators, and may guide further developments of Keller--Osserman theory for related fully nonlinear equations.
\end{remark}

\begin{remark}[Comparison with the determinant case]\label{Rem7.1}
For every $1\leq k<n$, the sharp Keller--Osserman conditions obtained in this paper have the same form. In particular, the subcritical exponents $\frac12$ and $\frac1{2-q}$ and the function $P$ in \eqref{P} are independent of $k$. This common structure relies on $c_3=n-k>0$. When $k=n$, this coefficient vanishes, and the determinant case treated in Part I \cite{JJL2026} has different integral exponents and the structural critical value $2-\frac1n$ instead of $2$.

The following logarithmic nonlinearities illustrate this distinction. Let $\beta\geq0$.
\begin{enumerate}[label=(\roman*)]
\item For every $1\leq k<n$, the gradient-independent equation
$$
S_k^{\frac1k}[\Delta uI-D^2u]=u(\ln(1+u))^\beta\quad\text{in }\mathbb R^n
$$
admits a positive nontrivial admissible classical entire subsolution if and only if $\beta\leq2$. When $k=n$, the condition is $\beta\leq2-\frac1n$.

\item Let $0<q\leq1$. For every $1\leq k<n$, the equation
$$
S_k^{\frac1k}[\Delta uI-D^2u]=u^{1-q}(\ln(1+u))^\beta(1+|Du|^q)\quad\text{in }\mathbb R^n
$$
admits a positive nontrivial admissible classical entire subsolution if and only if $\beta\leq2-q$. For the corresponding determinant equation, the condition is $\beta\leq2-\frac1n-q$.
\end{enumerate}
These conclusions follow from Theorem \ref{Th1.1} and its counterpart in \cite{JJL2026}.
\end{remark}

\begin{remark}[The influence of the gradient coefficient and exponent]\label{Rem7.2}
The role of $q$ depends on the asymptotic sign of $g$. If $g\equiv0$, the equation is gradient-independent, and the value of $q$ is irrelevant.

Under the hypotheses of the main theorems, if $\lim\limits_{t\to+\infty}g(t)>0$, the Keller--Osserman integrals are finite for $1<q<2$; see Remark \ref{Rem4.1}. At $q=2$, the corresponding positively weighted integrals are also finite; see Remarks \ref{Rem6.criticalA} and \ref{Rem6.criticalB}. Hence, no positive admissible entire subsolution exists for $1<q\leq2$, and the increasing convex radial solutions blow up at finite radii.

If $\lim\limits_{t\to+\infty}g(t)\leq0$, the negative gradient term may compensate for the growth of $f$. For instance, let $p,m\geq0$ and $0<q\leq2$, and consider
$$
S_k^{\frac1k}[\Delta uI-D^2u]=1+u^p-(1+u)^{-m}|Du|^q\quad\text{in }\mathbb R^n.
$$
By \eqref{KO_negative_divergence_equivalent}, this equation admits a positive admissible classical entire subsolution if and only if
$$
p\leq1\quad\text{or}\quad p+m\leq q.
$$
In particular, entire existence is possible even with a superlinear zero-order term when $1<p\leq q-m$.

Thus, $q\leq1$ is necessary for entire existence when $g$ is eventually positive. The value $q=2$ is instead the structural exponent at which the derivative estimates change form.
\end{remark}

\begin{remark}[Behavior at the origin]\label{Rem7.3}
The degenerate and singular gradient terms lead to different behavior at the origin.
Under condition (A), if $0<q<1$ and $f(a)=0$, the constant solution $u\equiv a$ coexists with an immediately increasing radial solution satisfying $\varphi_a^\prime(0)=\varphi_a^{\prime\prime}(0)=0$. This explains the nonconstancy requirement in Theorem \ref{Th1.1}, since constant solutions exist independently of the Keller--Osserman condition.
For $q<0$, Lemma \ref{Lem2.2}~(iv) and Remark \ref{Rem2.2} give the common Sobolev range $1\leq\delta<1-\frac1q$ for the constructed radial entire subsolutions. If $g(a)>0$ and the two-sided bound in Remark \ref{Rem2.2-iv}~(ii) holds, the refined asymptotics improve this range to $1\leq\delta<n(1-\frac1q)$ for the associated radial function. The $C^1$ regularity used in the comparison argument is included in the definition of a strong subsolution.
\end{remark}

%


\noindent{\bf Acknowledgements.}

J. Ji was supported by the SEU Innovation Capability Enhancement Plan for Doctoral Students (Grant No. CXJH\_SEU 24123).
F. Jiang was supported by the National Natural Science Foundation of China (Grant No. 12271093), and by the Shanghai Institute for Mathematics and Interdisciplinary Sciences (Grant No. SIMIS-ID-2025-AD).
M. Li was supported by the Alexander von Humboldt Research Fellowship.
F. Jiang and M. Li were supported by the Jiangsu Provincial Scientific Research Center of Applied Mathematics (Grant No. BK20233002).

\bibliographystyle{amsplain}

\end{document}